\documentclass[a4paper,11 pt]{amsart}

\usepackage{geometry} %Margins
\usepackage{amssymb}  %,amscd,amsthm,mathrsfs}
\usepackage{mathtools}      % improves amsmath
\usepackage{mathabx}        % makes many symbols (as $\leq$) more beautiful; must be loaded after mathtools 
\usepackage[bb=fourier,cal=euler,scr=rsfs]{mathalfa}	% selecting fancy math fonts
\usepackage{enumitem}       % special itemize with custom tags and labels that work 
\usepackage[colorlinks=true, linkcolor=red, citecolor=blue]{hyperref} % Links in the pdf. Backref option: backref=page (incluir eso en los parentesis rectos) each bibliographical entry denotes where it was cited

\usepackage[ansinew]{inputenc}
\hypersetup{bookmarksdepth = 3} % Depth of sections/subs... to have bookmark links in the pdf;
\numberwithin{equation}{section}     % Makes labeled equations easier to find.

\setlist[enumerate,1]{label={\upshape(\roman*)},ref=\roman*}
\setlist[enumerate,2]{label={\upshape(\alph*)},ref=\alph*}

\newtheorem{theorem}{Theorem}[section]
\newtheorem{teo}[theorem]{Theorem}
\newtheorem*{teoA}{Theorem A}
 
\newtheorem{coro}[theorem]{Corollary}
\newtheorem*{coroB}{Corollary B}
\newtheorem{lemma}[theorem]{Lemma}
\newtheorem{lema}[theorem]{Lemma}

\newtheorem{addendum}[theorem]{Addendum}
\newtheorem{prop}[theorem]{Proposition}

\newtheorem*{question}{Question}

\theoremstyle{definition}
\newtheorem{definition}[theorem]{Definition}
\newtheorem{remark}[theorem]{Remark}

\newcommand{\eps}{\varepsilon}

\newcommand{\wt}[1]{\widetilde{#1}}
\newcommand{\R}{\mathbb R}
\newcommand{\RR}{\R}
\newcommand{\ZZ}{\mathbb Z}
\newcommand{\HH}{\mathbb{H}}

\newcommand{\cI}{\mathcal{I}}

\newcommand{\cW}{\mathcal{W}}
\newcommand{\cF}{\mathcal{F}}

\newcommand{\cG}{\mathcal{G}}

\newcommand{\cB}{\mathcal{B}}
\newcommand{\cJ}{\mathcal{J}}
\newcommand{\wcF}{\widetilde \cF}

\newcommand{\wcG}{\widetilde \cG}

\newcommand{\cL}{\mathcal{L}}

\newcommand{\cC}{\mathcal{C}}
\newcommand{\cO}{\mathcal{O}}

\newcommand{\cS}{\mathcal{S}}

\newcommand{\cV}{\mathcal{V}}
\newcommand{\cU}{\mathcal{U}}

\newcommand{\mt}{\widetilde M}

\newcommand{\wihat}{\widehat}
\newcommand{\wihatD}{\widehat{D}}
\newcommand{\wihatM}{\widehat{M}}

\title[Transverse minimal foliations on unit tangent bundles]{Transverse minimal foliations on unit tangent bundles and applications}

\author{Sergio R.\ Fenley} 
\address{Florida State University, Tallahassee, FL 32306}
\email{fenley@math.fsu.edu}

\author{Rafael Potrie} 
\address{Centro de Matem\'atica, Universidad de la Rep\'ublica, Uruguay \&  IRL-IFUMI, Laboratorio del Plata, CNRS}
\email{rpotrie@cmat.edu.uy}
\urladdr{http://www.cmat.edu.uy/~rpotrie/}

\thanks{S.F. was partially supported by Simons Foundation grant 280429 and 637554., National Science Foundation grant 
DMS-2054909, and by the Institute for Advanced Study.
 R. P. was partially supported by CSIC I+D project `Estructuras Topol\'ogicas de sistemas parcialmente hiperb\'olicos y aplicaciones'.} %and by MathAmSud-Latitude.

\begin{document}

\begin{abstract}
We show that if $\cF_1$ and $\cF_2$ are two transverse minimal foliations on $M = T^1S$ then either they intersect in an Anosov foliation or there exists a Reeb surface in the intersection foliation. The existence of a Reeb surface is incompatible with partially hyperbolic foliations so we deduce from this that certain partially hyperbolic diffeomorphisms in unit tangent bundles are collapsed Anosov flows. We also conclude that every volume preserving partially hyperbolic diffeomorphism of a unit tangent bundle is ergodic. 
\bigskip

\noindent {\bf Keywords: } Foliations, 3-manifold topology, Anosov flows, partial hyperbolicity, ergodicity. 
%3-manifold topology, foliations.
%
\medskip

\noindent {\bf MSC 2010:} Primary:  57R30, 37C86, 37D30; 
Secondary: 53C12, 57K30, 37C15.
\end{abstract}

\maketitle

\medskip
\medskip

%{\em Acknowledgements:} 
%SF was partially supported by 
%grant award number $280429$ of the Simons foundation.
%RP was partially supported by CSIC no.618, FVF-2017-111 and
%FCE-1-2017-1-135352.

\section{Introduction} 
This article studies geometric and dynamical properties of one dimensional subfoliations obtained as the intersection of two transverse minimal foliations on unit tangent bundles of higher genus surfaces. We prove some strong geometric properties that imply that under certain conditions, the foliation must be homeomorphic to the orbit foliation of the geodesic flow for a hyperbolic metric on the surface. 

One big motivating example for us comes from partially hyperbolic dynamics in dimension $3$: under very general orientability conditions there is a pair of 
two dimensional branching foliations, which are approximated by regular foliations. The pair of foliations are transverse to each other, yielding a one dimensional subfoliation of both. Suppose that one proves that the subfoliation is the flow foliation of a topological Anosov flow. Then the partially hyperbolic diffeomorphism is what is called a collapsed Anosov flow \cite{BFP,FP2}. In particular this has some important consequences, such as accessibility and ergodicity in the volume preserving case \cite{FP,FP3}.

This naturally leads to the following very general question
which was the
initial goal of this project:

\begin{question}
Let $\cF_1$ and $\cF_2$ be two transverse minimal foliations in a closed 3-manifold. Are there simple conditions that guarantee that the intersection foliation is homeomorphic to an Anosov foliation? 
\end{question}

Here, by \emph{Anosov foliation} we mean the orbit foliation of a (topological) Anosov flow\footnote{A short definition of a topological Anosov flow is an expansive flow preserving a foliation, see \cite{BM} for a nice introduction. In this paper we work in unit tangent bundles, where every topological Anosov flow is orbit equivalent to the geodesic flow of some constant curvature metric (which is a smooth Anosov flow) so we will not differentiate between them.}. It is worth pointing out that the scope of the question is quite general, and that the conditions are somewhat more restrictive than studying vector fields tangent to a foliation (for instance, the horocycle flow is a subfoliation of the weak stable foliation of an Anosov flow but it is rarely obtained as the intersection of two foliations). Having a vector field tangent to a foliation already imposes some obstructions on the foliation and the manifold (see for instance \cite{BH,HPS}), so we expect here to have even more.  This very general question can be extended to the problem of understanding in general the one-dimensional foliations induced by intersecting two general transverse foliations in a closed 3-manifold; we have included minimality to simplify certain formulations\footnote{One can always blow up one of the foliations and the intersection will no longer be an Anosov foliation. Other phenomena can also arise, see for instance the examples constructed in \cite{BBP}.}  but it certainly makes sense to ask this question in general (in fact, a similar question for three transverse taut foliations is suggested in \cite[\S 7.1]{Thurston}). 

We stress that the question of analyzing general transverse intersections of foliations in $3$-manifolds is very natural, interesting in itself, and has appeared in  other contexts, such as 
Anosov and pseudo-Anosov flows transverse to foliations \cite{Fenley1.5,Fenley2,Fenley4,Thurston}.

In this article we start the general study of geometric properties of one dimensional subfoliations of a pair of transverse foliations in $3$-manifolds. In this generality the problem is at this point a
fairly complex one (see \S~\ref{ss.newresults} for recent progress). Here we restrict to a class of $3-$manifolds: unit tangent bundles of surfaces of negative Euler characteristic. The main reason is that in this case we have a strong rigidity for single foliationsand this substantially helps to study this problem. It is also relevant for us since many people working in partially hyperbolic dynamics are more familiar with this family of 3-manifolds. On the other hand many of the techniques introduced in this paper should be useful for the general problem. In fact, since this paper was released, much progress has been made in the problem that we survey in \S~\ref{ss.newresults} and show the impact it has had. 

%In unit tangent bundles of closed surfaces  of negative Euler characteristic, as we shall explain later,
%all minimal foliations are essentially the weak stable foliation of an Anosov flow (by a result of \cite{Matsumoto} which strongly builds on \cite{Ghys2}).

In spite of the fact that in unit tangent bundles minimal foliations are homeomorphic to the weak stable foliation of an Anosov flow  (by a result of \cite{Matsumoto}, see also \cite{Ghys2}) , T. Barbot pointed out to us the paper \cite{MatsumotoTsuboi} which gives a beautiful example showing that even in these manifolds,
there may be obstructions to the intersection being an Anosov foliation (in fact in \cite{MatsumotoTsuboi} they produce a triple of pairwise transverse minimal foliations). 

In this paper we show that, in unit tangent bundles,  the conditions that obstruct the intersection to be homeomorphic to an Anosov foliation, are similar to those appearing in the example of \cite{MatsumotoTsuboi}. In \S~\ref{ss.MatsumotoTsuboi} we describe the example of \cite{MatsumotoTsuboi} from the point of view of this paper as well as discuss some possible extensions, and in \S~\ref{s.consReeb} we prove that the obstruction to intersect in a foliation homeomorphic to an Anosov foliation can be explained by behavior identical to the ones discussed in \S~\ref{ss.MatsumotoTsuboi}. 

The two dimensional foliations $\cF_1, \cF_2$ we deal with have Gromov hyperbolic leaves. The strategy employed here to prove the Anosov behavior of the intersection foliation $\cG = \cF_1 \cap \cF_2$ in certain cases is the following: show geometric properties of the subfoliations inside the leaves of the two dimensional foliations. More specifically we try to show that the subfoliations are by uniform quasigeodesics inside these two dimensional leaves. We obtain a structure that is enough to identify the obstruction for this to happen: \emph{Reeb surfaces}. These are surfaces in some leaf of $\cF_1$ or $\cF_2$ which are finitely covered by a two dimensional annulus whose boundary circles are leaves of the foliation $\cG$ and the leaves of $\cG$ in the interior of the annulus spiral towards the boundary components in opposite directions (see figure~\ref{fig-reebsurface}).

\begin{figure}[ht]
\begin{center}
\includegraphics[scale=0.78]{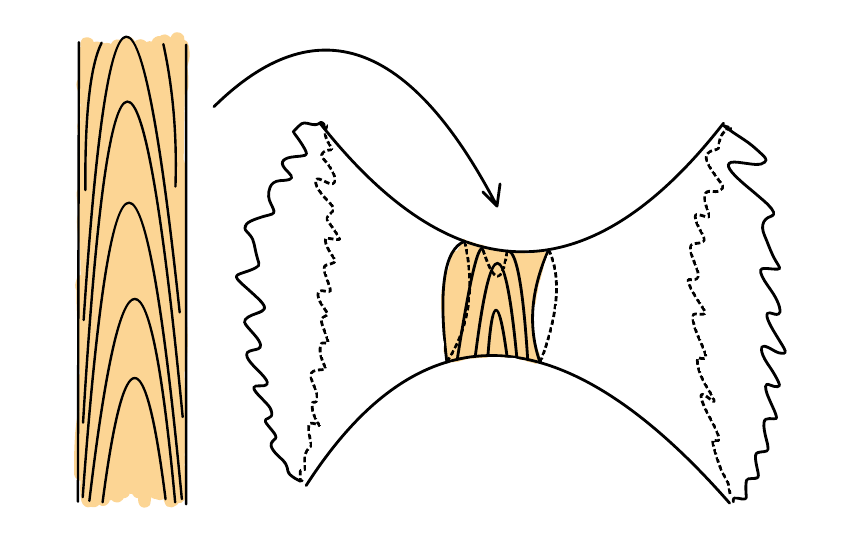}
%\begin{picture}(0,0)
%\put(-203,153){$o_1$}
%\put(-130,153){$p$}
%\put(-53,144){$o_2$}
%\put(-160,52){$A_1^u$}
%\put(-129,52){$A_2^s$}
%\put(-87,77){{\small $\cG^u(p)$}}
%\put(-196,91){{\small $\cG^s(p)$}}
%\put(-113,225){$A_1^s$}
%\put(-82,220){$A_2^u$}
%\end{picture}
\end{center}
\vspace{-0.5cm}
\caption{{\small A Reeb surface and its lift to the universal cover.}}\label{fig-reebsurface}
\end{figure}

Our main result is the following:

\begin{teoA}\label{teoA}
Let $S$ be a closed orientable hyperbolic surface and let $M = T^1 S$. Let $\cF_1, \cF_2$ be minimal, two dimensional foliations in $M$ which are transverse to each other. Let $\cG$ be the intersection of $\cF_1$ with $\cF_2$. Then either $\cG$ is homeomorphic to the orbit foliation of the geodesic flow of a hyperbolic metric on $S$ or $\cG$ contains a Reeb surface. 
\end{teoA}

We remark that minimality is necessary as one can easily obtain counterexamples by blowing up weak stable and weak unstable foliations of Anosov flows.  We note that we get a description of $\cG$ when there are Reeb surfaces: we refer the reader to \S~\ref{s.consReeb} for precise formulations (see in particular, Corollary \ref{coro.consReeb}). The description says that if there are Reeb surfaces, then some structure
very similar to the Matsumoto-Tsuboi example has to occur. This is a rigidity
result in the sense that counterexamples to Anosov behavior for $\cG$ have very few possibilities. 

An important point is that in Theorem A there is no (a priori) dynamical hypothesis, but there is a strong dynamical consequence. On the two dimensional level, there is an a  priori result which is the starting point of our study here and which explains the choice to work with unit tangent bundles rather than general Seifert fibered manifolds: as mentioned earlier, Matsumoto \cite{Matsumoto} proved that each $\cF_i$ individually is topologically equivalent to the weak stable foliation\footnote{In contrast, for general Seifert manifolds, it was already known that minimal foliations come from representations of surface groups into $\mathrm{Homeo}(S^1)$ by \cite{Thu,Brit} (see also \cite{Calegari-book}), but what we use here is the rigidity result of Matsumoto that gives conjugacy to a Fuchsian representation. This is also based on previous work of Ghys \cite{Ghys2}.} of the geodesic flow in $M$. While many arguments in our proof hold in more generality, we have chosen to avoid using too much terminology and background on foliations so that someone willing to believe Theorem \ref{teo.matsumoto} could in principle follow the proof of the main theorem. It should also be helpful to convey the main strategy and ideas that we are trying to push. See \S~\ref{ss.newresults} for updates on recent results. 

As a consequence we prove the following result for partially hyperbolic diffeomorphisms (we refer the reader to \S~\ref{s.partiallyhyperbolic} for precise definitions and more general results).

\begin{coroB} Let $f: M \to M$ be a volume preserving partially hyperbolic
diffeomorphism in $M = T^1 S$ with $S$ a closed orientable
surface of genus $g \geq 2$.  Then, $f$ is a collapsed Anosov flow. 
Hence $f$ is accessible and if $f$ is $C^2$, then $f$ is ergodic.
\end{coroB}

\subsection{Outline of the strategy and organization of the paper}\label{ss.outline}

A way to detect that a one dimensional foliation is an Anosov foliation is to show that when seen as the orbits of a flow, this flow is expansive. This is well known to imply that the flow is pseudo-Anosov (\cite{IM,Paternain}), and since our one-dimensional foliation preserves a two dimensional foliation, being pseudo-Anosov is enough to show it is an Anosov foliation (see \cite[\S 5]{BFP}).

In this article we will instead consider a more geometric point of view.
As in \cite{FP2,BFP}, the main strategy to show that the flow generated by the foliation is a topological Anosov flow, is to show that its flow foliation  subfoliates the two dimensional leaves by quasigeodesics. We say that the foliation is \emph{leafwise quasigeodesic} when in each leaf of say $\wt{\cF_1}$ the leaves of the intersected foliation $\wt{\cG}$ are quasigeodesics (note that we do not ask the one dimensional leaves to be quasigeodesics of $\mt$). 

The key idea is to try extend arguments that work in a compact surface (see \cite{Levitt} or \cite[Appendix A]{HPS}). Note however that in closed surfaces, compactness gives many deck transformations that preserve the universal cover, while here, leaves of the foliation in the universal cover may typically have small stabilizer. 
Here, the key point is to use minimality and compactness of $M$ to bring leaves  together. Transversality of the foliations allows to push some behavior to nearby leaves, and then, arguments like in \cite{FPmin} can help partially reproducing the arguments in the case of closed surfaces. However, there are some subtleties when trying to push behavior to nearby leaves associated with the transverse geometry of leaves in the universal cover. The fact that our two dimensional foliations are \emph{Anosov} (in the sense that they are homeomorphic to the weak stable foliation of an Anosov flow when 
$M = T^1 S$) is extremely useful. This allows for a very strong and precise form of `pushing' at most points in the leaf approaching the boundary at infinity (see Proposition \ref{prop-pushing}). By this, we mean that if two leaves are closeby in the leafspace, then, we can sort of copy the intersected foliation $\widetilde{\cG}$ in one leaf to nearby leaves by following the intersection with leaves of the other foliation. Note that pushing requires having two transverse foliations and this is crucial in our arguments. In fact, for flows tangent to a foliation, more diverse behavior is possible as it is shown in \cite{FP2} for the strong stable foliation of the examples constructed in \cite{BGHP}.

The general organization for showing that the foliation $\wt{\cG}$ is leafwise quasigeodesic follows the rough outline that was used in \cite{FP2} for foliations that come from some special dynamical systems. Not all steps hold in full generality as examples show, but we still describe here the main steps and explain under which assumptions they work.

\begin{itemize}
\item {\bf Landing:} In each leaf $L \in \wt{\cF_1}$ we look at the restriction of $\wt{\cG}$ to the leaf $L$. Then every ray of any given leaf of $\wt{\cG}$ (which is always properly embedded in $L$) has a well defined unique limit point in the compactification $L \cup S^1(L)$ by the Gromov boundary of $L$. In Theorem \ref{prop.landing} we show that in our setting, this holds in full generality. 

\item {\bf Small visual measure:} There are many ways a properly embedded ray in a hyperbolic plane can be extended to the boundary, in particular, we wish to rule out the possibility that it lands like horocycles do. For this, a technical property that we call small visual measure is relevant. An important consequence of small visual measure is that geodesic rays starting at a point of a ray $r$ of a leaf $c \in \wt{\cG}$ and landing at the same point as the ray $r$ must be contained in a uniform neighborhood of $r$. This is also something that holds always in our setting as we prove in \S\ref{s.smallvisual}. We point out here that the fact that $\cG$ is obtained as the intersection of two transverse foliations is crucial, as the horocyclic flow of an Anosov flow subfoliates a minimal foliation of $T^1S$ but its leaves do not satisfy the small visual measure property. %(Wilder behavior occurs for the strong stable foliation of the examples in \cite{BGHP}.)

\item {\bf No bubble leaves:} To get a quasigeodesic foliation we need to rule out the existence of leaves $c$ of $\wt{\cG}$ so that both rays of $c$ land in the same point of $S^1(L)$ (where $L$ is the two dimensional leaf containing $c$).
We call such a leaf $c$ a \emph{bubble leaf}. We note that the example in \cite{MatsumotoTsuboi} contains bubble leaves, so we cannot expect to prove the non-existence of bubble leaves in general. But the existence of some non-bubble leaves is important in our analysis of the small visual measure and the construction of Reeb surfaces. 

\item {\bf Hausdorff leaf space:} Another consequence of a one dimensional foliation subfoliating a two dimensional foliation by Gromov hyperbolic leaves and being leafwise quasigeodesic is that the leaf space of the foliation $\wt{\cG}$ in each leaf $L \in \wt{\cF_i}$ must be Hausdorff. A priori, the non existence of bubble leaves is not enough to rule out non-Hausdorffness, so more analysis is needed. This is the content of \S \ref{s.dichotomy} where we show that non-Hausdorfness of the leaf space leads to Reeb surfaces. 

\item {\bf Getting the quasigeodesic property:} Having Hausdorff leaf space is not enough to deduce the leafwise quasigeodesic property as the horocycle flow shows. In our context, we can show that it is enough and we do so in \S \ref{s.Hsdff}. 

\end{itemize}

The paper is organized as follows: In \S \ref{s.minimalfol} we study general properties of minimal foliations of $T^1S$, in particular getting the consequences of \cite{Matsumoto} that we will use. In \S \ref{s.transverse} we show some properties of pairs of transverse foliations, in particular defining and describing the landing property and showing that in some settings it is possible to \emph{push} behavior to nearby leaves. In \S \ref{s.Landing}, \S \ref{s.smallvisual} we address landing and the small visual measure property. When the leaf space of $\wcG$ is leafwise Hausdorff we show in \S \ref{s.Hsdff} that the foliation $\cG$ must be an Anosov foliation. In \S \ref{ss.MatsumotoTsuboi} we revisit the examples from \cite{MatsumotoTsuboi} and describe their properties from the point of view of this paper as well as some possible extensions. Finally, in \S\ref{s.dichotomy} we complete the proof of Theorem A. In \S\ref{s.consReeb} we explore further properties that in some sense show that the examples of \cite{MatsumotoTsuboi} are in some sense the only possible way to introduce Reeb surfaces. In \S\ref{s.partiallyhyperbolic} we study the applications of our result to classification and ergodicity of partially hyperbolic dynamics (we note that the applications to partial hyperbolicity do not use \S\ref{ss.MatsumotoTsuboi} and \S\ref{s.consReeb}, which can be skipped by the reader interested only in the applications to partial hyperbolicity). 

\subsection{Recent developments}\label{ss.newresults}
Since this paper was released in early 2023, several new developments have been obtained that we explain here, trying to emphasize the influence of this particular paper. Let us comment on the papers \cite{BaFP,FP4,FP6}. 

In \cite{BaFP} we extended part of the results of this paper to a more general setting, in particular, we showed that if $\cF_1, \cF_2$ are transverse $\RR$-covered Anosov foliations in a 3-manifold which are uniformly equivalent, then, they either intersect in the orbit foliation of an Anosov flow, or they contain a Reeb surface. The goal of \cite{BaFP} was to present a completely different approach and also present results in higher dimensions (and some results in dimension 3 that are important in \cite{FP6}), and while it reproves Theorem A of this paper, the proof is completely different and the techniques do not allow to obtain the more precise results that we obtain here in \S~\ref{s.consReeb}, and which we consider to be one of the main contributions of this paper (and also, that up to now do not have any counterpart in any context). 

The paper \cite{FP4} that was released later than this paper also has some intersection, though this and that paper are somewhat independent. In \cite{FP4} we reversed the strategy presented in \S~\ref{ss.outline} by assuming from the start a very strong property on the intersected foliation: that the leafspace of the intersected foliation is Hausdorff. Under that assumption, we were able to work out the full outline of \S~\ref{ss.outline} by showing landing, small visual measure (assuming that the manifold has fundamental group which is not virtually solvable) and finally the quasigeodesic property of leaves. While \cite{FP4} works in a much wider generality than this paper (no assumptions on the topology of $M$), it would not result in a major reduction of the current paper, in particular, it would only serve in reducing \S~\ref{s.Hsdff}, to  which one could apply directly the results in \cite{FP4}, but it would seem unnatural, as here we have a more direct proof. 

Finally, let us mention that recently, we have announced that we can prove a general statement needed for the classification of partially hyperbolic dynamics (see \cite{FP6}). We note that the general statement in \cite{FP6} owes to the ideas that had been developed here, but \cite{FP6} does not use or depend on this paper. The main result in \cite{FP6} gives a general result about transverse foliations with Gromov hyperbolic leaves and shows that the only obstruction to the intersected foliation being leafwise quasigeodesic is to have what we called \emph{generalized Reeb surfaces}, which extend the concept of Reeb surfaces used here. We note that \cite{FP6} gives hope in progressing in the understanding of general transverse foliations, and in our opinion makes \S~\ref{s.consReeb} of this paper even more relevant, since no analogue of this has been shown, and the existence of Reeb surfaces in some setting could be combined with the techniques in \cite{FP6} to see if one can produce some incompressible torus in $M$. This could for instance be relevant to address the following question which we believe to have positive chance to be true (see also the question at the end of \S~\ref{s.consReeb}):

\begin{question}
Let $\cF_1$ and $\cF_2$ be two transverse minimal foliations in a closed hyperbolic 3-manifold. Then, they intersect in the orbit foliation of a (topological) Anosov flow. 
\end{question}

\medskip

{\small \emph{Acknowledgements:} We would like to thank Thierry Barbot for pointing out \cite{MatsumotoTsuboi} which was crucial for this work (and particularly for \S \ref{s.consReeb}). We would also like to thank several referees that have given important feedback to the paper. }

\section{Minimal foliations on unit tangent bundles}\label{s.minimalfol}
We consider $M=T^1S$ where $S$ is a closed, orientable genus $g \geq 2$ surface. Consider $\cF$ a minimal foliation on $T^1S$. Such foliations have been completely classified by Matsumoto \cite{Matsumoto} showing that they are homeomorphic to the weak stable foliation of the geodesic flow on $S$ for a hyperbolic metric. We will expand on this as well as on previous results in \cite{Thu,Brit} to describe the foliations in a way that is useful for our purposes. 

We will assume throughout the article that the foliations we consider are $C^{0,1+}$. This means that leaves
are $C^1$ surfaces, see for instance \cite{CanCon}. This assumption is mostly for convenience, as having smooth leaves simplifies the definition of distances and lengths inside leaves. Note that in  \cite{Calegari-smoothing} it is shown that for each foliation there is a smooth structure in $M$ which makes it $C^{0,\infty+}$, but it is unclear that this can be done simultaneously for both foliations. To avoid discretizations of distances and local problems, we will keep this assumption throughout. 

%\begin{remark} 
%This assumption is for convenience: first of all, given a foliation
%$\cF$ there is $\cF_1$ topologically equivalent to $\cF$ and
%satisfying the assumption, by Calegari's result in \cite{Calegari-smoothing}.
%The issue here, is that one cannot do this simultaneously if there
%are two foliations $\cF_1, \cF_2$.
%Also one can also discretize the problem, and define coarse
%length of curves, and so on, and in that way define
%Gromov hyperbolicity even if the leaves are not $C^1$. 
%But the goal of this article is not utmost generality,
%hence the assumption above.
%\end{remark}

\subsection{Unit tangent bundles}\label{ss.unittangentbundles} 
Consider the universal covering map $\pi: \tilde S \to S$ in which one can identify $\tilde S \cong \HH^2$ as follows: fix a discrete subgroup 
$\Gamma$ of $\mathrm{Isom}_+(\HH^2) \cong \mathrm{PSL}(2,\RR)$ (acting on the right) where $\Gamma \cong \pi_1(S)$.  We can identify $\pi$ with the quotient map from $\HH^2 \to \HH^2/_{\Gamma}$. 

Since $\mathrm{PSL}(2,\RR)$ naturally identifies with $T^1 \HH^2$ we obtain that $T^1S$ is identified by this action with $T^1 \HH^2/_{\Gamma}$. We parametrize $T^1 \HH^2$ by $\HH^2 \times \partial \HH^2$ by identifing a unit vector $v \in T_x \HH^2$ with the pair $(x, v_+) \in \{x\} \times \partial \HH^2$ where $v_+$ is the limit point of the geodesic in $\HH^2$ starting at $x$ with speed $v$. The action of $\mathrm{PSL}(2,\RR)$ on $\partial \HH^2 \cong S^1$ is given by extending the action of isometries on geodesic rays (if one uses the upper half model of $\HH^2$ this action corresponds to the standard action by rational transformations on $\RR \cup \{\infty\}$). 

We will denote by $M=T^1S$, and by $\wihat M= T^1\tilde S$ its intermediate cover with deck transformations identified with the action of $\Gamma$ in the coordinates given by the identification of $T^1 \tilde S \cong \HH^2 \times \partial \HH^2$. We will denote by $\wt M$ the universal cover of $M$ which covers $\wihat M$ with deck transformation group associated with the center of $\pi_1(M)$ (which corresponds to the deck transformation associated to the circle fiber of the circle bundle over $S$).  

\subsection{Horizontal foliations}\label{ss.horizontal}

Consider the foliation $\wihat\cF_{ws}$ of  $T^1 \tilde S \cong \HH^2 \times \partial \HH^2$  given by $\wihat\cF_{ws} = \{ \HH^2 \times \{\xi\} \}_{\xi \in \partial \HH^2}$.  This foliation is $\Gamma$-invariant and since the $\Gamma$ action on $\partial \HH^2$ is minimal it descends to a minimal foliation $\cF_{ws}$ of $M$ which is exactly the weak stable foliation for the geodesic flow associated to the metric on $S$ induced by the choice of $\Gamma \subset \mathrm{Isom}_+(\HH^2)$.  

The following result from \cite{Matsumoto} will be very important in our study and says that $\cF_{ws}$ is the unique minimal foliation of $M$ up to homeomorphisms isotopic to the identity on the base: 

\begin{theorem}[Matsumoto \cite{Matsumoto}]\label{teo.matsumoto} 
Let $\cF$ be a minimal foliation of $M$, then, there exists a homeomorphism $h: M \to M$ inducing the identity on the base such that $h(\cF) = \cF_{ws}$. 
\end{theorem}

We note that the result of \cite{Matsumoto} is stated for $C^2$ foliations without compact leaves,
but all that is needed is that $\cF$ is \emph{horizontal}, see also
\cite{Ghys2,Thu}. 

To get the horizontal property for $\cF$ we
use the fact that $\cF$ is minimal and apply Brittenham's theorem  \cite{Brit}. There is a slightly technical issue in Brittenham's result: in
\cite{Brit} a lamination is one that is carried by a branched surface,
so technically a foliation is not a lamination and must first be
split along a finite set of leaves to produce an essential lamination $\cL$ which, due to \cite{Brit}, has a minimal sublamination
which is either horizontal or vertical, which since $\cF$ is minimal this is $\cL$ itself. But $\cF$ cannot be vertical\footnote{If it were, it would induce a non singular foliation in the base surface which has non-zero Euler characteristic.}. It follows that $\cF$ is horizontal
and then one can apply Matsumoto's result. See also \cite{Calegari-book} for generalities on foliations on circle bundles. 

That $h$ induces the identity on the base means that if $p: T^1S \to S$ is the projection, then the induced actions on the fundamental group verify that $p_\ast \circ h_\ast = p_\ast$. Note however, that $h$ may not be homotopic to identity as a map of $M$ and therefore two minimal foliations $\cF_1$ and $\cF_2$ may not be \emph{uniformly equivalent} in the sense that leaves in the universal cover may not be bounded distance away from a leaf of the other foliation (see \cite{Thurston} for discussion on this notion, we note that this notion is different from the notion of having homotopic plane fields which is also used sometimes). 

We will use some other properties of minimal foliations on $M$.  Some of these hold more generally for \emph{Reebless foliations} (due to Novikov's theorem, see \cite{Calegari-book}). We state the properties we need in the setting we will use where the proofs are a direct consequence of the corresponding properties for $\cF_{ws}$ and Theorem \ref{teo.matsumoto} (the last point also uses \cite{Brit} for smoothness):  

\begin{coro}\label{coro.foliations}
Let $\cF$ be a minimal foliation on $M = T^1S$ and denote by $\wihat\cF$ and $\wt \cF$ the lifts to the covers $\wihatM$ and $\mt$ respectively. Then: 
\begin{enumerate}
\item if $\tau$ is a curve transverse to $\wt \cF$ then it intersects each leaf at most once, 
\item the leaf space $\wihat\cL = \wihatM/_{\wihat\cF}$ of $\wihat\cF$ is homeomorphic to $S^1$ and the leaf space $\wt \cL$ of $\wt \cF$ is homeomorphic to $\RR$,
\item\label{it.britenham} there is a (smooth) isotopy of $\cF$ which makes every leaf transverse to the circle fibers of $T^1S$. 
\end{enumerate}
\end{coro}

Note that a transversal to $\wt \cF$ or $\cF$ is a continuous curve $\tau: I \to M$ where $I$ is some interval verifying that for every $t \in I$, there is $\eps>0$ so that the curve $\tau|_{(t-\eps,t+\eps)}$ is monotone in the leaf space of a foliation chart of $\wt \cF$ or $\cF$ (which is an interval) around $\tau(t)$.  We will many times abuse notation and denote by $\tau$ the image of a transversal.

\subsection{Universal circle}\label{ss.universalcircle}
One consequence of Matsumoto's result (Theorem \ref{teo.matsumoto}) is that every minimal foliation in $M$ is homeomorphic to our model and so we can compare the geometry of leaves with that of hyperbolic disks simultaneously. This allows to make natural projections into $\HH^2$ of lifts of leaves of a minimal foliation $\cF$ in $M$ to the intermediate cover $\wihatM$ or universal cover $\wt M$. 

We consider $p: M \to S$ to be the projection of the fiber bundle $S^1 \to M=T^1S \to S$. By our coordinate choices, the map $p$ lifts to a projection $\hat p : \wihat M \to \HH^2$, that, given the identification $\wihatM \cong \HH^2 \times \partial \HH^2$ corresponds to the map $\hat p(x,\xi) = x$. We also denote by $\tilde{p}: \wt M \to \HH^2$ the lift to the universal cover. 

For a minimal foliation $\cF$ on $M$ we will denote by $\wihat\cF$ and $\wt \cF$ the lifts to $\wihatM$ and $\wt M$ respectively. For concreteness we will consider the following metric on $\wihatM \cong \HH^2 \times \partial \HH^2$ given by the fact that the metric on $\HH^2 \times \{\xi\}$ is the one that makes projection in the first coordinate an isometry, makes the sets $\HH^2 \times \{\xi\}$ and $\{x\} \times \partial \HH^2$ orthogonal and measures distances in $\{x\} \times \partial \HH^2$ via the visual metric, that is, given an interval $[\xi_1, \xi_2] \subset \partial \HH^2$ we define its length as the angle between the geodesic rays starting from $x$ and landing on $\xi_1$ and $\xi_2$ respectively (and measured in the direction on which rays land in the interior of the interval). This metric is invariant under the action of $\Gamma$ and thus produces a Riemannian metric on the quotient $M$ which makes the projection $p: M \to S$ to be a Riemannian submersion on $S$ when given the metric induced by the action of $\Gamma$ on $\HH^2$. In coordinates $\wihatM \cong \HH^2 \times \partial \HH^2$ the Riemannian metric is:

\begin{equation}\label{eq:metricMhat}
\langle v , w \rangle_{(x,\xi)} = \left(\langle v_{\HH^2} , w_{\HH^2} \rangle^2_{T_x\HH^2} + \langle v_{\partial \HH^2} , w_{\partial \HH^2} \rangle_{(x,\xi)}^2 \right)^{1/2}.
\end{equation}

\noindent where $v_{\HH^2}, w_{\HH^2}, v_{\partial \HH^2}, w_{\partial \HH^2}$ are the projections of the vectors on the first and second coordinates respectively, the inner product $\langle \cdot, \cdot \rangle_{T_x\HH^2}$ is the standard inner product in the hyperbolic plane and $\langle \cdot, \cdot \rangle_{(x,\xi)}$ is an inner product on $T_{(x,\xi)} (\{x\} \times \partial \HH^2)$ given by the identification of the landing of geodesic rays of $T^1_{x}\HH^2$ with $\partial \HH^2$ at the point $\xi \in \partial \HH^2$. 

The metric on $\wt M$ will be the path metric associated with
the pull back of the Riemannian metric above by the universal cover projection. 
This metric in $\wt M$ will be denoted by $d: \wt M \times \wt M
\to \mathbb{R}_{\geq 0}$.

Given a leaf $L$ of a foliation $\cF$ we consider the path distance in $L$ induced by the restriction of the ambient Riemannian metric on $L$. The choices we have made are not important if we consider objects up to  \emph{quasi-isometry}: if one chooses another metric, one obtains a distance that is quasi-isometric to the first one. Recall that a map (not necessarily continuous) $q: (X_1,d_1) \to (X_2,d_2)$ is a $Q$-quasi-isometry if for every $x,y \in X_1$ one has:

\begin{equation}\label{eq:QI}
\frac{1}{Q} d_1(x,y) - Q \leq d_2(q(x),q(y)) \leq Qd_1(x,y) + Q, 
\end{equation} 

\noindent and the image of $q$ is $Q$-dense in $X_2$. Being quasi-isometric (meaning there exists a quasi-isometry between the spaces) is an equivalence relation between metric spaces which is particularly relevant for Gromov hyperbolic spaces such as $\HH^2$. We will use several basic properties of Gromov hyperbolic spaces and refer to \cite{BridsonHaefliger} for proofs of those facts. 

Theorem \ref{teo.matsumoto} implies the following which in particular shows that leaves of a minimal foliation are Gromov hyperbolic: 

\begin{prop}\label{prop.QI}
There is a uniform constant $Q_0:= Q_0(\cF)$ such that for every $L \in \wihat\cF$ (or $\in \wt \cF$) it follows that the restriction $\hat p|_{L}: L \to \HH^2$ (resp. $\tilde p|_{L}: L \to \HH^2$) is a $Q_0$-quasi-isometry. 
\end{prop}

\begin{proof}
After isotopy one can assume the foliation is horizontal (cf. Corollary \ref{coro.foliations} (\ref{it.britenham})), 
and thus, by choosing an appropriate metric we have that the projection is an isometry. Since $M$ is compact it follows that the lift of either $\hat p$ or $\tilde p$ restricted to any leaf is a uniform quasi-isometry. Since the statement (up to changing the constants) is invariant under change of metric on $M$, we conclude. 
\end{proof}

\begin{remark} A far reaching generalization is Candel's uniformization theorem (see \cite[Chapter 7]{Calegari-book}). It can be used to show that every minimal foliation on a 3-manifold with fundamental group of exponential growth admits a metric which makes every leaf of negative curvature (see \cite[\S 5]{FP} and \cite[Appendix A]{BFP}). 
\end{remark}

We can identify the Gromov boundary of $L$
or, equivalently,  the \emph{circle at infinity} $S^1(L)$ of each leaf $L \in \wihat\cF$ (or $\wt \cF$) with $\partial \HH^2$ in a canonical way.  Notice that the \emph{universal circle} of $\cF$ as defined in \cite{Thurston} is also canonically identified with $\partial \HH^2$ in this case.

\subsection{Non-marker points}\label{ss.nonmarker}

Here we introduce the notion of marker and non-marker points for the foliations we are interested in. More general definitions can be found in \cite{Calegari-book}. We start by analyzing $\cF_{ws}$ using the metric given by \eqref{eq:metricMhat} and then in the next subsection we show similar properties for every minimal foliation (because by Theorem \ref{teo.matsumoto}, they  are all homeomorphic to $\cF_{ws}$). 

An orientation will be fixed on $\HH^2$. Then given an oriented geodesic $\ell \in \HH^2$ we denote $H_+(\ell)$ and $H_-(\ell)$ the half spaces determined by $\ell$ (that is, the closure in
$\HH^2$ of the connected components of $\HH^2 \setminus \ell$) with respect to the chosen orientations. Specifically $H_+(\ell)$ is the half space to the
left of $\ell$ and $H_-(\ell)$ is the half space to the 
right of $\ell$, with respect to the orientation in $\mathbb{H}^2$. For $X \subset \HH^2$ and $C>0$ denote $B_C(X)$ to be the set of points $x \in \HH^2$ whose distance to $X$ is less than or equal to $C$.

\begin{prop}\label{prop.elementaryhyperbolic}
Given $\eps \in (0,\pi)$ there exists $C:=C(\eps)$ so that for every interval $[\eta,\xi] \in \partial \HH^2$ the set of points $x \in \HH^2$ so that the visual length from $x$ of the interval $[\eta,\xi]$ is less than $\eps$ is $H_-(\ell) \setminus B_C(\ell)$ where $\ell$ is a geodesic joining $\eta, \xi$ oriented so that the interval $[\eta, \xi]$ is contained in the closure of $H_+(\ell)$ in the compactification $\HH^2 \cup \partial \HH^2$. \end{prop}
\begin{proof}
For every $x \in H_+(\ell)$ the visual length from $x$ of $[\eta,\xi]$ is greater than or equal  to   $\pi>\eps$, so the set we wish to describe does not intersect $H_+(\ell)$. 

Now, fix a geodesic ray $r:[0,\infty) \to \HH^2$, parametrized
by unit speed,  starting at some point $r(0)$ in $\ell$, orthogonal to it and contained in $H_-(\ell)$. For $t >t'$ the ideal geodesic triangle joining $r(t), \xi,\eta$ contains the triangle joining $r(t'), \xi,\eta$ so the inner angle at the point $r(t)$ is smaller than the one at the point $r(t')$ by Gauss-Bonnet formula. As it varies continuously, and verifies that the angle is $\pi$ at $r(0)$ and tends to $0$ at $r(\infty)$ it follows that every such geodesic ray has a unique point $r(t_0)$ on which the angle is exactly $\eps$ and it follows also that $d(r(t_0),\ell) = t_0 := C(\eps)$.  To show that this does not depend on $r(0)$ consider the one parameter family of isometries of $\mathbb{H}^2$ fixing $\xi, \eta$; these isometries map $r(t_0)$ transitively
along the boundary of $B_{t_0}(\ell)$ preserving angles. Hence this proves 
that $t_0$ does not depend on $r(0)$, and depends only
on $\eps$. This proves the proposition. 
\end{proof}

\begin{figure}[ht]
\begin{center}
\includegraphics[scale=0.98]{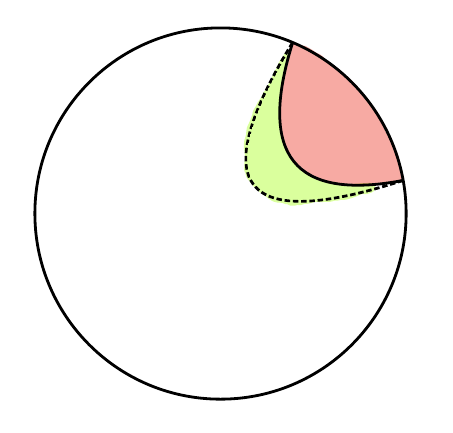}
\begin{picture}(0,0)
\put(-82,191){$\xi$}
\put(-20,123){$\eta$}
\put(-73,144){$H_+(\ell)$}
\put(-97,119){$B_C(\ell)$}
%\put(-129,52){$A_2^s$}
%\put(-87,77){{\small $\cG^u(p)$}}
%\put(-196,91){{\small $\cG^s(p)$}}
%\put(-113,225){$A_1^s$}
%\put(-82,220){$A_2^u$}
\end{picture}
\end{center}
\vspace{-0.5cm}
\caption{{\small The set $D_\eps(I)$ for $I = [\xi,\eta]$ is 
the complement of the union of $H_+(\ell)$ and $B_C(\ell)$. If $\eps$
is very small, than $C$ is very large.}}\label{fig-setDeps}
\end{figure}
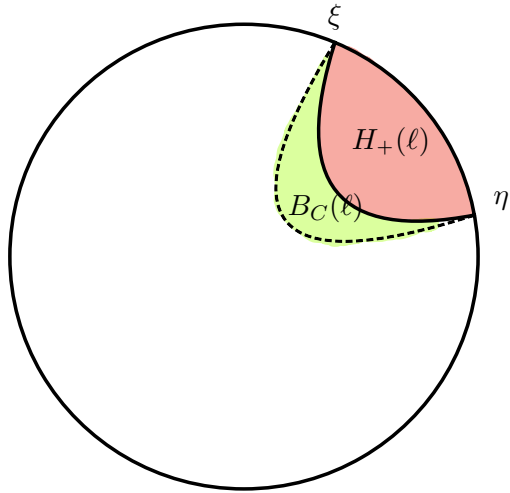

We can parametrize leaves of $\wihat\cF_{ws}$ by $\partial \HH^2$ as these are of the form $\HH^2 \times \{\xi\}$ with $\xi \in \partial \HH^2$. We denote by $L_\xi = \HH^2 \times \{\xi\}$ and call $\xi$ the \emph{non-marker} point of $L_\xi$. This will be denoted as $\alpha(L_\xi) = \xi$. The key point is that the choices of coordinates make this point in $\partial \HH^2$ special with respect to the leaf $L_\xi$ as it will now be explained. 

Given $\eps>0$ and an interval $I \subset \partial \HH^2$  we define the set (see figure \ref{fig-setDeps})

\begin{equation}\label{eq:Deps}
D_{\eps}(I) = \{ x \in \HH^2 \ : \ \forall \xi,\eta \in I \: \ d_{\wihat{M}}((x,\xi), L_{\eta}) < \eps \} 
\end{equation}

\noindent
where $d_{\wihat{M}}$ is the distance in $\wihat{M}$.
Note that the point $(x,\xi)$ belongs to $L_\xi$ by definition. We will always assume that $I$ is not $\partial {\mathbb{H}}^2$ nor a single point. For any interval $I=[\xi_-,\xi_+]$, the points in $D_\eps(I)$ form
a subset of $\mathbb{H}^2$ and the points in $\partial \HH^2$ for which the biggest distance between the corresponding leaves is achieved is when $\{\eta, \xi\}=\{ \xi_-, \xi_+ \}$. 
%then to see which It is clear that to define the set $D_\eps(I)$ for 
%interval $I=[\xi_-,\xi_+]$ 
So it is enough to look at the distance of points of the form $(x,\xi_-)$ to $L_{\xi_+}$ (or $(x,\xi_+) \in L_{\xi_-}$). From the previous proposition we deduce:   

\begin{coro}\label{coro.setD}
Given $I \subset \partial \HH^2$ an  interval and some constant  $\eps \in (0,\pi)$ there exists $C:=C(I, \eps)$ such that $D_{\eps}(I)$ is equal to the set $H_-(\ell) \setminus B_C(\ell)$. Here $\ell$ is the geodesic joining the endpoints of $I$ oriented so that $I$ is contained in the closure of $H_+(\ell)$ in $\HH^2 \cup \partial \HH^2$. 
\end{coro}

%\marg{S: I think there should be equality and not containment here. \\ R: It is containment because the closure of $H_+(\ell)$ has points in $\HH^2$ too.. \\ S: I still think it's equality. Notice that $B_C(\ell)$ is closed
%in  $\HH^2$ and $D_\eps(I)$ is open in $\HH^2$.}

We will denote the lift of the points of $D_\eps(I)$ in $L_\xi$ as 

\begin{equation}\label{eq:Depsxi} 
D_{\eps}(\xi, I) = \{ (x,\xi) \in \HH^2 \times \{\xi\} \ : \ x \in D_\eps(I) \} \subset L_\xi. 
\end{equation}

\subsection{Minimal foliations in $M = T^1 S$}
\label{ss.minimal}

Now, let $\cF$ be an arbitrary minimal foliation on $M$. 
By Theorem \ref{teo.matsumoto} there is a homeomorphism $h: M \to M$ 
mapping $\cF$ to $\cF_{ws}$. 

We will produce maps $\Phi_L$ for $L$ in $\wcF$
which are uniform quasi-isometries (cf. Proposition \ref{prop.QI}).
First fix an isotopy of $M$, which
restricted to each individual leaf of 
$\cF$ is as regular as the leaves of $\cF$ are, to 
get to a foliation $\cF'$ which is transverse to the standard
Seifert fibration of $M = T^1 S$ (cf. Corollary \ref{coro.foliations} (\ref{it.britenham})). This is the
Seifert fibration whose fibers are the
unit vectors over a given point in $S$.
The lift of this isotopy to $\wt M$ is denoted by $\nu$.
For any leaf of $\wcF'$ project to $\HH^2$ using the lift of
this Seifert fibration, this projection is $\tilde{p}$. 
Since the angle between $\cF'$ and
the Seifert fibration is bounded below, this projection
is a uniform quasi-isometry.
The composition of the initial isotopy (lifted to $\mt$)
and the projection is well defined up to a bounded distortion
(depending on $\cF$).
We denote this composition by $\Phi_L$.
The conclusion is that we obtain an equivariant collection
of uniform quasi-isometries $\Phi_L$ from leaves $L \in \wcF$
to $\HH^2$. Note also that if $\tilde p: \mt \to \HH^2$ is the standard projection, there is a uniform constant $C_0>0$ so that $d_{\HH^2}(\Phi_L(x), \tilde p(x))< C_0$. 
Note that $\Phi_L = \tilde{p} \circ \nu$ restricted to $L$.

This allows to: 

\begin{itemize}
\item Associate to each $L \in \wcF$ a point,
which we will denote throughout the article by  $\alpha(L) \in \partial \HH^2$,
and  that we will call the \emph{non-marker point} of $L$. It also identifies the leaf space of $\wihat\cF$ with $\partial \HH^2 \cong S^1$. 
This uses both $h$ and the quasi-isometries above. First the map $h$
shows that for any leaf $F$ of $\cF$ there is one ideal
direction which is transversely non contracting (the non marker
direction), and all other directions are contracting.
In particular, by lifting to appropriate covers, this is
also true for leaves in $\wcF$ or $\wihat\cF$.
The uniform quasi-isometries then allow for $L \in \wcF$ to
obtain the unique non marker point $\alpha(L)$ in $\partial \HH^2$.
Finally using $h$ again one can show that the map $\alpha$
is a homeomorphism when considered as a map from the leaf
space of $\wihat\cF$ and $\partial \HH^2 \cong S^1$.

\item For each $L$ in $\wcF$ (or $\wihat\cF$)
Identify the (Gromov) boundary at infinity $S^1(L)$ of $L$ with $\partial \HH^2$ using the fact that $\Phi_L$ is a quasi-isometry. Note 
the very important fact that the induced map $\Phi_L : S^1(L) \to \partial \HH^2$ is independent on the choice of the isotopy 
of $\cF$ to a horizontal foliation.
\item We have maps
\begin{equation}\label{eq:PhiL} 
\Phi_L : L \cup S^1(L) \to \HH^2 \cup \partial \HH^2, 
\end{equation} 
As explained before each such map, when restricted to $L$  is a diffeomorphism, for each $L$;  and the 
diffeomorphisms have uniformly bounded derivatives (thus, it is a uniform quasi-isometry, independent on $L$). 
\end{itemize} 

Note that $\alpha(L)$ belongs to $\partial \HH^2$ but we can identify it canonically with a point in $S^1(L)$ via $\Phi_L$ and sometimes we will go back and forth with these identifications. Also, these facts hold for leaves $L \in \wihat\cF$ and we will abuse notation and denote by $\Phi_L: L \cup S^1(L) \to \HH^2 \cup \partial \HH^2$ all such maps.

\subsection{Nearby sets in distinct leaves}

For $\eps>0$, a leaf $L \in \wt{\cF}$ (or $\hat{\cF}$ using the
metric $d_{\wihat M}$)  and an interval $I$ of the leaf space of $\wt{\cF}$ (or $\wihat\cF$) we define

\begin{equation}\label{eq:setDfol}
\wihatD_\eps(L, I) = \{ y \in L \ : \ \forall E \in I \ , \  d(y, E) < \eps \}
\end{equation}

Using the map $\Phi_L$ and the fact that $\cF$ is homeomorphic to $\cF_{ws}$ and Corollary \ref{coro.setD} we deduce: 

\begin{prop}\label{prop.setDfol}
There is a constant $C>0$ independent of $L$ such that the set $\wihatD_\eps(L,I)$ verifies that it is at Hausdorff distance less than $C$ from $\Phi_L^{-1}(D_\eps(I))$. In particular, since $\Phi_L$ is a quasi-isometry, there is another constant $C'>0$ such that the Hausdorff distance between $\Phi_L(\wihatD_\eps(L,I))$ and $D_\eps(I)$ is less than $C'$. Moreover, if $L_1 \in I$ and $x_n \in \wihatD_\eps(L,I)$ is a sequence of points converging to some point $\xi \in S^1(L)$ which is not $\alpha(L')$ for some $L' \in I$ then we have that $d(x_n, L_1) \to 0$. 
\end{prop} 

\begin{proof}
Recall that we start with an isotopy from $\cF$ to a horizontal foliation,
then project using the Seifert fibration. These are the maps $\Phi_L$.
Choose $\eps'$ depending on $\eps$ and so that if points are within
$\eps'$ then after undoing the isotopy the points are at most $\eps$
from each other.
This shows that 

$$\Phi^{-1}_L(D_{\eps'}(I)) \ \ \subset \ \ \hat D_\eps(L,I).$$

\noindent
Since $D_{\eps'}(I)$ and $D_\eps(I)$ are a bounded Hausdorff distance from
each other, then there is $C_1$ with $B_{C_1}(\Phi^{-1}_L(D_\eps(I))
\subset \hat D_\eps(L,I)$.

Conversely, given $\eps$, since the isotopy is the lift of a
compact isotopy there is $\eps_1 = \eps_1(\eps)$ so that 
$\Phi_L(\hat D_\eps(L,I)) \subset D_{\eps_1}(I)$.
But $D_{\eps_1}(I)$ is a bounded Hausdorff distance
from $D_\eps(I)$, so there is $C_2 > 0$ so that 

$$\Phi_L(\hat D_\eps(L,I)) \ \ \subset \ \ B_{C_2}(D_\eps(I)).$$

\noindent
Taking $\Phi^{-1}_L$, and noticing that it is a quasi-isometry,
produces $C_3 > 0$ so that $\hat D_\eps(L,I) \subset 
B_{C_3}(\Phi^{-1}_L(D_\eps(I))$.
This finishes the proof of the proposition.
\end{proof}

This proposition will combine well with Corollary \ref{coro.setD} to control the geometry of the sets $D_\eps(L,I)$. 

\subsection{Minimality of the action in the universal circle}\label{s.minimality} 

In the next proposition we collect some facts about the action of the fundamental group of $S$ in the boundary $\partial \HH^2$. Recall that we have chosen a fixed hyperbolic metric on $S$ which is induced by a subgroup $\Gamma \cong \pi_1(S)$ of isometries of $\HH^2$ that induces an action on $\partial \HH^2$ (recall that for such groups every non identity element acts as a hyperbolic isometry, so it has exactly two fixed points in $\partial \HH^2$, one attracting and one repelling). 

Note that the fundamental group $\pi_1(M)$ of $M$ is a central extension of $\pi_1(S)$ (that is, there is a surjective morphism $\pi_1(M) \to \pi_1(S)$ so that the preimage of the identity is the center of the group $\pi_1(M)$ and is generated by the homotopy class of the fibers) and its action on $\partial \HH^2$ is induced by the action of its projection on $\pi_1(S)$. 

We have:

\begin{prop}\label{prop.minimality}
The action of $\Gamma \cong \pi_1(S)$ in $\partial \HH^2$ is minimal (i.e. every orbit is dense). Moreover, given $U,V$ open sets in $\partial \HH^2$ there exists an element $\gamma \in \Gamma$ such that: 
\begin{itemize}
\item $\gamma(U) \cap V \neq \emptyset$, 
\item the fixed points of $\gamma$ are one contained in $U$ and one contained in $V$.  
\end{itemize}
\end{prop}

This result is classical, a proof can be found for instance in \cite{FPmin} where we prove an extension to general uniform $\RR$-covered foliations (\cite[Proposition 5.3]{FPmin} applied to the foliation by compact surfaces in $S \times S^1$ gives the previous result). Note that since we can identify the leaf space of $\wihat\cF$ with $\partial \HH^2$ this also gives information about the action of $\Gamma$ acting on the leaf space of $\wihat\cF$ on $T^1\HH^2$ which is a circle identified with $\partial \HH^2$ via the map $\alpha$ sending a leaf into its non-marker point. If we go to the universal cover, then, the central extension of $\Gamma$ that is $\pi_1(M)$ provides all lifts of the action of $\Gamma$ in $\partial \HH^2$ to the universal cover (in particular, there are always lifts with fixed points and the action is minimal).

\subsection{Some plane topology} 

We will use the following standard consequence of the classical Schoenflies theorem (see e.g. \cite[\S 9]{Moise}). Note that we will always be using piecewise smooth curves, so the proof is simpler.

\begin{prop}\label{prop.plane}
Let $c$ be a properly embedded curve in the plane. Then, the complement of $c$ is the union of two topological open disks whose boundary in the plane is exactly $c$. 
\end{prop}

Since we will always be working on $\HH^2$ where we have a natural compactification $\HH^2 \cup \partial \HH^2$ homeomorphic to a disk, we want to understand the complements of properly embedded curves in this compactification. For a set $K \subset \HH^2$ the \emph{limit set} of $K$ is the closure of $K$ in $\HH^2 \cup \partial \HH^2$ intersected with $\partial \HH^2$. We consider a circular order in $\partial \HH^2$.
Let $c$ be a properly embedded curve in $\HH^2$ and denote by $I_1$ and $I_2$ the limit sets of the two rays of $c$ (i.e. consider $x \in c$ and denote by $c_1, c_2$ the connected components of $c \setminus \{x\}$, then $I_i$ is limit set of $c_i$).
Notice that both $I_1, I_2$ are connected subsets of $\partial \HH^2$
and are either $\partial \HH^2$ or an interval. In
the case that they (or one of them) are not $\partial \HH^2$ we denote
them by: $I_1 := [a, b], \ I_2 := [c,d]$ in the circular order
of $\partial \HH^2$\footnote{We allow $a=b$ or $c=d$ in which case the interval is a point.}.

\begin{coro}\label{coro.separationH2}
Let $c$ be a properly embedded curve in $\HH^2$ as above. 
Let $D^+$ and $D^-$ denote the connected components of $\HH^2 \setminus c$,
and let $J^+, J^-$ be their respective limit sets. Then 
\begin{itemize}
\item 
If $I_1 \cup I_2  = \partial \HH^2$ then both $J^+$ and $J^-$ coincide with $\partial \HH^2$ (this includes the case where one of the intervals 
$I_1$ or $I_2$ coincides with $\partial \HH^2$). 

\item Assume that $[a,b] \cap [c,d] = \emptyset$, and suppose
that $a, b, c, d$ are circularly ordered in $\partial \HH^2$.
Then one of $J^+, J^-$ is $[a,d]$
and the other is $[c,b]$.
\item Finally suppose $[a,b], [c,d]$ intersect (and their union is not $\partial 
\HH^2$). Then one of $J^+, J^-$ is $\partial \HH^2$ and the other
is $[a,b] \cup [c,d]$.
\end{itemize}
\end{coro}

\begin{proof}
Note that the limit set $J^\pm$ of $D^{\pm}$ is a compact connected subset of the boundary $\partial \HH^2$. Since $c$ is the boundary in $\HH^2$ of both $D^{+}$ and $D^-$ it follows that $I_1 \cup I_2$ is contained in both $J^{+}$ and $J^-$. This already proves the first point. 

For the second and third items, it is implicitly assumed that
none of $I_1, I_2$ are $\partial \HH^2$.
%Since $J^+$ and $J-$ are intervals, and 
Since $D^+ \cup D^-$ must accumulate in all of $\partial \HH^2$, it follows that if one considers a point  $\xi \in \partial \HH^2 \setminus (I_1 \cup I_2)$ then it has a neighborhood $N$ in the compactification $\HH^2 \cup \partial \HH^2$, so that
$N \cap \HH^2$ which is contained in either $D^+$ or $D^-$. 
To finish we must show that $(b,c)$ is contained in one of 
$J^+, J^-$ and $(d,a)$ is contained in the other.
To do that consider a geodesic $\mu$ in $L$ with one ideal point in
$(b,c)$ and the other in $(d,a)$. The ideal points are disjoint
from $I_1 \cup I_2$ so $\mu$ has rays contained in $D_1 \cup D_2$.
If both rays are contained in say $D_1$ it follows that
both rays of $c$ are also contained in the same complementary
component of $\mu$. This contradicts that $c$ limits on both $[a,b]$ and
$[c,d]$. Hence only one ideal point is in $J^+$ and the
other is in $J^-$, and consequently one of $(b,c), (d,a)$ is contained
in $J^+$ and the other in $J^-$. 
This proves the second statement. 

For the third statement: since $I_1$ intersects $I_2$, then
the above fact implies that
$\partial \HH^2 \setminus
(I_1 \cup I_2)$ is contained in one and only one of $J^+$ or $J^-$.
Since $I_1 \cup I_2 \subset J^+ \cap J^-$, the third statement follows.
\end{proof}

We will also need the following consequence of Proposition \ref{prop.plane}

\begin{prop}\label{prop.disjointrays}
Let $r_1,r_2$ be two disjoint properly embedded rays in $\HH^2$ which limit in intervals $I_1$ and $I_2$ respectively. Assume that $I_1 \neq \partial \HH^2$, then $I_2$ cannot be contained in the interior of $I_1$. 
\end{prop}
\begin{proof}
Consider a point $o \in \HH^2$ and two geodesic rays $g_1, g_2$ from $o$ landing at points $\xi_1, \xi_2$ in the interior of $I_1$ and a ray $g_3$ landing at a point $\xi_3 \notin I_1$. We claim that there are arcs $v_n$ of $r_1$ joining $g_1,g_2$ which converge in the topology of $\HH^2 \cup \partial \HH^2$ to the interval $[\xi_1, \xi_2]$ contained in $I_1$. This is because given $R>0$ we have that outside the disk $B_R(o)$ of radius $R$ centered in $o$ there are sequences of points $x_n,y_n$ of $r_1$ converging to points in $I_1 \setminus [\xi_1,\xi_2]$ and so that the arc $J_n$ joining $x_n$ to $y_n$ is completely outside $B_R(o)$ and cannot intersect $g_3$. Thus, there must be an arc inside $J_n$ that we call $v_n$ joining $g_1$ and $g_2$ and which is outside $B_R(o)$ and contained in the region bounded by $\partial B_R(o) \cup g_1 \cup g_2$ which accumulates in $[\xi_1,\xi_2]$. This proves our claim.

\begin{figure}[ht]
\begin{center}
\includegraphics[scale=0.98]{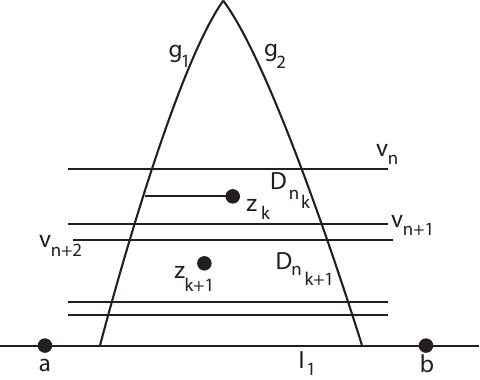}
\begin{picture}(0,0)
%\put(-203,153){$o_1$}
%\put(-130,153){$p$}
%\put(-53,144){$o_2$}
%\put(-160,52){$A_1^u$}
%\put(-129,52){$A_2^s$}
%\put(-87,77){{\small $\cG^u(p)$}}
%\put(-196,91){{\small $\cG^s(p)$}}
%\put(-113,225){$A_1^s$}
%\put(-82,220){$A_2^u$}
\end{picture}
\end{center}
\vspace{0.0cm}
\caption{{\small The limit set of $r_1$ is the interval $I_1$ with
endpoints $a, b$. $g_1, g_2$ are geodesic rays with ideal
points in the interior of $I_1$. The segments $v_n$ are segments in 
$r_1$ which limit to $[a,b]$. The points $z_k$ are in $r_2$
and have to connect outside the $v_n$ to one of $g_1$ or $g_2$ in $D_{n_k}$.
So the intersections with (say) $g_1$ limit to an arbitrary point in 
the interior of $I_1$.}}
\label{fig-t3}
\end{figure}

We refer to Figure \ref{fig-t3} for the situation in this Lemma.

Let $g'_i$ be the subray of $g_i$ starting in $v_n \cap g_i$.
Now, if we consider the regions $D_n$ obtained as the closure in $\HH^2 \cup \partial \HH^2$ bounded by $v_n \cup g_1 \cup g_2$ and accumulating in $[\xi_1, \xi_2]$ we get that after subsequence we can assume it is a nested sequence $D_{n+1} \subset D_n$ of disks so that $\bigcap D_n = [\xi_1,\xi_2]$. If $r_2$ is a properly embedded ray disjoint from $r_1$ which accumulates in a point $\xi \in (\xi_1,\xi_2)$ we can see that there is a sequence of points $z_k \in r_2$ converging to $\xi$ and thus must be contained in some sets of the form $D_{n_k} \setminus D_{n_{k+1}}$. Since $r_2$ is connected, to join $z_k$ with $z_{k+1}$ it must exit $D_{n_k}$ without intersecting $D_{n_{k+1}}$ thus there must be an arc $t_k$ of $r_2$ that joins $z_k$ to some point $w_k$ in either $g_1$ or $g_2$. Up to subsequence, we can assume without loss of generality that $w_k$ is in $g_1$ and we deduce that $r_2$ must limit in $[\xi_1,\xi]$. Since $\xi_1,\xi_2$ were arbitrary, we deduce that $r_2$ cannot limit in an interval completely contained in the interior of $I_1$ as desired. 
\end{proof}

\begin{remark}\label{rem.disjointrays}
Note that in the previous proposition, if $I_1 = \partial \HH^2$, then we cannot consider the ray $g_3$ which forces the curves $v_n$ to join $g_1,g_2$ in the 'correct' side. It is possible however to show that there are some restrictions on the possible landing regions of disjoint properly embedded rays in this setting that will be used later, but we will not give a precise general statement. 
\end{remark}

%%%%%%%%%%%%%%%%%%%%%%%
\section{Pairs of transverse foliations}\label{s.transverse}

In this section we will consider two transverse minimal foliations $\cF_1$ and $\cF_2$ in $M=T^1S$. We will call $\cG:= \cF_1 \cap \cF_2$ the one-dimensional foliation by the connected components of intersections of  leaves of $\cF_1$ and $\cF_2$. Note that since $S$ is orientable it induces orientations on $M$, $\cF_1$ and $\cF_2$ and thus also on $\cG$, we fix one such orientation and lift it to $\mt$. 

 We consider the maps $\Phi_L^1$ and $\Phi_E^2$ from the compactification of leaves $L \in \wihat\cF_1$ (or $\widetilde \cF_1$) and $E \in \wihat\cF_2$ (or $\widetilde \cF_2$) as defined in equation \eqref{eq:PhiL}. Since the foliation will be implicit from the leaf, we will usually omit the supraindex and denote $\Phi^1_L$ by $\Phi_L$ when it is clear that $L \in \wt{\cF_1}$ or $\hat{\cF_1}$. 

\subsection{Some general properties}\label{ss.landing}

Given a leaf $c \in \cG$ and $x \in c$ we will denote by: 

\begin{equation}
c_x^+ \cup c_x^- = c \setminus \{x\}, 
\end{equation}

\noindent the two rays of $c$ where $c_x^+$ is the one in the positive orientation starting at $x$. We will denote by $c^+$ and $c^-$ the \emph{ray class} meaning the equivalence class of rays of $c$ so that given two such rays,
one of them is contained in the other. 

Given a curve $c \in \wihat\cG$ (or $\widetilde \cG$) we know that $c$ is a connected component of $L \cap E$ where $L \in \wihat\cF_1$ and $E \in \wihat\cF_2$ (or in the corresponding lifts to the universal cover\footnote{For the remainder of this section we will omit saying that things hold both in the cover $\wihatM$ and the universal cover $\mt$.}). The following is a direct consequence of the quasi-isometric properties of $\Phi_L^1$ and $\Phi_E^2$:

\begin{lemma}\label{lema.proyection}
The curves $\Phi_L^1(c)$ and $\Phi_E^2(c)$ are proper and a bounded distance away from each other in $\HH^2$. \end{lemma}

In particular, if we denote by $\overline{c}_L$ the closure of $c$ in $L \cup S^1(L)$ and by $\overline{c}_E$ the respective closure in $E \cup S^1(E)$ we get the following important property:

\begin{coro} \label{coincide}
The sets $\Phi_L(\overline{c}_L \setminus c)$ and $\Phi_E(\overline{c}_E \setminus c)$ are contained in $\partial \HH^2$ and coincide. 
\end{coro}

We will then denote for $c \in \wihat\cG$: 

\begin{equation}\label{eq:indepfol}
\partial \Phi(c) := \Phi_L(\overline{c}_L \setminus c)=\Phi_E(\overline{c}_E \setminus c) \subset \partial \HH^2.
\end{equation}

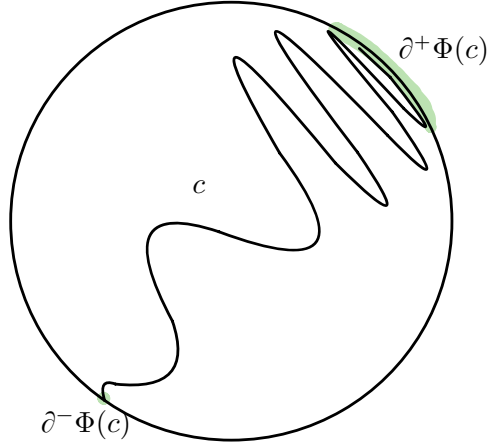
\begin{figure}[ht]
\begin{center}
\includegraphics[scale=0.78]{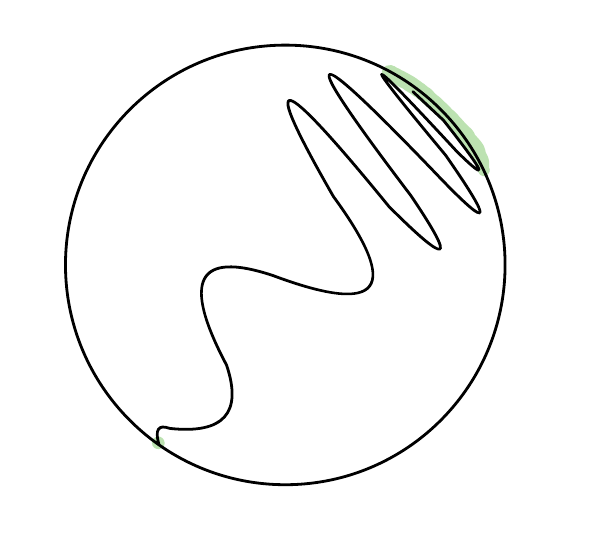}
\begin{picture}(0,0)
\put(-133,113){$c$}
\put(-190,23){$\partial^{-}\Phi(c)$}
\put(-56,164){$\partial^+\Phi(c)$}
%\put(-160,52){$A_1^u$}
%\put(-129,52){$A_2^s$}
%\put(-87,77){{\small $\cG^u(p)$}}
%\put(-196,91){{\small $\cG^s(p)$}}
%\put(-113,225){$A_1^s$}
%\put(-82,220){$A_2^u$}
\end{picture}
\end{center}
\vspace{-0.5cm}
\caption{{\small An example where the negative ray of $c$ lands and 
the positive ray of $c$ accumulates in an interval which is not a point.}}\label{fig-landing}
\end{figure}

Since $c$ is properly embedded we can write $\partial \Phi(c) = \partial^+ \Phi(c) \cup \partial^- \Phi(c)$ where each denotes the accumulation points of the positive and negative rays of $c$ once a point is removed (it is easy to see that this is independent on the removed point and so it is a property of the ray class). Note that each of $\partial^+\Phi(c)$ and $\partial^- \Phi(c)$ are compact connected and non empty subsets of $\partial \HH^2$ (see figure \ref{fig-landing}).
Now we define a fundamental property to be analyzed in this
article:

\begin{definition} \label{landdefinition}
Given $c \in \wihat\cG$ we say that the \emph{positive ray lands} (resp. the negative ray lands) if $\partial^+ \Phi(c)$ is a point (resp. $\partial^-\Phi(c)$ is a point). 
\end{definition}

For notational simplicity, we also denote by $\partial \Phi(r)$ the accumulating set of a ray when $r$ is a ray of a leaf $c \in \cG_L$. 
In addition if $r$ lands, we say that $\partial \Phi(r)$ is the landing
point of $r$.
If $r$ in $L$ lands, then $r$ also limits in
a single point $b$ in $S^1(L)$ (in particular
$\Phi_L(b) = \partial \Phi(r)$). We also say that $r$ lands
in $b$.

\subsection{Pushing to nearby leaves} 
If $\eps$ is sufficiently small it is smaller than charts on which $\cF_1$ is horizontal and $\cF_2$ vertical and this means that
the whole structure of $\cG_L$ that
one sees in a leaf $L$ of (say) $\wt{\cF_1}$ can be pushed
isotopically to nearby leaves $L'$ of $\wt{\cF_1}$ in
the set $\wihatD_\eps(L,I)$. The roles of the foliations
can be reversed.

\begin{prop}\label{prop-pushing}
There exists $\eps>0$ and $C>0$ such that if $L,E$ are leaves of $\wt{\cF_1}$ which are contained in an interval $I$ of the leaf space of $\wt{\cF_1}$ and $r$ is a segment of a leaf of $\wt \cG$ which is contained in $\wihatD_\eps(L, I)$ (cf. equation \eqref{eq:setDfol}) such that $r \subset F \cap L$ where $F \in \wt{\cF_2}$, then, there exists a segment $r'$ of a leaf in $\wt \cG$ contained in $F \cap E$ such that $\Phi^1_L(r)$ and $\Phi^1_E(r')$ are at distance less than $C$. In particular, if $r$ is a ray, then the landing sets $\partial \Phi(r)$, $\partial \Phi(r')$ of $r$ and $r'$ coincide. 
\end{prop}

\begin{proof}
This is a direct consequence of transversality of the foliations. If $\eps$ is small enough, then as long as points remain at distance less than $\eps$ the intersection of $F$ with $L$ and $E$ will happen in a closeby set. Since the maps $\Phi^1_L$ and $\Phi_E^1$ are quasi-isometries and a bounded distance away from the projection to $\HH^2$, the uniform bound $C$ is obtained. 
\end{proof}

\subsection{Non separated leaves} 

Given a leaf $L \in \widetilde{\cF_i}$ ($i=1,2$) we defined the foliation $\cG_L=\widetilde{\cG}|_{L}$. Denote by $\cL_L= L/_{\cG_{L}}$ its leaf space which is a one dimensional (possibly) non-Hausdorff manifold.  

When $\cL_L$ is non-Hausdorff one has \emph{non-separated leaves}, that is, distinct leaves $c_1,c_2 \in \cG_L$ which are accumulated by a sequence $d_n$ of leaves of $\cG_L$ (since these are foliations, this is equivalent to having sequences $x_n, y_n \in d_n$ so that $x_n \to x_\infty \in c_1$ and $y_n \to y_\infty \in c_2$).  

\begin{remark}\label{rem.nonseparated}
Note that two leaves $c_1, c_2 \in \cG_L$ may be separated while there is no transversal to $\cG_L$ which intersects both, in that case, there must be leaves $e_1,e_2$ (one of which could coincide with $c_1$ or $c_2$) which are non-separated and separate $c_1$ from $c_2$ in the sense that they lie in different connected components of $L \setminus \{e_i\}$. 
\end{remark}

In principle, there is no a priori relation between $\cL_L$ and $\cL_E$ for different leaves $L, E$. However, we can show:

\begin{prop}\label{p.dichotomyHsdff}
Assume that there is some $L \in \widetilde{\cF_i}$ ($i=1,2$) such that the leaf space $\cL_L$ of $\cG_L$ is Hausdorff. Then, for every $E \in \widetilde{\cF_1}$ and $F \in \widetilde{\cF_2}$ the leaf spaces $\cL_E$ and $\cL_F$ are Hausdorff. 
\end{prop}

To show this, we first need the following useful property (see figure \ref{fig-doubleintersect}):

\begin{lema}\label{lem.nonHsdfftwointersect}
Fix $L \in \wt{\cF_1}$. The leaf space $\cL_L$ of $\cG_L$ is non-Hausdorff if and only if there is $E \in \wt{\cF_2}$ such that $L \cap E$ is not connected.  Moreover, if two distinct leaves $c_1,c_2 \in \cG_L$ are non-separated then, there is $E \in \wt{\cF_2}$ such that $c_1 \cup c_2 \subset L \cap E$. 
\end{lema}

\begin{figure}[ht]
\begin{center}
\includegraphics[scale=0.78]{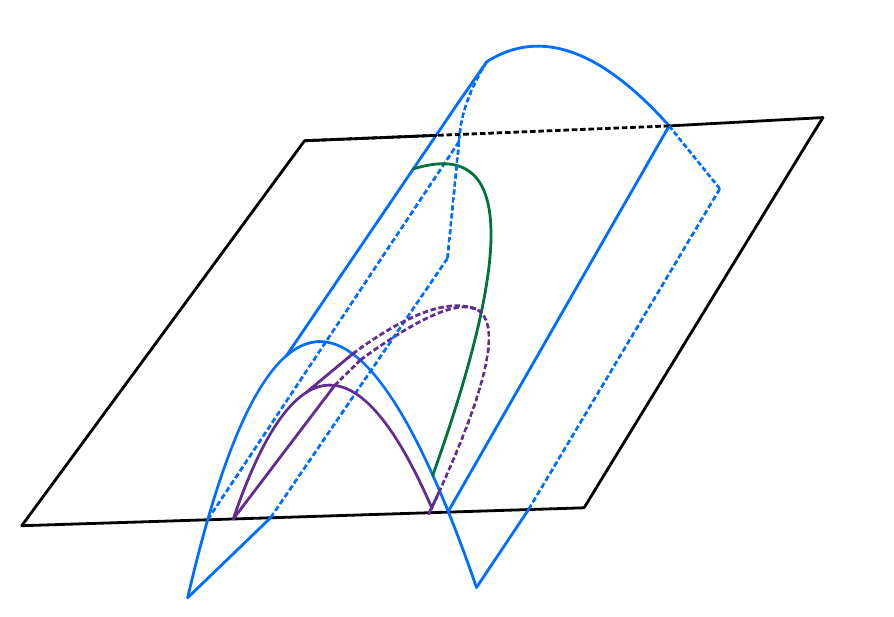}
%\begin{picture}(0,0)
%\put(-203,153){$o_1$}
%\put(-130,153){$p$}
%\put(-53,144){$o_2$}
%\put(-160,52){$A_1^u$}
%\put(-129,52){$A_2^s$}
%\put(-87,77){{\small $\cG^u(p)$}}
%\put(-196,91){{\small $\cG^s(p)$}}
%\put(-113,225){$A_1^s$}
%\put(-82,220){$A_2^u$}
%\end{picture}
\end{center}
\vspace{-0.5cm}
\caption{{\small Intersection in more than one connected component forces non-Hausdorff leaf space of the induced one dimensional foliation.}}\label{fig-doubleintersect}
\end{figure}

\begin{proof}
Assume first that $L \cap E$ is not connected, therefore, there are leaves $c_1, c_2 \in \cG_L$ which belong to $E$. Assume that the leaf space $\cL_L$ of $\cG_L$ is Hausdorff, this means that there exists a transversal to $\cG_L$ joining $c_1$ and $c_2$. This transversal is a transversal to $\wt{\cF_2}$ intersecting $E$ twice, a contradiction to Corollary \ref{coro.foliations}. 

Now, assume that $c_1$ and $c_2$ are non-separated leaves in $\cG_L$.  Let $d_n \in \cG_L$ be a sequence of leaves converging to both $c_1$ and $c_2$. It follows that $d_n \in E_n \cap L$ where $E_n$ is a sequence of leaves of $\wt{\cF_2}$. Consider $x_n, y_n \in d_n$ so that $x_n \to x_\infty \in c_1$ and $y_n \to y_\infty \in c_2$. It follows that $E_n$ converges to both the leaf through $x_\infty$ and the leaf through $y_\infty$. Since the leaf space of $\wt{\cF_2}$ is Hausdorff (see Corollary \ref{coro.foliations}), these leaves must coincide, and thus $c_1 \cup c_2 \subset E \cap L$ where $E \in \wt{\cF_2}$ is the leaf containing $x_\infty$ (and $y_\infty$). 
\end{proof}

\begin{proof}[Proof of Proposition \ref{p.dichotomyHsdff}]
Note first that by the previous lemma, if there is a leaf $L \in \widetilde{\cF_1}$ for which $\cL_L$ is non-Hausdorff, then the same holds for some leaf $E \in \widetilde{\cF_2}$ and vice-versa. 

This implies that it is enough to show that if there is a leaf $L \in \widetilde{\cF_1}$ for which $\cL_L$ is non-Hausdorff, then the same holds for every $L' \in \widetilde{\cF_1}$. 

Since $\cF_1$ is minimal, given an open interval $I$ in the leaf space $\cL_1 = \mt/_{\wt{\cF_1}}$ and a leaf $L' \in \wt{\cF_1}$ there is some deck transformation $\gamma \in \pi_1(M)$ so that $\gamma L' \in I$. Thus, since having non-Hausdorff leaf space is invariant under deck transformations, it is enough to show that there is an interval in the leaf space where every leaf there has non-Hausdorff leaf space. 

Take $L \in \wt{\cF_1}$ and $E \in \wt{\cF_2}$ such that $L \cap E$ is not connected and let $c_1, c_2 \in \cG_L$ be two different connected components. Note that $c_1, c_2$ also belong to $\cG_E$. Consider a transversal $\tau: (-\eps,\eps) \to E$ to $\cG_E$ through
some point $x \in c_1$  (i.e. $\tau(0)=x$) and denote by $c_t$ the leaf of $\cG_E$ through the point $\tau(t)$. Up to changing orientation in $\tau$, we can assume that for $t >0$ the leaves $c_t$ and $c_2$ belong to different connected components of $E \setminus c_1$. 
This is because $c_2$ 
is in one component of $E \setminus c_1$.
Denote $L_t$ the leaf of $\wt{\cF_1}$ through $\tau(t)$. Given a point $y \in c_2$ it follows that for small enough $t \in (-\eps,\eps)$ the leaf $L_t$ intersects $E$ close to $y$, thus, in the same connected component of $E \setminus c_1$ as $c_2$. It follows that for small $t>0$ the leaf $L_t$ intersects $E$ in two connected components and thus $\cG_{L_t}$ has non-Hausdorff leaf space as we wanted to show. 
\end{proof}

\subsection{Non separated rays and landing}\label{ss.nonseparated}

If $c \in \cG_L$, recall that a ray of $c$ is a connected component of $c \setminus \{x\}$ for some $x \in c$ and a ray class is one equivalence class of rays by the relation of one of them
being contained in the other; there are exactly two ray classes for each $c \in \cG_L$. 

\begin{lema}\label{lema.nonseparated}
Let $c_1, c_2 \in \cG_L$ be two non-separated leaves, then, there is exactly one ray class $r$ of $c_1$ and transversals $\tau_1, \tau_2:[0,\eps)\to L$ with $\tau_i(0) \in c_i$ and an increasing homeomorphism $\varphi: (0, t_1] \to (0,t_2]$ for some small $t_1,t_2$ with the property that if $0<t\leq t_1$ and $e_t$ is the ray class of the leaf of $\cG_L$ through $\tau_1(t)$ oriented as $r$ then $e_t$ intersects $\tau_2$ at the point $\tau_2(\varphi(t))$. 
\end{lema}

\begin{proof}
The existence of such segments is because being non separated implies that there are leaves $e_n \in \cG_L$ accumulating in both, so if we fix two transversals $\tau$ and $\eta$ to $c_1$ and $c_2$ we have segments of $d_n$ of $e_n$ intersecting $\tau$ and $\eta$ in points $x_n$ and $y_n$ and so that $x_n \to c_1$ and $y_n\to c_2$. The segments $d_n$ have length going to infinity. 

Assume first that there are segments $d_n$ and $d_m$ with opposite orientation or which cut $\tau$ or $\eta$ in different connected components of $\tau \setminus c_1$ or $\eta \setminus c_2$. We get a Jordan curve $\cJ$ made by $d_n \cup d_m$ and two segments of $\tau$ and $\eta$ joining the points $x_n$ and $x_m$ and $y_n$ and $y_m$. Assume without loss of generality that $x_n$ is closer to $c_1$ than $x_m$. Consider the ray $r_m$ of $e_m$ starting from $x_m$ which does not contain $d_m$. By construction, $r_m$ enters the interior of the curve $\cJ$ and since it cannot intersect $d_n$ or $d_m$ it must exit $J$ intersecting either $\tau$ or $\eta$ in some point different from $x_m$ or $y_m$ which is a contradiction with Corollary \ref{coro.foliations} (in this case, it is just Poincar\'e-Bendixon's theorem that is being used since we are arguing inside $L$). 

Now, the same argument implies that between two segments $d_n$ and $d_m$ every leaf intersecting the transversal $\tau$ must exit through $\eta$ thus completing the proof of the lemma by choosing convenient parametrizations $\tau_1$ and $\tau_2$ of $\tau$ and $\eta$ respectively. 
\end{proof}

Define a \emph{non-separated ray class} of a leaf $c \in \cG_L$ to be one
which is non separated from some other leaf in $\cG_L$ as
in the previous Lemma. Now we can prove:

\begin{prop}\label{prop-nonseplands}
Let $c_1, c_2 \in \cG_L$ be non separated leaves of $\cG_L$ so that $c_1 \cup c_2 \subset L \cap E$ with $L \in \wt{\cF_1}$ and $E \in \wt{\cF_2}$. Then, if $c_1^+$ is the non-separated ray class of $c_1$ with $c_2$ then one has that the positive ray lands and $\partial^+ \Phi (c_1) = \alpha(E)$. 
\end{prop}

For the following proof, a crucial fact\footnote{Such a property is true for general foliations homeomophic to weak stable foliations of Anosov flows, but not true for general minimal foliations.} will be that if $\xi \neq \alpha(L)$ and $\eps>0$ is small, then there is a neighborhood $I$ of $L$ in the leaf space and a neighborhood $J$ of $\xi$ in $S^1(L)$ so that in $L \cup S^1(L)$ the set $\wihatD_\eps(L,I)$ contains $N \cap L$, where
$N$ is a neighborhood of $J$ in $L \cup S^1(L)$ (see Proposition \ref{prop.setDfol}). 

Note that if we assume that $\alpha(E) \notin \partial^+ \Phi(c_1)$ then it is easier to obtain a contradiction, since we can fix a transversal $\tau$ to $c_1^+$ suficiently close to some $\xi \in \partial^+ \Phi(c_1)$ so that the full ray from this point is contained in $\wihatD_\eps(E,I)$ and thus, we can show that if $E'$ is a leaf nearby to $E$ intersecting $\tau$ in the direction of $c_2$, then the $E' \cap L$ must `follow' closely the curve $c_1$ all along the ray, but on the other hand, since $c_1$ and $c_2$ are non-separated the curve must joint a point close to $\tau$ to some point in a transversal to $c_2$, and since leaves of $\wcF_2$ cannot intersect a transversal twice, one gets a contradiction. The main dificulty in what follows is that $c_1^+$ could approach $\alpha(E)$ and then we need to argue more carefully to get control on the nearby intersection.

\begin{proof} %[Proof of Proposition \ref{prop-nonseplands}]
Assume that $\partial^+ \Phi(c_1) \neq \{\alpha(E)\}$, this means that it contains some point $\xi \neq \alpha(E)$ in $\partial \HH^2$. 

Fix a sequence $x_k \in c_1^+$ so that $\Phi_L(x_k) \to \xi$ which also implies that $\Phi_E(x_k) \to \xi$ 
by Lemma \ref{lema.proyection}. We can assume by taking some subsequence that for some small $\eps$ smaller than the one given by Proposition \ref{prop-pushing} and some small intervals of leaves $I$ of $\wt{\cF_2}$ containing $E$ in the boundary we have that $x_k \in \wihatD_\eps(E,I)$.  

We fix a sequence of transversals $\tau_k: [0,\eps)\to L$ 
to $\cG_L$ in $L$,
through $c_1$ with $\tau_k(0)=x_k$ of length $\eps$ for $\eps$ as above. 
We assume that the transversals intersect the component
of $L \setminus c_1$ which contains $c_2$.
We also fix a transversal $\eta: [0,\eps) \to L$ so that $\eta(0) \in c_2$ as in Lemma \ref{lema.nonseparated}. 

Since $\xi$ is a marker point for $\wt{\cF_2}$ (i.e. $\xi \neq \alpha(E)$) by our choice of $I$ we know that if a leaf $E' \in I$, it intersects $\tau_1([0,\delta))$ for some small $\delta$, and by Proposition \ref{prop-pushing} it contains a curve  which remains at distance less than $\eps/2$ from $E$ and which limits in $\xi$
(in other words, whose image under $\Phi_{E'}$ limits
in $\xi$). Note that this curve has nothing to do with the ray $c_1^+$. 

In particular, we know that $E'$ will intersect $\tau_k$ for all sufficiently large $k$ (in fact, Proposition \ref{prop-pushing} also implies that $E'$ will intersect $\tau_k$ in points arbitrarily close to $x_k$). Note however that it could be that the intersection of $E'$ with $\tau_k$ happens in a different connected component of $E' \cap L$ since we do not know that $c_1^+$ stays far from $\alpha(E)$. Our goal in what follows is to show that this 'splitting' does not take place. 

Since $c_1^+$ is non separated from $c_2$ it follows from Lemma \ref{lema.nonseparated} that, up to reducing $\delta$, for every $t \in (0,\delta]$ we have that the leaf $e_t$ of $\cG_L$ through $\tau_1(t)$ intersects $\eta$ in some point defining a segment $d_t$ which joins both transversals. 

We claim that for every large $k$, every point in $\tau_k((0,\eps))$ must belong to some $d_t$: to see this, first take the segment $d_{t_1}$ and consider $k$ large so that $\tau_k((0,\eps))$ is completely contained in the region 
of $L$ bounded by $c_1^+$, the arc of the transversals $\tau_1$ joining $\tau_1(0)$ with $\tau_1(t_1)$, $d_{t_1}$, the arc of $\eta$ joining $\eta(0)$ with the intersection point with $d_{t_1}$ with $\eta$ and the corresponding ray of $c_2$ (this bounds a region by Proposition \ref{prop.plane} and by the choice of the transversals $\tau_k$ their image is contained in this region). Now, fix some point $\tau_k(s)$ with $s \in (0,\eps)$. By continuity of $\cG_L$ there is $t \in (0,\delta)$ such that $e_t$, the curve of $\cG_L$  through $\tau_1(t)$ intersects $\tau_k$ in some $\tau_k(s')$ with $s' < s$. Now consider the
region $\cJ$ bounded by the Jordan curve formed by $d_t, d_{t_1}$ and the arcs of $\eta$ and $\tau_1$ joining their endpoints. Then we get that the curve $e \in \cG_L$ passing through $\tau_m(s)$ must escape $\cJ$ intersecting once the transversal $\tau_1$ and once the transversal $\eta$, and the intersection point with $\tau_1$ is in a point $\tau_1(t')$ with $t < t' < t_1$, which shows the claim. 

Note that for different values of $t$ the segments $d_t$ belong to distinct leaves of $\wt{\cF_2}$ since each leaf intersects a transversal at most once (see Corollary \ref{coro.foliations}). On the other hand, given some $t$, since $d_t$ is compact and the points $x_k$ escape to infinity, it follows that $d_t$ cannot intersect $\tau_k$ for very large $m$ because $c_1$ is
properly embedded.
This  contradicts the fact that the leaf $E' \in \wt{\cF_2}$ through 
some point $\tau_1(t)$ described before
must intersect $\tau_k$ for all $k$. This contradiction then implies that 
$c_1$ cannot limit on $\xi \not = \alpha(E)$, or in other words
$\partial^+ \Phi(c_1) =\{ \alpha(E)\}$. 
\end{proof} 

\begin{remark}
We note that if $L \in \wt{\cF_1}$ and two curves $c_1,c_2 \in \cG_L$ are non-separated, then, as proved in Lemma \ref{lem.nonHsdfftwointersect}, they belong to $L \cap E$ with some $E \in \wt{\cF_2}$. However, it should be noted that it does not follow from Lemma \ref{lem.nonHsdfftwointersect} that $c_1$ and $c_2$ must be non-separated in the leaf space of $\cG_E$ (only that there is no transversal in $E$ from $c_1$ to $c_2$, recall Remark \ref{rem.nonseparated}) and so it does not follow that $\alpha(L)=\alpha(E)$.  See \S~\ref{ss.MatsumotoTsuboi} for concrete examples. 

\end{remark}

%%%%%%%%%%%%%%%%%%%%
\section{Landing of rays}\label{s.Landing}

In this section we will show: 

\begin{teo}\label{prop.landing}
Let $\cF_1,\cF_2$ be two transverse foliations on $M=T^1S$ and let $\cG$ denote the foliation obtained by intersection. Then, every ray of a leaf of $\wt{\cG}$ lands (cf. Definition \ref{landdefinition}).  
\end{teo}

The strategy of the proof is as follows. We will first show that if enough rays land, then every ray should land. In Proposition \ref{prop-nonseplands} we have shown that if rays are non-separated then they should land in the non-marker point of one of the foliations. We will use this to reduce to the case where leaves have Hausdorff leaf space where landing can also be shown rather directly by `pushing' to close-by leaves.

\subsection{Some rays land then every ray lands}

We will first prove the following: %% LEMA 3.5 in T1S

\begin{prop}\label{prop-onelandsallland} 
Assume that there is a leaf $c \in \cG_L$ for some $L \in \wt{\cF_i}$ ($i=1$ or $2$) which has one ray which lands in a point which is not $\alpha(L)$ (i.e. $\partial^+\Phi(c)= \{\xi\} \neq \{\alpha(L)\}$ or $\partial^- \Phi(c)= \{\xi\} \neq \{ \alpha(L)\})$. Then, every ray of any leaf of $\wt{\cG}$ lands in both foliations. 
\end{prop}

We will use minimality of the foliation to show that there is a dense set of points in $S^1(L)$ which are landing points of leaves of $\cG_L$. 

\begin{lema}\label{lema.landmarkerthendense}
Under the assumptions of Proposition \ref{prop-onelandsallland}, 
for every leaf $E \in \wcF_i$ and 
for every $I \subset \partial \HH^2$ non trivial interval,
 there is some leaf $e \in \cG_E$ so that either $\partial^+\Phi(e)= \{\eta\}$ or $\partial^-\Phi(e) = \{\eta\}$ for some $\eta \in I$. 
\end{lema}

\begin{proof}
Let us work in $\wihatM$, else, we can apply some deck transformations in the center of $\pi_1(M)$ so that everything makes sense. 
Let $\Gamma = \pi_1(M)/{\mathbb{Z}}$ where $\mathbb{Z}$ is the center,
notice that $\Gamma \cong \pi_1(S)$ (cf. \S~\ref{ss.unittangentbundles}).

Up to considering a subinterval of $I$ we can assume that the closure of $I$ is disjoint from $\alpha(E)$.  Consider a  deck transformation $\gamma \in \Gamma$ of
$\wihatM \to M$
such that $\gamma \xi \in I$ (see Proposition \ref{prop.minimality}). 
Again, up to reducing $I$ we can assume that $\gamma \alpha(L) \notin I$.  

Fix some open interval $J$, such that the closures of $J, I$ are 
disjoint, and in addition such that $\alpha(E),
\gamma \alpha(L)$ are 
in the interior of $J$.

Fix $\eps>0$ given by Proposition \ref{prop-pushing}. 
There is a subray $r \subset c$ landing in $\xi$ such that the projection of $\gamma r$ is completely contained in $\wihatD_\eps(E, J)$.
This is  because $\gamma r$ lands in $\gamma \xi \in I \subset J^c$ (cf. equation \eqref{eq:setDfol}) and $\alpha(E), \gamma \alpha(L)$
are in $J$.
Here $\gamma r \subset \gamma L$ and the pushed through 
ray is in $E$ and has ideal  point in $\gamma \xi \in I$. 
This finishes the proof.
\end{proof}

\begin{proof}[Proof of Proposition \ref{prop-onelandsallland}] 
Let $E \in \wcF_i$.
Suppose that some ray $c$ in $\cG_E$ does not land, and let $I$
be its limit set, which is then not a point. 

Let $J$ be a closed non-degenerate interval whose closure is strictly contained in the
interior of $I$ (possibly $I = \partial \HH^2$).

By Lemma \ref{lema.landmarkerthendense} one can find in $E$ at least two rays $e_1, e_2$ of $\cG_E$  landing in different points $\xi_1 \neq \xi_2 $ in the interior of $J$. We can further assume that $e_1, e_2$ are disjoint.
Fix a curve joining the starting points of $e_1$ and $e_2$ and this defines two regions $D^+$ and $D^-$ whose limit sets in $\partial \HH^2$ are the two closed intervals $[\xi_1,\xi_2]$ and $[\xi_2,\xi_1]$ in $\partial \HH^2$ with respect to the circular order (cf. Corollary \ref{coro.separationH2}). Since every ray  of $\cG_E$ is properly embedded and rays are disjoint or contained in one another, it follows that up to removing a compact piece, then the
ray $c$  must belong to either $D^+$ or $D^-$, and so
can only accumulate on $[\xi_1,\xi_2]$ or $[\xi_2,\xi_1]$. 
This contradicts the fact that $c$ accumulates on
all of $I$, and $I$ is an interval intersecting the interior
of both $[\xi_1,\xi_2]$  and $[\xi_2,\xi_1]$. 
This proves the Proposition.
\end{proof}

We also prove the following stronger statement that will be used later:

\begin{lema}\label{lema-alphainlimit}
Let $e \subset L \cap E$ so that $\partial^+ \Phi(e)=I$ where $I$ is a not a point
(in other words the positive ray of $e$ does not
land). Then $\alpha(L) \in I$. In addition 
if $I \neq \partial \HH^2$, then
for every $\xi$ in the interior of $I$ there is no ray of a leaf of $\cG_E$ which lands in $\xi$. The same holds for $\partial^-\Phi(e)$. 
\end{lema}

\begin{proof}
In this lemma we also work in $\wihatM$, but arguments can be easily adapted to do the same in $\mt$. 
All leaves in this lemma are in the same foliation.
First notice that if $I = \partial \HH^2$ then there is nothing
to prove.

First assume that $\alpha(L) \notin I$. Consider $\gamma \in \pi_1(M)$ so that $\gamma$ acting on $\partial \HH^2$ has a repelling fixed point $\gamma^- \notin I$ (and close to $\alpha(L)$) and an attracting fixed point $\gamma^+ \in I$ (cf. Proposition \ref{prop.minimality}).

Now, consider the leaf $E'$ associated to the repelling fixed point, that is, such that $\alpha(E')= \gamma^-$. It follows that this leaf is invariant under $\gamma$. Consider a ray $r$ of $e$, such that $\partial \Phi(r) = I$ which is contained in $\wihatD_\eps(E, J)$ where $J$ is some interval of the leaf space containing $E'$ and $E$ and $\eps>0$ given by Proposition \ref{prop-pushing}. Using Proposition \ref{prop-pushing} we deduce that there is a ray $r'$ of a leaf of $\cG_{E'}$ so that
$\partial \Phi(r) = I$. Moreover, $\gamma r'$ is a ray in $\cG_{E'}$ with
$\partial \Phi(\gamma(r)) = \gamma(I)$, which is strictly contained in $I$. This contradicts Proposition \ref{prop.disjointrays}. 

The second statement follows immediately from Proposition \ref{prop.disjointrays}.
\end{proof}

\subsection{Finding landing rays} 

In this section we will show the following: 

\begin{prop}\label{prop.someraylands}
Let $L \in \wt{\cF_1}$ and $\gamma \in \pi_1(M) \setminus \{\mathrm{id}\}$ such that $\gamma L = L$. Then, there is some $c \in \cG_L$ which has at least one ray which lands. 
\end{prop}

We consider a closed curve in the annulus $A = L/_{\langle \gamma \rangle}$ not homotopically trivial (for instance take a closed geodesic in $A$ with some metric in $L$ induced by $M$) and denote by $\ell$ a lift to $L$ of this closed curve, which is thus $\gamma$ invariant, that is $\gamma \ell =\ell$. For $K>0$ we denote as $B_K(\ell)$ the closed tubular neighborhood of radius $K$ around $\ell$ for some $\gamma$-invariant metric on $L$; note that the quotient $B_K(\ell)/_{\langle \gamma \rangle}$ is compact. 

We first show: 

\begin{lema}\label{lema-escape1}
Assume that $\gamma$ does not fix any leaf of $\cG_L$. For every $K>0$ there is $K'>0$ so that if $c \in \cG_L$ then any connected component of $c \cap B_K(\ell)$ has length less than $K'$. 
\end{lema}

\begin{proof}
Assume by contradiction that there are sequences of arcs $d_n$ of leaves $c_n \in \cG_L$ of length going to infinity such that $d_n \subset B_K(\ell)$. There is a sequence $\gamma^{k_n}$ of  iterates of $\gamma$ with $k_n \in \ZZ$ so that $\gamma^{k_n} d_n$ has its midpoint in a compact fundamental domain of $B_K(\ell)$ which is $\gamma$-invariant by choice. It follows that in the limit, $\gamma^{k_n}d_n$ converges to at least one leaf $c_\infty \in \cG_L$ which is completely contained in $B_K(\ell)$. 

Let $\pi_\gamma: L \to L/_{\langle \gamma \rangle}$ be the projection.
Then $\pi_\gamma(c_\infty)$ is contained in a compact annulus.
It follows that each ray of $c_\infty$ either projects to
a closed curve or to a curve asymptotic to a closed curve. 
This is a contradiction to hypothesis.
\end{proof}

%Up to considering $\gamma^{-1}$ we can assume that $c_\infty$ has one landing point in the repeller $\gamma^-$ of $\gamma$. Consider $\gamma^n c_\infty$ with $n \to \infty$ and a leaf $c'_\infty$ in the limit which has both $\gamma^+$ and $\gamma^-$ as landing points (such a leaf exists because $\gamma^n c_\infty$ is always contained in $B_K(\ell)$). 
%
%We claim that $c'_\infty$ (it could be that $c'_\infty=c_\infty$) is fixed by $\gamma$, contradicting our assumption. To prove the claim, assume $\gamma c'_\infty \neq c'_\infty$ and consider the region $\cJ$ delimited by their union (c.f. Proposition \ref{prop.plane}).  It follows that the interiors of $\gamma^k \cJ$ and $\gamma^m \cJ$ are disjoint if $k \neq m$. Since the set of accumulation points of $\gamma^n c_\infty$ is $\gamma$-invariant, it follows that $\gamma^{n_0}c $ belongs to the interior $\cJ$ for some $n_0$ but this then implies that $\gamma^n c_\infty $ cannot accumulate in $c'_\infty$.  \marg{R: I know this proof is way simpler using axes. It is somewhat being used. But I tried to be more 'elementary'.} 

Note that if there is a leaf $c \in \cG_L$ which is fixed by $\gamma$ then immediately we have that the landing points of $c$ are $\partial^{\pm}\Phi_L(c) = \{ \gamma^+, \gamma^-\}$ the attractor and repeller of $\gamma$ acting at infinity. %So the previous lemma says that if no leaf lands, then we must have that every ray is eventually disjoint from $\ell$. 
%\marg{S: Question, I did not get the last statement. What if  all rays in $L$ have limit set all of $S^1(L)$?}

Now we show that under the assumptions of the previous lemma, rays must keep intersecting $B_K(\ell)$ indefinitely. 

\begin{lema}\label{lema-returns}
Under the assumptions of Lemma \ref{lema-escape1} we have that there is $K>0$ such that if $c \in \cG_L$ and $r$ is a ray of $c$ that
does not land, then $r$ must intersect $B_K(\ell)$.  
\end{lema}

We stress that $K$ is independent of $c$ and $r$. Note that since $r$ cannot be contained in $B_K(\ell)$ by the previous lemma, we deduce that every ray needs to enter and leave $B_K(\ell)$ 
infinitely many times.

\begin{proof} 
Fix $\eps>0$ given by Proposition \ref{prop-pushing}. Now fix some large $K>0$ so that if $I$ is a closed interval of the leaf space of $\wt{\cF_1}$ containing $L$ as an endpoint, then $\wihatD_\eps(L,I)$ contains the connected component of the complement of $B_K(\ell)$ which does not accumulate in $I$ (see Proposition \ref{prop.setDfol}). Note that this value of $K$ is independent of the interval $I$ as long as $L$ is an endpoint and $I$ is sufficiently small (so that it is contained in the closure of one of the components of the complement of $\ell$). 

Assume by contradiction that $r$ is completely contained in the complement of $B_K(\ell)$ which then must be contained in a unique connected component because $r$ is connected. Therefore, the limit set $I=\partial \Phi_L(r)$ is a proper interval contained in $\partial \HH^2$ which does not intersect the interior of the limit set  of the image by $\Phi_L$ of the connected component of $L \setminus B_K(\ell)$ which does not contain $r$. Now pick an interval $J$ of the leaf space with $L$ as an endpoint so that $\alpha(E) \in J$ only if $E=L$. For any $E \neq L$ in $I$ we can then push the ray $r$ to $E$ applying Proposition \ref{prop-pushing} and we deduce that there is a ray $r'$ in $E$ which lands in an interval which is disjoint from $\alpha(E)$ contradicting Lemma \ref{lema-alphainlimit}. 
\end{proof}

Now we can proceed with the proof of Proposition \ref{prop.someraylands}:

\begin{proof}[Proof of Proposition \ref{prop.someraylands}.]
As noted, if $\gamma$ fixes some leaf of $\cG_L$ then this leaf must land in the attractor and repeller of $\gamma$ so the proposition holds. Therefore, we can assume we are under the assumptions of Lemma \ref{lema-escape1}. 

Let $c$ be a leaf of $\cG_L$ and 
$r$ a ray of $c$. If $r$ lands, the Proposition is proved,
so assume that $r$ does not land.  Then we know that $r$ must keep intersecting $B_K(\ell)$ for $K$ as in Lemma \ref{lema-returns}. Consider a sequence $d_n$ of segments of $r$ completely contained in the complement of $B_K(\ell)$ except for the endpoints which belong to $\partial B_K(\ell)$. Using Lemma \ref{lema-escape1} we can choose the segments $d_n$ with arbitrarily large length, since otherwise we would get arbitrarily large segments of $r$ contained in $B_{K'}(\ell)$ for some larger $K'$. 

Let $x_n, y_n$ be the endpoints of $d_n$. We claim first that $d(x_n,y_n)$ must be bounded: if not, then up to considering an inverse we can assume that $\gamma x_n$ is contained in the segment $J$ of $\partial B_K(\ell)$ joining $x_n$ and $y_n$. Since $d_n \cup J$ is a Jordan curve and $\gamma y_n$ does not belong to $J$ we deduce that $\gamma d_n$ intersects $d_n$ which is not possible since they belong to different leaves of $\cG_L$.

There is a sequence $k_n \in \ZZ$ so that $\gamma^{k_n} x_n$ belongs to a compact set. 
Thus, up to taking subsequences both $\gamma^{k_n}x_n$ and $\gamma^{k_n}y_n$ converge to points $x_\infty$ and $y_\infty$ in $\partial B_K(\ell)$. Note that $x_\infty$ and $y_\infty$ cannot be in the same leaf of $\cG_L$ and being accumulated by $\gamma^{k_n}d_n$ their leaves must be non-separated. Therefore the proposition follows from Proposition \ref{prop-nonseplands}. 
\end{proof}

\subsection{Proof of Theorem \ref{prop.landing}} 

To prove Theorem \ref{prop.landing} we want to apply Proposition \ref{prop-onelandsallland}. For this, it will be useful to use both foliations. We first show that if there is a leaf of $\cG_L$ which has both rays landing in the same point, then after going to a closeby leaf of the other foliation, we will be able to apply Proposition \ref{prop-onelandsallland} and we will get landing for all leaves in both foliations:

\begin{lema}\label{lema-bubblecenter}
Let $L \in \wt{\cF_1}$ such that there is $c \in \cG_L$ which verifies that both rays land in the same point, i.e. $\partial^+ \Phi_L^1(c) = \partial^- \Phi_L^1(c)= \{\xi\}$. Then, there is some leaf $E \in \wt{\cF_2}$  with $\alpha(E) \neq \xi$ such that it contains a leaf $c' \in \cG_E$ so that $\partial^+ \Phi_E^1(c') = \partial^- \Phi_E^1(c')= \{\xi\}$. 
\end{lema}

%This lemma applies switching the foliations too as all arguments are symmetric. Also, we note that we will get that $c'$ has both rays also limiting in $\xi$, so what we will show is that we can choose $E$ so that $\alpha(E) \neq \xi$. 
\begin{proof}
Consider the Jordan curve obtained by $c \cup \{\xi\}$ in $L \cup S^1(L)$ and denote by $D$ the disk bounded by it. 

Every curve $\hat c \in \cG_L$ intersecting $D$ must be contained in $D$ and thus has both rays landing on $\xi$. Given a transversal $\tau: [0,t_0) \to L$ to $\cG_L$ with $\tau(0) \in c$ and $\tau(t) \in D$ for $t>0$ we denote $E_t \in \wt{\cF_2}$ to the leaf through $\tau(t)$. 

If we denote by $c_t$ the curve of $\cG_L$ containing $\tau(t)$ we get that $c_t \subset L \cap E_t$. 
Since in $L$ both rays of $c_t$ land and the landing point is $\xi$,
we get by Corollary \ref{coincide} that for every small $t$,
then $c_t$ lands in $E_t$ and the landing point is $\xi$. Since the non-marker point $\alpha(E_t)$ varies monotonically we can choose $t$ so that $\alpha(E_t) \neq \xi$. 
\end{proof}

The argument can be slightly extended to get: 

\begin{lema}\label{lema-bubblezigzag} 
Let $c \in \cG_L$ which verifies $\partial^+ \Phi_L(c) = \alpha(L)$ and $\partial^- \Phi_L (c) \neq \partial \HH^2$. Then all rays of $\wcG$ land.
\end{lema}

\begin{proof}
Let $I = \partial^- \Phi_L(c)$. 
Suppose first that $I = \alpha(L)$,
 then both rays of $c$ land in the
same point $\alpha(L)$. 
The previous lemma implies that
there is some ray of $\wcG$ in some leaf $E$ of $\wcF_2$
landing in a point different from $\alpha(E)$. Then 
Proposition \ref{prop-onelandsallland} implies that all
rays of $\wcG$ land and the lemma is proved. 

If $I$ is a single point $\xi \not = \alpha(L)$, then
again Proposition \ref{prop-onelandsallland} also
implies that all rays land, and 
the lemma is also proved.

\vskip .05in
Finally from now on assume that
$I$ is not a single point.
Since $I \not = \partial \HH^2$, 
Lemma \ref{lema-alphainlimit} implies that $\alpha(L) \in I$.

The curve $c$ separates $L$ in two connected components, 
each  diffeomorphic to a disk (see Proposition \ref{prop.plane}). Denote these components $D^+$ and $D^-$. It follows that up to relabeling we can assume that the closure of $\Phi_L(D^+)$ in $\HH^2 \cup \partial \HH^2$ is equal to $\Phi_L(c \cup D^+) \cup I$ while the closure of $\Phi_L(D^-)$ in $\HH^2 \cup \partial \HH^2$ is equal to $\Phi_L(c \cup D^-) \cup \partial \HH^2$. See figure~\ref{fig-bubblenoland}.

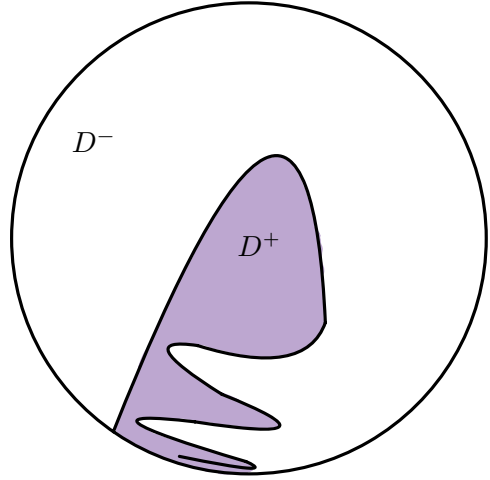
\begin{figure}[ht]
\begin{center}
\includegraphics[scale=0.88]{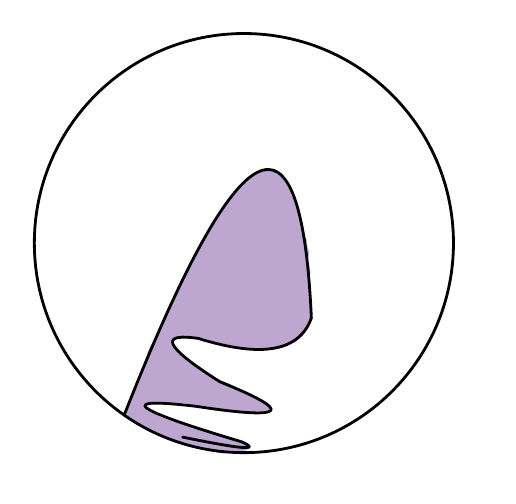}
\begin{picture}(0,0)
\put(-185,143){$D^-$}
%\put(-130,153){$p$}
\put(-123,104){$D^+$}
%\put(-160,52){$A_1^u$}
%\put(-129,52){$A_2^s$}
%\put(-87,77){{\small $\cG^u(p)$}}
%\put(-196,91){{\small $\cG^s(p)$}}
%\put(-113,225){$A_1^s$}
%\put(-82,220){$A_2^u$}
\end{picture}
\end{center}
\vspace{-0.5cm}
\caption{{\small A depiction of the objects in the proof of Lemma~\ref{lema-bubblezigzag}.}}\label{fig-bubblenoland}
\end{figure}

Consider a transversal $\tau: (-\eps,\eps) \to L$ to $\cG_L$ with $\tau(0) \in c$. Denote by $c_t$ the leaf of $\cG_L$ through the point $\tau(t)$. We can assume that if $t>0$ then $\tau(t) \in D^+$, thus, the limit set $\partial^{\pm}\Phi_L(c_t)=I_t$ of $c_t$ must be contained in $I$.
Again as in the beginning
of the proof we either finish the proof of the lemma
or we are in the case that, the interval $I_t$ is non-degenerate for 
any $t > 0$.
 Hence $I_t$ contains $\alpha(L)$ by Lemma \ref{lema-alphainlimit}. 
In addition $I_t$ is 
 weakly monotonically decreasing, meaning that if $t>t'$ then $I_t \subset I_{t'}$ (they could coincide). 

Denote by $E_t$ the leaf of $\wt{\cF_2}$ so that $c_t \subset L \cap E_t$. 
We first claim that $\alpha(E_0) = \alpha(L)$.
Otherwise again Proposition \ref{prop-onelandsallland} would imply that
all rays of $\wcG$ land in their appropriate leaves, but at
this point we are assuming that the negative ray of $c$ does not
land.
This contradiction shows that $\alpha(E_0) = \alpha(L)$.

In addition Lemma \ref{lema-alphainlimit}, applied to
$\cF_2$,  implies
that $\alpha(E_t)$ is contained in $I_t$. 
For small $t > 0$, $\alpha(E_t)$ is contained
in the interior of $I$. 
Notice that $\alpha(E_t)$ varies continuously with $t$.
We fix some small $t_0 > 0$.
%Suppose that both rays of $c_{t_0}$ land in $L$  and
%the landing point of both is $\alpha(L)$.
%Then $I_{t_0}$ reduces to $\alpha(L)$ and so $\alpha(E_{t_0}) = \alpha(L)$,
%a contradiction. Hence $I_{t_0}$ is a non degenerate interval.
Now fix  a $t_1$ with 
 $0<t_1<t_0$ and 
 so that $E_{t_1}$ is invariant under some deck transformation $\gamma$ which has a fixed point outside $I$ (such leaves are dense, so there must be one). The other fixed point of $\gamma$ is necessarily $\alpha(E_{t_1})$.
We can arrange that this point is
contained in the interior of $I_{t_0} \subset I_{t_1} \subset I$. 
This is because by choosing $t_1$ small we can ensure that $\alpha(E_{t_1})$
is in the interior of $I_{t_0}$, and since
$I_{t_0} \subset I_{t_1}$, the claim follows.

We get that $c_{t_1}$ has a ray converging to an interval $I_{t_1}$ which contains only one of the fixed points of $\gamma$. 
We proved before that $\alpha(L) = \alpha(E_0) \in I_t$ for any $t$.
In addition since $I_{t_0} \subset I_{t_1}$ we get that $\alpha(E_{t_0})$
is also in $I_{t_1}$, so the interval $[\alpha(E_0), \alpha(E_{t_0})]$ contained
in $I$ is contained in $I_{t_1}$. We conclude that $\alpha(E_{t_1})$ is
an interior point of $I_{t_1}$, and $I_{t_1}$ does not contain
the other fixed point of $\gamma$.
Applying $\gamma$ to $E_{t_1}$ we obtain two rays whose landing sets are proper intervals one contained in the interior of the other
($I_{t_1}$ and $\gamma(I_{t_1})$), contradicting  Proposition \ref{prop.disjointrays} and completing the proof.
\end{proof}

\begin{proof}[Proof of Theorem \ref{prop.landing}]
Using Proposition \ref{prop.someraylands} we know that there is at least one $L \in \wt{\cF_1}$ and one $c\in \cG_L$ so that it has one landing ray. Up to orientation we can assume that $\partial^+ \Phi_L(c) = \xi$. Using Proposition \ref{prop-onelandsallland}, either the result holds or 
$\xi = \alpha(L)$. Using Lemma \ref{lema-bubblezigzag} we can further assume that $\partial^-\Phi_L (c) = \partial \HH^2$ (it is the only possibility of that lemma that does
not yield that all rays land), in particular this assumption
means that there is a ray
of $\wcG$ that does not land.

Note that all these results apply to both $\wt{\cF_1}$ and $\wt{\cF_2}$ so we will now work with $\wt{\cF_2}$. Showing landing in both foliations is equivalent because of Lemma \ref{lema.proyection}. 

Let $\tau: (-\eps,\eps) \to L$ be a transversal to $\cG_L$ with $\tau(0) \in c$. Call $c_t$ the leaf of $\cG_L$ passing through $\tau(t)$ and $E_t \in \wt{\cF_2}$ so that $c_t \subset L \cap E_t$. Notice that the $\{ E_t \}$ are 
pairwise distinct leaves.

We know that $\alpha(E_0)=\alpha(L)$ because $\partial^+ \Phi_{E_0}(c) = \partial^+ \Phi_L(c) = \alpha(L)$ (else we can apply Proposition \ref{prop-onelandsallland} to $\wt{\cF_2}$). 
We choose the transversal $\tau$ small so that $\alpha(E_t)$ is
injective in $(-\eps, \eps)$ and $\alpha(E_t)$ is very close
to $\alpha(L)$. Let $I_t$ be the short closed interval in $\partial \HH^2$
from $\alpha(L)$ to $\alpha(E_t)$ and $J_t$ the closure of
the complementary interval in $\partial \HH^2$. Note that for $t>0$ since $c_t \in L \cap E_t$,  if it has a landing ray, then  Proposition \ref{prop-onelandsallland} applied to both foliations implies it must land $\alpha(L)$ and in $\alpha(E_t)$ which are different, so the limit set of both rays is a non-trivial connected set.  Lemma \ref{lema-alphainlimit} applied to both foliations implies that the limit set of each ray of $c_t$ must contain both $\alpha(L)$ and $\alpha(E_t)$ thus it is some non-trivial connected set containing $I_t$ or $J_t$. 

For any $E_t$ which is a deck translate of $E_0$ there is a ray of $\cG_{E_t}$ which lands (i.e. the image of the landing ray of $c_0$).
By Proposition \ref{prop-onelandsallland} again, it follows
that this ray must land in $\alpha(E_t)$.
Let $r_t$ be one such ray. Consider $\cB= \{ t \in (-\eps, \eps) \ : \  \exists \gamma \in \pi_1(M) \ E_t =\gamma E_0 \}$. Note that by minimality $\cB$ is dense in $(-\eps,\eps)$. We will consider $\cB^+ = \cB \cap (0,\eps)$. 

Consider $t \in \cB$ and let $g_t$ be a ray of $c_t$. Denote by $I_g$ its limit set (that is
$\partial \Phi_{E_t}(g_t)$) in $\partial \HH^2$. Recall that $I_g$ must contain $I_t$ or $J_t$ (in particular, it not a point). The ray $g_t$ is disjoint from
$r_t$, because $r_t$ lands and no ray of $c_t$ lands. If $I_g \neq \partial \HH^2$, then applying Lemma \ref{lema-alphainlimit} (first to $L$ and $c_0^+$ and then to $E_t$ and $r_t$) we deduce that $I_g$ coincides with either $I_t$ or $J_t$ because there cannot be a ray landing in its interior. If $I_g= \partial \HH^2$ let us show a contradiction with a similar argument: in  $L$  the ray $g_t$ has to be disjoint from the ray of $c_0$
which lands in $\alpha(L)$, in particular there cannot be
a sequence of segments in $g_t$ which limit to 
an interval in $S^1(L)$ with 
$\alpha(L)$ in its interior. 
When seen in $E_t$ then $g_t$ has to be disjoint from $r_t$,
so again it cannot have a sequence of segments with 
a limit which is an interval in $S^1(E_t)$ with $\alpha(E_t)$
in its interior. It follows that $I_g$ could not be all of $\partial \HH^2$ and thus $I_g$ is either $I_t$ or $J_t$.

\begin{figure}[ht]
\begin{center}
\includegraphics[scale=0.98]{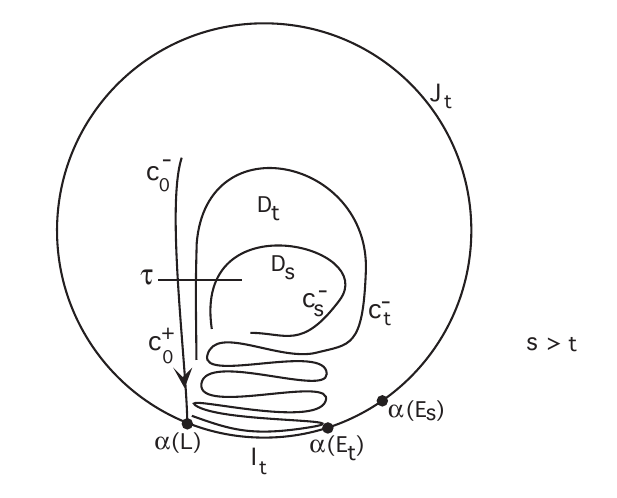}
%\begin{picture}(0,0)
%\put(-203,153){$o_1$}
%\put(-130,153){$p$}
%\put(-53,144){$o_2$}
%\put(-160,52){$A_1^u$}
%\put(-129,52){$A_2^s$}
%\put(-87,77){{\small $\cG^u(p)$}}
%\put(-196,91){{\small $\cG^s(p)$}}
%\put(-113,225){$A_1^s$}
%\put(-82,220){$A_2^u$}
%\end{picture}
\end{center}
\vspace{-0.5cm}
\caption{{\small A depiction of parts of the argument of the proof of Theorem \ref{prop.landing}. Inclusion of sets $D_s \subset D_t$ as well as the 'pairing' configuration for rays.}}\label{fig-t1}
\end{figure}

%Orient $c_t$ coherently. 
Recall that $c^+_0$ is a ray which
lands in $\alpha(L)$.
Let $D_t$ be the component of $L \setminus c_t$ which does not
contain $c_0 = c$, in particular $D_t$ is an open 
set in $L$ which has boundary (in $L$) equal to $c_t$.
 Notice that if $s > t > 0$, then $D_s \subset D_t$,
and in particular that $c_s$ is contained in $D_t$.
For each $t>0$ let $c^+_t, c^-_t$ be the subrays of $c_t$
determined by $\tau(t)$ and oriented coherently with $c_0$.

\begin{figure}[ht]
\begin{center}
\includegraphics[scale=0.98]{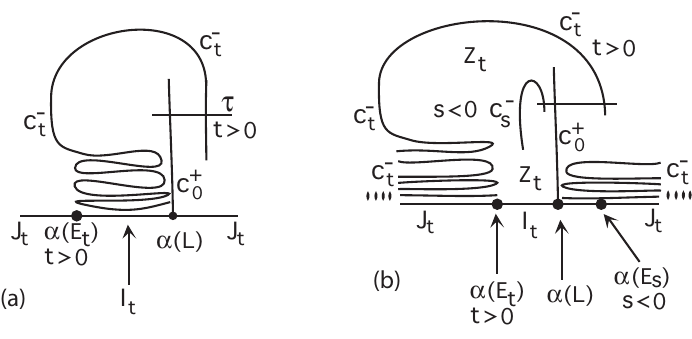}
\begin{picture}(0,0)
%\put(-203,153){$o_1$}
%\put(-130,153){$p$}
%\put(-53,144){$o_2$}
%\put(-160,52){$A_1^u$}
%\put(-129,52){$A_2^s$}
%\put(-87,77){{\small $\cG^u(p)$}}
%\put(-196,91){{\small $\cG^s(p)$}}
%\put(-113,225){$A_1^s$}
%\put(-82,220){$A_2^u$}
\end{picture}
\end{center}
\vspace{0.0cm}
\caption{{\small A depiction of the cases that $\tau(t), \alpha(t)$ cross.
In (a) we depict the case that the limit set of $c^-_t$ is $I_t$,
in (b) the case that the limit set is $J_t$.
For simplicity of viewing we depict $\partial \HH^2$ in 
a line segment, concentrating on what is happening near $\alpha(L)$.}}
\label{fig-t2}
\end{figure}

Consider $t \in \cB^+$.
We know that $\partial \Phi_L(c^-_t)$ can only be $I_t$ or
$J_t$ for each such $t$. There are also two possibilities depending on where
$\alpha(E_t)$ is (so in total, four possibilities). We say that $\alpha(E_t)$ and $\tau(t)$
\emph{cross} if we consider a ray $r$ from $\tau(0)$ to $\xi \in J_t$ avoiding $c_0^+ \cup \tau((0,t])$ and there for every curve in $L \cup S^1(L)$ joining $\tau(t)$ with $\alpha(E_t)$ the curve intersects $c_0^+ \cup r$. Note that this is independent on the choice of $r$, and it is also independent on $t>0$. We say they \emph{pair} otherwise (i.e. there is a curve joining $\tau(t)$ with $\alpha(E_t)$ and avoiding $c_0^+ \cup r$). 

%when for $t$ sufficiently small one can start
%from $\tau(t)$, then  follow a path in $L$ near $c_0$ without touching $c_0$, and then a path in $L$ very near $S^1(L)$ and the path
%limiting 
%to $\alpha(E_t)$ (and only to $\alpha(E_t)$) in $S^1(L)$
%so that the whole path does not intersect $c_0$. Otherwise we say that $\alpha(E_t)$
%and $\tau(t)$ \emph{cross}: in that case the path has to cross
%to the other side of $c_0$ to limit to $\alpha(E_t)$.

We suppose first that $\tau(t), \alpha(E_t)$ pair, 
see this in Figure \ref{fig-t1}.
Suppose first that for some  $t \in \cB^+$ we have that $\partial \Phi_L(c^-_t) 
= I_t$. In this case, note that the limit of $\Phi_L(D_t)$ is also $I_t$.  
For any $s > t$ we have $D_s \subset D_t$, so we now obtain
that $\partial \Phi_L(c_s) \subset I_t$. 
But for such $s$ the 
limit set of $c^-_s$ has to contain $\alpha(E_s)$ which is not contained in $I_t$ a contradiction. See figure \ref{fig-t1}.

Now assume that  $\tau(t), \alpha(E_t)$ pair but $\partial \Phi_L(c^-_t)
= J_t$. In particular $c^-_t$ has points which limit
to $\alpha(E_t)$ but $c^-_t$ does not limit in any
point in the interior of $I_t$. 
For any $s > t$ in $\cB^+$
we cannot have that $\partial \Phi_L(c^-_s)$ is 
$J_s$. This is because the ray $c^-_t$ limiting
to $\alpha(E_t)$ forces the ray $c^-_s$ to limit
to some point in the interior of $I_s$. Thus, we can apply the previous analysis to $s$ to get a contradiction. This finishes the analysis in the case where the transversal $\tau$ pairs. 
%Hence
%we obtain that $\partial \Phi_L(c^-_s) = I_s$.
%This situation was analyzed in the first possibility above.

We now consider the case where $\tau(t), \alpha(E_t)$ 
cross. 
Suppose first that for some $t \in \cB^+$ the ray $c^-_t$
limits to $I_t$. We depict this in Figure \ref{fig-t2}(a). Then for $s< 0$, it follows that the limit
sets of both $c^+_{s}, c^-_{s}$ have to be contained in 
$I_t$. This property is impossible, because for $s<0$ contained in $\cB$ the limit sets would have
to be either $I_{s}$ or $J_{s}$ and neither is contained
in $I_t$.

The remaining case is that for any $t \in \cB^+$ we have that $\tau(t), \alpha(E_t)$ cross and $\partial \Phi_L(c^-_t) = J_{t}$. 
Fix some $t\in \cB^+$. We will need to consider a different disk $Z_t$ in $L$ whose boundary is
made up of $c^-_t$, the arc  $\tau([0,t])$ and $c^+_0$ and chosen so that it contains $c_t^+$. Notice that this disk limits only on $J_t$ (see Figure
\ref{fig-t2}(b)). 
Now consider $s$ with $0 < s < t$, and $s \in \cB^+$. 
Then $c^+_s$ is contained in $Z_t$, so its
limit set has to be also contained in $J_t$. On the other hand
it has to be either $J_s$ or $I_s$ but neither is contained
in $J_t$, which is a contradiction.

This finishes the proof of Theorem \ref{prop.landing}.
\end{proof}

%%%%%%%%%%%%%%%%%%

%%%%%%%%%%%%%%%%%%%%
\section{Small visual measure}\label{s.smallvisual}

Let $\cF_1, \cF_2$ be two transverse minimal foliations on $M=T^1S$ and $\cG=\cF_1 \cap \cF_2$ the intersection foliation.

We define the \emph{shadow} of a subset $X \subset L$ seen from $x \in L$ as the set of points $\xi \in \partial \HH^2$ for which there is a geodesic ray in $\HH^2$ starting at $\Phi_L(x)$ and passing through some point $y \in \Phi_L(X)$ with ideal point $\xi$.

As in \cite[\S 4.3]{FP2} we will say that $\cG$ has the \emph{small visual measure} property if  for every $\eps>0$ there exists $R>0$ so that if $L \in \wt{\cF_i}$, $x \in L$ and $\ell$ a segment of some leaf of $\cG_L$ so that $\ell \cap B_R(x)=\emptyset$, then the shadow of $\ell$ seen from $x$ has length smaller than $\eps$ (meaning that the angle of the interval of vectors such that geodesic rays in $\HH^2$ from $\Phi_L(x)$ to the shadow of $\ell$ is less than $\eps$).

In this section we will show: 

\begin{prop}\label{prop.visual}
Let $\cF_1,\cF_2$ be two transverse minimal foliations on $M=T^1S$ and let $\cG$ denote the foliation obtained by intersection. Then, $\cG$ has the small visual measure property.  
\end{prop}

\begin{figure}[ht]
\begin{center}
\includegraphics[scale=0.78]{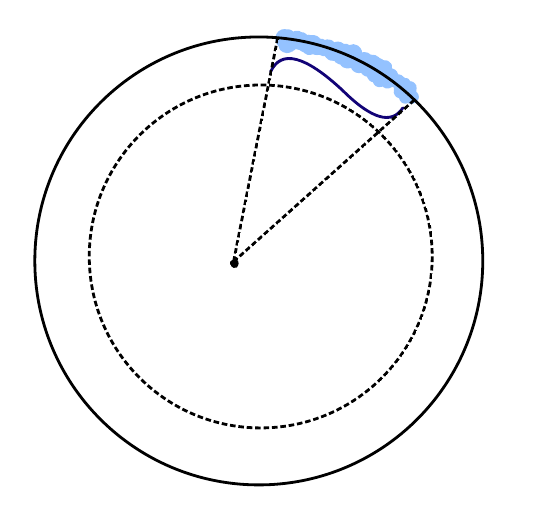}
%\begin{picture}(0,0)
%\put(-203,153){$o_1$}
%\put(-130,153){$p$}
%\put(-53,144){$o_2$}
%\put(-160,52){$A_1^u$}
%\put(-129,52){$A_2^s$}
%\put(-87,77){{\small $\cG^u(p)$}}
%\put(-196,91){{\small $\cG^s(p)$}}
%\put(-113,225){$A_1^s$}
%\put(-82,220){$A_2^u$}
%\end{picture}
\end{center}
\vspace{-0.5cm}
\caption{{\small The shadow of a set from a point $x \in L$ is the set of points in $S^1(L)$ blocked by this set (when seen in $\HH^2$). To have small visual measure means that if an arc of $\cG_L$ is very far from $x$ then its shadow has to be very small (independently of its length).}}\label{fig-shadow}
\end{figure}

Before we proceed with the proof let us explore some of its consequences. In \cite[Lemma 5.9]{FP2} the following is proved. 

\begin{prop}\label{prop.visualconsequence}
Assume that $\cG$ has the small visual measure property, then, there exists $a_0>0$ such that for every segment $\ell$ of leaf of $\cG_L$ we have that the geodesic segment joining the endpoints of $\ell$ is contained in $B_{a_0}(\ell)$.  
\end{prop}

%\begin{proof}
%Assume by contradiction that the conclusion does not hold, then, there are sequences of segments $\ell_n$ of $\cG_L$ so that the geodesic segment $g_n$ joining the endpoints of $\ell_n$ is not contained in $B_n(\ell_n)$. It follows that there is a point $x_n \in g_n$ so that $B_n(x_n)$ does not intersect $\ell_n$. This implies that seen from $x_n$ the visual measure of $\ell_n$, which is at distance larger than $n$ is at least $\pi$. This is in contradiction with the definition of the small visual measure property for $\cG_L$.  
%\end{proof}

We stress that $B_{a_0}(\ell)$ is the neighborhood of $\ell$ of size
$a_0$ in $L$, rather than in $\mt$.

Note that from this we get that small visual measure allows to control endpoints of rays of leaves of $\cG_L$.  That is, if one considers a point $x \in L$ which belongs to some leaf $c \in \cG_L$ and considers points $y_n$ in $c_x^+$ going to infinity. It follows that the geodesic segments joining $x$ and $y_n$ are all contained in the $a_0$ neighborhood of $c_x^+$, thus, it follows that the geodesic ray joining $x$ with $\partial^+\Phi(c)$ is contained in the $a_0$-neighborhood of $c_x^+$. 

%\begin{remark}
%Note that we only show one subset property: the geodesic segment is contained in a  uniform neighborhood of the arc of leaf of $\cG_L$. To show that leaves of $\cG_L$ are quasigeodesics in $L$, the main point is to obtain the 
%converse subset property, that is, that leaves of $\cG_L$ are 
%contained in uniform neighborhoods of geodesics connecting their
%endpoints.
%\end{remark}

\subsection{Bubble leaves}\label{ss.bubbleleafs} 

\begin{definition} (bubble leaves) \label{defi.bubble}
We say that $c \in \wt{\cG}$ is a \emph{bubble leaf} if $\partial^+ \Phi(c) = \partial^- \Phi(c) = \{\xi\}$ where $\xi$ is some point in $\partial \HH^2$. 
\end{definition}

We have dealt with such leaves in Lemma \ref{lema-bubblecenter} showing that nearby leaves of the other foliation must also have bubble leaves with the same endpoint.   We will perform a similar argument now, but trying to control the size of the interval on the leaf space where this holds. 

The goal of this subsection is to show the following:

\begin{prop}\label{prop.somenonbubble}
There exists a leaf $c \in \wt{\cG}$ such that $\partial^+\Phi(c) \neq \partial^-\Phi(c)$ (c.f. equation \ref{eq:indepfol}). 
\end{prop}

We will need to control the place where the bubble leaves land. For this, we will separate the leaf $L$ in \emph{bubble regions}. To introduce
this, let us first make some definitions. Notice first that by Corollary \ref{coro.separationH2} we know that if $c \in \cG_L$ is a bubble leaf with $\partial^{\pm}\Phi(c)=\{\xi\}$ we can define $D_c$ to be the disk in the complement of $c$ so that the accumulation of $\Phi_L(D_c)$ in $\partial \HH^2$ is exactly $\xi$. 
 
\begin{definition} (bubble region of $c$) \label{bubbleregion}
 Given $L \in \wt{\cF_i}$ and $c \in \cG_L$ a bubble leaf with $\partial^{\pm}\Phi(c)=\{\xi\}$ we denote by $\cB(c, L)$ the \emph{bubble region of $c$} in $L$ as the union of all leaves $c' \in \cG_L$ such that there is some $c'' \in \cG_L$ which is a bubble leaf and so that $D_c \cup D_{c'} \subset D_{c''}$. We call $\xi$ the \emph{landing point} of the bubble region. 
\end{definition}

We also remark that in general the bubble region of $c$ is not the 
union of all bubble leaves $c'$ in $\cG_{L_i}$ so that 
$\partial \Phi^{\pm}(c') = \{ \xi \}$. For example 
it could be that $c, c'$ bound maximal disjoint disks $D_c, D_{c'}$.
By maximal for $c$ we mean there is no leaf $c"$ of $\cG_{L_i}$ 
distinct from $c$, and such that
$D_{c"} \supset D_c$. In this case $D_c, D_{c'}$ are two disjoint 
bubble regions with the same ideal point $\xi$.

 \begin{lema}\label{lema-bubbleregions}
 Every bubble region is either open or closed. 
Each leaf contains at most countably many distinct bubble regions. 
If every leaf in $\cG_L$ is a bubble leaf, then there is a unique open bubble region $\cB$ which verifies that $\Phi_L(\cB)$ accumulates in all of $\partial \HH^2$. 
 \end{lema}
 
 In figure \ref{fig-bubbles} a depiction of the conclusion of this lemma is presented.

\begin{figure}[ht]
\begin{center}
\includegraphics[scale=0.48]{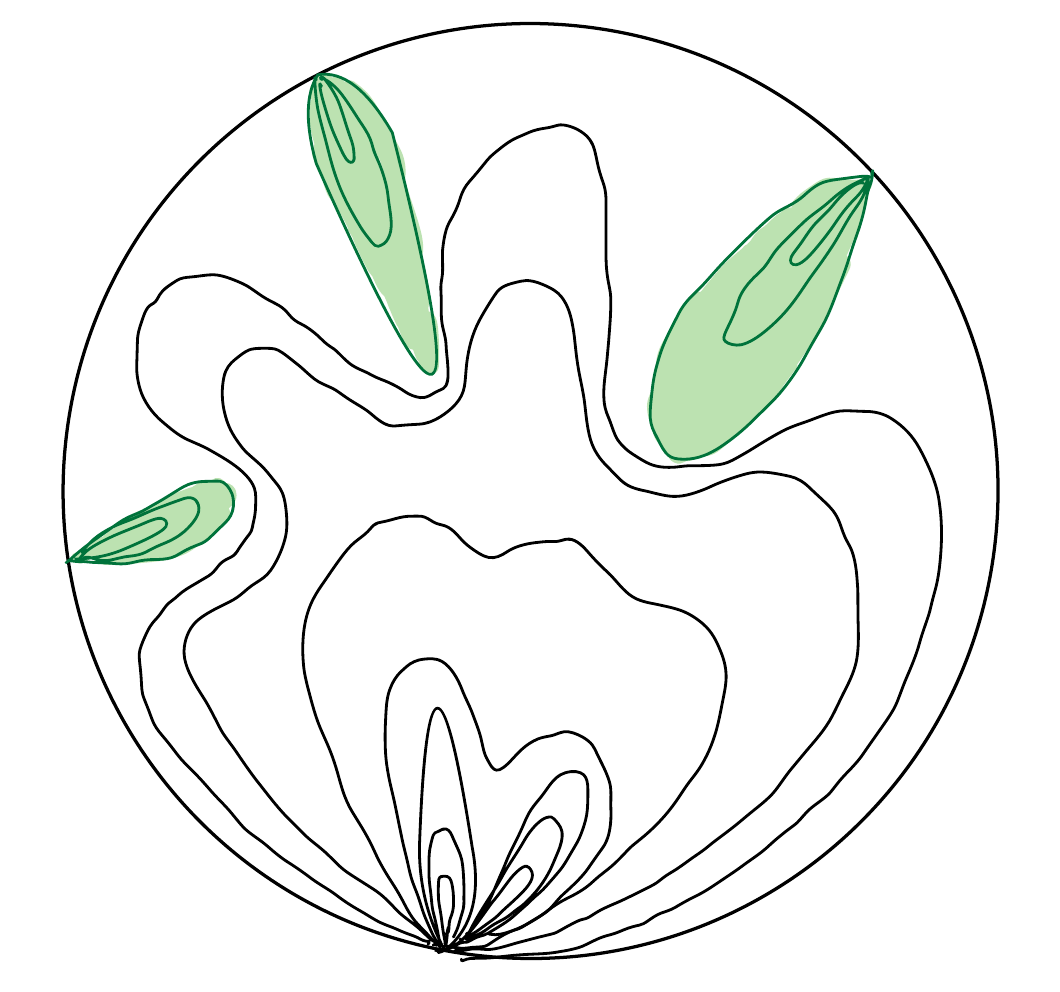}
%\begin{picture}(0,0)
%\put(-203,153){$o_1$}
%\put(-130,153){$p$}
%\put(-53,144){$o_2$}
%\put(-160,52){$A_1^u$}
%\put(-129,52){$A_2^s$}
%\put(-87,77){{\small $\cG^u(p)$}}
%\put(-196,91){{\small $\cG^s(p)$}}
%\put(-113,225){$A_1^s$}
%\put(-82,220){$A_2^u$}
%\end{picture}
\end{center}
\vspace{-0.5cm}
\caption{{\small Depiction of a leaf $L$ so that every leaf of $\cG_L$ is a bubble leaf. The painted regions are the closed bubble regions and their complement is the open bubble region.}}\label{fig-bubbles}
\end{figure}

 \begin{proof}
 Each bubble region is a set saturated by $\cG_L$.  There is a natural partial order between the leaves in a bubble region given by $c_1 \prec c_2$ if $D_{c_1} \subset D_{c_2}$.  It follows that if there is a maximal element in the bubble region, then it is unique (by definition of bubble region, every pair of elements is smaller or equal than some element in the region) and the bubble region is closed, while if there is no maximal element then the bubble region is open. As bubble regions have non-empty interior we know that there are at most countably many bubble regions. 
 
Assume now that every leaf of $\cG_L$ belongs to some bubble region. If there is a unique bubble region, then it is an open bubble region since there cannot be a maximal element. We next prove that not every bubble is closed: Let $B$ be a closed bubble with boundary leaf $c$ and let $\tau$ be
a small transversal to $\cG$ in $L$ with an endpoint in $c$
and not intersecting the interior of $B$.
Parametrize $\tau$ as $x_t, 0 \leq t \leq 1$ with $x_0$ in $c$.
Suppose a bubble region $B'$ intersects $\tau$ not in $x_0$. Then
its boundary intersects $\tau$ (as $x_0$ is not in $B'$). 
In addition its boundary intersects $\tau$ in a single
point, because if it intersected $\tau$ in two points, one
would produce a bubble leaf intersecting $\tau$ in two points,
contradiction. It follows that there is a single bubble region
$B'$ containing $\tau \setminus \{ x_0 \}$ and $B'$ is
an open bubble.
%In particular notice that $B'$ limits on every point
%in $\partial B'$.

%Since $L$ is path-connected, it cannot happen that all bubble regions are closed\marg{R: A closed interval cannot be written as union of countably many disjoint intervals. I wrote a proof which is commented, and a reference is W. Sierpinski, Un th\'eor\`{e}me sur les continus, T\^{o}hoku Math. J. 13 (1918), 300--305.}. 

Finally, given an open bubble region $\cB$ we can consider $\wihat\cB$ to be $\cB$ together with all the closed bubbles whose boundary intersects the boundary of $\cB$. It follows that $\wihat\cB$ is an open set and that if $\cB'$ is another open bubble region, then $\wihat\cB \cap \wihat\cB' = \emptyset$, thus, we get there is a unique open bubble region. 

To conclude, using Corollary \ref{coro.separationH2} we know that the complement of a closed bubble region verifies that it accumulates in all of $S^1(L)$, so the same holds for the complement of countably many closed bubble regions, thus the open bubble region accumulates everywhere, also using
the properties derived above. 
\end{proof}

%% PROOF THAT PATH CONNECTED SPACE CANNOT BE COUNTABLE UNION OF DISJOINT NON TRIVIAL CLOSED SETS: It is enough to show this for intervals. To see this, write C_1, \ldots, C_n, \ldots closed sets of the interval [0,1].  Denote by $D_n$ the complement of $C_1 \cup \ldots \cup C_n$ in $[0,1]$ which is an open non-empty set because connectedness (so cannot be finite union). It follows that D_n contains a sequence of nested intervals, thus, non-trivial intersection which cannot belong to any $C_n$. 

We remark that the uniqueness of open bubble regions in a leaf $L$ of $\wcF_i$
does not necessarily work if not every leaf in $\cG_L$ is a bubble
leaf.

Since open bubble regions are special, it is natural to expect that their landing point will also be special: 
 
 \begin{lema}\label{lema-openbubbleland}
 If every leaf in $\wt{\cG}$ is a bubble leaf then in each $L \in \wt{\cF_i}$ the unique open bubble region of $\cG_L$ verifies that its landing point is $\alpha(L)$. 
 \end{lema}
 
 \begin{proof}
 The proof is by contradiction. Assume that there is a leaf $L  \in \wt{\cF_i}$ where the unique open bubble region $\cO$ has landing point $\xi \neq \alpha(L)$. We fix a closed interval $I$ in the leaf space of $\wt{\cF_i}$ containing $L$ in its interior,
so that for every $L' \in I$ we have that $\alpha(L') \neq \xi$. Denote by $J \subset \partial \HH^2$ the interval $J=\{ \alpha(L') \ : \ L' \in I\}$. We consider $\eps>0$ from Proposition \ref{prop-pushing} and fix the region $\wihatD_\eps(L,I)$. 
 
 We first notice that Lemma \ref{lema.landmarkerthendense} shows that the existence of a leaf whose landing point is $\xi \neq \alpha(L)$ implies that in each leaf $L$ the set of landing points of rays in the leaf is dense in $\partial \HH^2$. Thus, for a given leaf $L$ there must be rays converging to this dense set of points and thus closed bubble regions in $L$ so that
the union of the landing points of these is a dense in $\partial \HH^2$. 
Pick one closed bubble region 
$\cB$ whose ideal point is neither $\alpha(L)$ nor $\xi$.

Up to shrinking $I$ to a smaller interval
we assume that $\cB$ completely contained in $\wihatD_\eps(L,I)$. Moreover, if we call $c$ the maximal element of $\cB$, it is a leaf of $\cG_L$ whose endpoints are $\partial^{\pm}\Phi(c) = \{\eta\}$ with $\eta \notin J$. We can assume that every closed bubbles leaf whose landing point is in the interval $K$  joining $\xi$ and $\eta$ and not intersecting $J$ is contained in $\wihatD_\eps(L,I)$. 
Indeed any such bubble region $\cB'$ not contained
in $\hat D_{\eps}(L,I)$ has points
limiting to points in the interval $K$ as well as points a bounded
distance from the geodesic $g$ in $L$ with ideal points the endpoints
of $I$. Suppose there are infinitely many of these. They cannot 
accumulate to some point in $L$, hence segments of the boundary of these limit
on a non degenerate open interval of $S^1(L)$. This contradicts
that the set of ideal points of leaves of $\cG_L$ is dense in
$S^1(L)$. It follows that there are only finitely many bubbles
in $L$ satisfying this property and we take the last one
satisfying the property.
  
 Since $\cO$ is an open bubble, we can find a sequence of nested curves $e_n$ in $\cO$ which accumulates on $c$. More precisely, if $D_{e_n}$ are the disks defined by $e_n$ and contained in $\cO$ it follows that $D_{e_n} \subset D_{e_{n+1}}$ and $\cO = \bigcup_n D_{e_n}$. It follows that there are rays $r_n$ of $e_n$ which are contained in $\wihatD_\eps(L,I)$ which also accumulate on $c$ (note that the full curve $e_n$ must enter $\wihatD_\eps(L,I)$, but it cannot happen that for large $n$ both rays are contained in
$\wihatD_\eps(L,I)$ since the union of the disks $D_{e_n}$ is all of $\cO$).

Now we fix a leaf $E \in I$ which is invariant under some $\gamma \in \pi_1(M) \setminus \{\mathrm{id}\}$ (cf. \S~\ref{s.minimality}) and so that the fixed points of $\gamma$ are different from $\xi$ (note that one must be contained in $J$). If we apply Proposition \ref{prop-pushing} we deduce that the leaf $c$ pushes to a leaf $c'$ of $\cG_E$ whose endpoints are $\eta$ and the rays $r_n$ push to rays $r_n'$ which have one endpoint in $\xi$ and accumulate on 
a ray of $c'$. If $e_n'$ are the leaves of $\cG_E$ containing the rays $r_n'$ we deduce that they must belong to an open bubble region since every leaf of $\cG_E$ is a bubble leaf and the leaves $e_n'$ are all landing in $\xi$. It follows that $E$ has an open bubble region with limit point $\xi$. Applying $\gamma$ we deduce that there is more than one open bubble region in $E$ contradicting the previous lemma.  
 \end{proof}
 
To prove Proposition \ref{prop.somenonbubble} we will proceed by contradiction and therefore assume that all leaves of $\wt{\cG}$ are bubble leaves. This will allow us to use the previous results. To be able to get a contradiction we will need to construct uncountably many bubble leaves with different landing points in a leaf of the other foliation. %For this, we cover $M$ with finitely many sufficiently small foliations boxes $\{B^i_j\}_{j}$ of $\cF_i$ (with $i=1,2$) and, for $x \in \mt$ denote by $B^i(x)$ the connected component of the lift of the box $B^i_j$ which contains $x$. 

%\begin{lema}
%
%\end{lema}
%
%\begin{proof}
%
%\end{proof}

 \begin{proof}[Proof of Proposition \ref{prop.somenonbubble}]
 Assume by contradiction that every leaf of $\wt{\cG}$ is a bubble leaf. 
 
 We claim that for each $L \in \wt{\cF_1}$ there is a non-trivial interval of the leaf space of $\wt{\cF_2}$ which contains bubble leaves whose limit points are $\alpha(L)$. This is just because Lemma \ref{lema-openbubbleland} states that the open bubble region $\cO_L$ of $L$, which exists due to Lemma \ref{lema-bubbleregions}, must land in $\alpha(L)$. 
%Let $p = \alpha(L)$. 
Considering a transversal $\tau$ to $\cG_L$ inside $\cO_L$ gives a non trivial interval $I_L$ of the leaf space of $\wt{\cF_2}$ (i.e. those leaves which intersect $\tau$) containing bubble leaves whose landing point is $\alpha(L)$. The interval $I_L$ also depends on $\tau$, but we are omitting this
dependence. Notice that we are not claiming that the elements of 
$\{ I_L, L \in \wcF_1 \}$ are pairwise distinct.
 
Pick $x_n$ dense and countable in the leaf space of $\wcF_2$ 
(which is homeomorphic to $\mathbb{R}$).
If for every $x_n$ there are only countably many $L \in \wcF_1$
so that $x_n \in I_L$, then the set of those $L \in \wcF_1$ for
which $I_L$ does not contain any $x_n$ is still uncountable.
Hence this set is non empty, contradicting the density of
$\{ x_n \}$, since any $I_L$ is an open set
in the leaf space of $\wcF_2$.

%We claim that there is a leaf $E$ of $\wcF_2$ there are 
%uncountably many $p$ for which $E \in I_p$.
%Otherwise for each rational point $x$ in the leaf space of 
%$\wcF_2$ (identify this leaf space to $\mathbb{R}$) there are
%countably many $p$ such that $x \in I_p$. Therefore there
%are countably many $p$ so that $I_p$ contains a rational.
%But any $I_p$ is a non trivial interval so it does contain
%a rational, contradiction.
%%Since for uncountably many non-trivial intervals in the line there must be a point which belongs to uncountably many of them, 
%
We deduce that there is a leaf $E \in \wt{\cF_2}$ for which
there are uncountably many $L \in \wcF_1$ so that $E \in I_L$ with pairwise distinct $\alpha(L)$. It follow that $E$ contains uncountably many bubble leaves which land in pairwise different points. This  contradicts the fact that a leaf can contain at most countably many distinct bubble regions (see Lemma \ref{lema-bubbleregions}). This contradiction completes the proof. 
 %Proof that uncountably many intervals in R must have a point that belongs to uncountably many intervals: By contradiction, assume all points belong to countably many. Remove from the list of intervals the intervals that contain rational points, those must be countably many, thus we are left with uncountably many intervals which do not contain rationals, so cannot be non-trivial. 
 \end{proof}

\subsection{Proof of Proposition \ref{prop.visual}} 
The proof will be by contradiction. Using what we have proved we will have a ray landing at a marker point. Then, assuming that the small visual property fails, we can use the following lemma and minimality to get a contradiction.

\begin{lema}\label{lema.zoomin}
For every $\delta>0$ there exists $K:=K(\delta)>0$ such that if: 
\begin{itemize}
\item $I$ is a segment in $\partial \HH^2$, 
\item $x \in \HH^2$ and 
\item $r$ is a geodesic ray starting from $x$ and landing in some point $\xi$ of $I$ separating $I$ in two intervals of visual measure $>\delta/2$ seen from $x$.
\end{itemize}
Then if $y \in r$ verifies that $d_{\mathbb{H}^2}(x,y)>K$, it follows that the visual measure of $I$ seen from $y$ is larger than $2\pi -\delta$. 
Here $d_{\mathbb{H}^2}$ is the hyperbolic metric in $\mathbb{H}^2$.
\end{lema}

\begin{proof}
Since isometries are transitive in $T^1\HH^2$ we can assume
in the upper half space model of $\HH^2$,
 that $x = i$, the geodesic ray $r$ is the ray $\{ai : a\geq 1\}$ and the interval $I$ contains, for some large $t$ depending on $\delta$ a set of the form $A_t=[t, +\infty) \cup \{\infty \} \cup (-\infty, -t]$. The lemma then reduces to computing the visual measure of $A_t$ from a point of the form $ai$ which goes to $2\pi$ with $a \to \infty$.
\end{proof}

\begin{proof}[Proof of Proposition \ref{prop.visual}] 
Consider $c \in \wt{\cG}$ given by Proposition \ref{prop.somenonbubble} and let $L \in \wt{\cF_1}$, $E \in \wt{\cF_2}$ so that $c \subset L \cap E$. We will argue for $L$ (the case for $E$ is symmetric).  

Let $r$ be a ray of $c$ so that its landing point $\partial \Phi(r)=\xi \neq \alpha(L)$. Let $I$ be a closed interval of the leaf space of $\wt{\cF_1}$ containing $L$ in its interior and such that if $L' \in I$ then $\alpha(L') \neq \xi$. Up to reducing the ray $r$ we can assume that $r \subset \wihatD_\eps(L,I)$ for $\eps$ as given in Proposition \ref{prop-pushing}. 

Assume now by contradiction that there is $\delta>0$ and a sequence $\ell_n$ of segments of $\cG_{L_n}$ with $L_n \in \wt{\cF_1}$ and points $x_n \in L_n$ so that $d_{L_n}(x_n, \ell_n) > 5n$ and 
$\ell_n$ has shadow in $S^1(L_n)$ of length $> 10 \delta$ when seen from $x_n$. Up to cutting the segments $\ell_n$ we can assume that the shadow has length $>5\delta$ and is disjoint from $\alpha(L_n)$. 

Consider, for each $n$ a geodesic ray from $x_n$ to the middle point of the shadow of $\ell_n$ and a point $y_n$ in the geodesic ray and at distance $=2n$ from $x_n$. By minimality of $\cF_1$ and up to changing slightly the geodesic ray, we can assume that there is $\gamma_n$ sending $y_n$ to a point very close to $L$ and in a compact set of $\mt$. Using Lemma \ref{lema.zoomin} we get that up to subsequence we have: 

\begin{itemize}
\item $d_{L_n}(y_n, \ell_n) > 2n$,
\item the visual measure of $\ell_n$ seen from $y_n$ is larger than $2\pi - a_n$ with $a_n \to 0$ (this is after identifying $S^1(L)$ with $\partial \HH^2$),
\item $\gamma_n y_n \to y_\infty \in L$. 
\end{itemize}

Note that we chose $\ell_n$ so that $\alpha(L_n)$ does not belong to the shadow of $\ell_n$ from $x_n$ and thus we get that the shadow from $\gamma_n y_n$ of $\gamma_n \ell_n$ contains $\xi$ in its interior if $n$
is big enough. This is because $\gamma_n L_n \to L$ and
$\alpha(\gamma_n L_n) \to \alpha(L)$. 
Since for large $n$ the leaf $\gamma_n L_n$ is close to $L$ we can assume that $\gamma_n L_n$ is in $I$ and we can push arcs $\ell'_n$ of $\gamma_n \ell_n$ to $L$ so that $\Phi_L(\ell'_n)$ accumulates in an interval containing $\xi$ in its interior. This forces $\ell'_n$ to eventually intersect $r$ which is a contradiction. 
\end{proof}

%\subsection{A consequence of small visual measure}

%As a consequence of the small visual measure property we deduce the following by taking limits: 

%\begin{coro}\label{coro.nonbubbleleaves}
%There exist $c \in \wt{\cG}$ leaves which are non-bubbles
%\end{coro}

%%%%%%%%%%%%%%%%%%%%
\section{Leafwise Hausdorff leaf space implies quasigeodesic}\label{s.Hsdff}

Here we will show the following: 

\begin{teo}\label{teo.hsdffcase}
Let $\cF_1$ and $\cF_2$ be transverse minimal foliations of $M = T^1S$ and 
let $\cG$ denote the foliation obtained by intersection. Assume that in some leaf $L \in \wt{\cF_1}$ the leaf space of $\cG_L$ is Hausdorff. Then, the foliation $\cG$ is an Anosov foliation. 
\end{teo}

This also follows from \cite{FP4} which gives an alternative proof. Here we present a proof in our restricted
setting, since it is vastly simpler and can show some of the ideas in a more transparent way. 

Recall that an \emph{Anosov foliation} is a one dimensional foliation in $M$ which is homeomorphic to the orbit foliation of a topological Anosov flow. As follows from \cite[\S 6]{BFP}, due to minimality, it is enough to show that the one-dimensional foliation is \emph{quasigeodesic}, that is, for every leaf $L \in \wt{\cF_i}$ the leaves of $\cG_L$ are quasigeodesics in $L$.  

Here we are working with $M=T^1S$ so it makes sense to compare our one-dimensional foliation with the one of the geodesic flow in negative curvature: this foliation, when restricted to a leaf $L$ of the weak-stable or weak-unstable foliation has the following properties that we will try to produce to show Theorem \ref{teo.hsdffcase}:

\begin{itemize}
\item one of the landing points of every leaf of $\cG_L$ is $\alpha(L)$, 
\item given any point $\xi \in \partial \HH^2 \setminus \{\alpha(L)\}$ there is a (unique) leaf whose landing points are $\xi$ and $\alpha(L)$. 
\end{itemize}

Thanks to some classical results, showing these properties will be enough to establish that the foliation is an Anosov foliation (see \cite{Ghys}). This implies that it is quasigeodesic (see \cite{Fenley}). We refer the reader to \cite[\S 5]{BFP} for more details.

\subsection{Bubble leaves and landing} 

First we need to show that bubble leaves (recall \S~\ref{ss.bubbleleafs}) produce non-Hausdorff behaviour. For this, the small visual measure property is crucial, in particular Proposition \ref{prop.visualconsequence}. 

We recall here that the assumptions of Theorem \ref{teo.hsdffcase} imply that in \emph{every} leaf of $\wt{\cF_1}$ and $\wt{\cF_2}$ the restriction of $\wt{\cG}$ has Hausdorff leaf space because of Proposition \ref{p.dichotomyHsdff}. 

%First we prove a general result:
%
%\begin{lema} \label{lema:general}
%Suppose that $\cF_1, \cF_2$ are transverse foliations in
%a closed $3$-manifold, and let $\cG = \cF_1 \cap \cF_2$.
%Suppose that there are $p_n, q_n$ in leaves
%$L_n$ of (say) $\wcF_1$, so that $p_n \to p, \ q_n \to q$,
%where $p, q$ are in the same leaf $L$ of $\wcF_1$. Then, if $p, q$
%are not in the same leaf of $\wcG$, it follows that $\cG_L$ does
%not have Hausdorff leaf space.
%\end{lema}
%
%\begin{proof}
%Denote by $c_p$ the leaf of $\cG_{L}$ containing $p$ and $c_q$ the leaf of $\cG_{L}$ containing $q$.  We claim that there is no transversal to $\cG_{L}$ in $L$  intersecting both $c_p$ and $c_q$. If that where the case, denote by $\tau_0: [0,1] \to L$ such transversal which we can assume verifies that $\tau(0)=p$ and $\tau(1)=q$. 
%Extend $\tau_0$ to a slightly larger 
%transversal $\tau: [-\eps, 1+\eps]$ to $\cG_{L}$ in $L$,
%with $\eps > 0$. Let $D_\tau$ be a transversal disk to $\wt{\cG}$ made of small segments transversal to $\wt{\cF_1}$ through the points $\tau(t)$, and
%contained in the respective leaves of $\wcF_2$. Since $p_n \to p$ and $q_n \to q$ it follows that the local leaf of $\wt{\cG}$ through those points must intersect $D_\tau$ for large $n$ $-$ to get this last property is the reason we extended $\tau_0$ 
%slightly. Since $p_n, q_n$ belong to the same leaf of $\cG_{L_n}$, this contradicts that $\cG_{L_n}$ is a foliation of a plane. This implies that the leaf space of $\cG_{L}$ cannot be Hausdorff as we wanted to show. 
%\end{proof}

\begin{lema}
If $\cG_L$ has a bubble leaf for some $L\in \wt{\cF_i}$, then there is a leaf $L' \in \wt{\cF_i}$ so that $\cG_{L'}$ has non-Hausdorff leaf space. 
\end{lema}
 
 %\ref{lem.nonHsdfftwointersect}
 
 \begin{proof}
 Consider a bubble leaf $c \in \cG_L$ (assume $L \in \wt{\cF_1}$) and fix a geodesic ray $r$ starting at some point in $c$ which has the same landing point $\xi \in \HH^2$ as $c$ (i.e. $\partial \Phi(r) = \partial^{\pm}\Phi(c)= \{\xi\}$). Let $x_0$ be the starting point of $r$ in $c$ and consider the rays $c_{x_0}^+$ and $c_{x_0}^-$. Passing to the limit in Proposition \ref{prop.visualconsequence} we know that there is some $a_0 >0$ such that  $r$ is contained in the $a_0$-neighborhood of $c_{x_0}^+$ and $c_{x_0}^-$. 
 
 Fix a sequence $y_n \in r$ such that $\Phi_L(y_n) \to \xi$ and consider points $p_n \in c_{x_0}^+$ and $q_n \in c_{x_0}^-$ at distance 
less than $a_0$ from $y_n$ in $L$. It follows that: 
 \begin{itemize}
 \item $d_L(p_n,q_n) \leq 2 a_0$, 
 \item $\Phi_L(p_n) \to \xi$ and $\Phi_L(q_n) \to \xi$, 
 \item the length of the segment of $c$ joining $p_n$ and $q_n$ goes to infinity. 
 \end{itemize} 
 
 Now consider deck transformations $\gamma_n \in \pi_1(M)$ which map $p_n$ to a compact fundamental domain $K \subset \mt$. Since $\gamma_n$ are isometries and $d_L(p_n,q_n)$ is bounded, the points $\gamma_n q_n$ also belong to a compact set. Up to subsequence we can assume that both $\gamma_n p_n$ and $\gamma_n q_n$ converge to points $p_\infty$ and $q_\infty$. Note that since the distance in $L$ from $p_n$ and $q_n$ is bounded, both points $p_\infty$ and $q_\infty$ belong to the same leaf $L' \in \wt{\cF_i}$. 
 
Since the length from $p_n$ to $q_n$ along $c$ goes to infinity 
it follows that $p_\infty$ and $q_\infty$ cannot belong to the same leaf of $\cG_{L'}$. On the other hand $p_n,q_n$ belong to the same leaf of $\wt{\cG}$ and thus the same leaf $E_n$ of $\wt{\cF_2}$ which since $p_n, q_n$ are close to $L'$ it follows that either $E_n \cap L'$ is not connected (in which case, the leaf space of $\cG_{L'}$ is not Hausdorff) or $E_n \cap L'$ is a sequence of leaves of $\cG_{L'}$ converging to both the leaf through $p_\infty$ and $q_\infty$ and again we get that $\cG_{L'}$ does not have Hausdorff leaf space.  
%Suppose not, let $c_\infty$ be the leaf of $\wcG$ containing both
%of them and $\nu$ the segment in $c_\infty$ from $x_\infty$ to
%$y_\infty$. Then using a product neighborhood of $\nu$ in 
%the foliation $\wcG$ it would follow that the segments in
%$\gamma_n c$ from $\gamma_n x_n$ to $\gamma_n y_n$ have to
%be in this neighborhood for $n$ sufficiently big. This would
%imply that the length of $\gamma_n c$ between these points
%is bounded, contradiction.

Since $p_\infty, q_\infty$ do not belong to the same leaf of
$\cG_{L'}$ but $p_n, q_n$ belong to the same leaf of $\wt{\cF_2}$ thus, $\cG_{L'}$ does not have Hausdorff leaf space as we wanted to show.
\end{proof}
 
 Now we show that at least one of the endpoints of each leaf in $\cG_L$ must be the non-marker point. Fix an orientation in $L$ and consider, for a given $c \in \cG_L$ which we know separates $L$ in two open disks $D^{\pm}_c$, each one verifying that the limit set of $\Phi_L(D^{\pm}_c)$ in $\partial \HH^2$ is one of the intervals joining the landing points of $c$ (see Corollary \ref{coro.separationH2}). 
 
 \begin{lema}\label{lema.hsdfflanding}
 Let $L \in \wt{\cF_i}$ and $c \in \cG_L$ then either $\partial^+\Phi(c)= \{\alpha(L)\}$ or $\partial^{-} \Phi (c) = \{\alpha(L)\}$. 
 \end{lema}

\begin{figure}[ht]
\begin{center}
\includegraphics[scale=0.56]{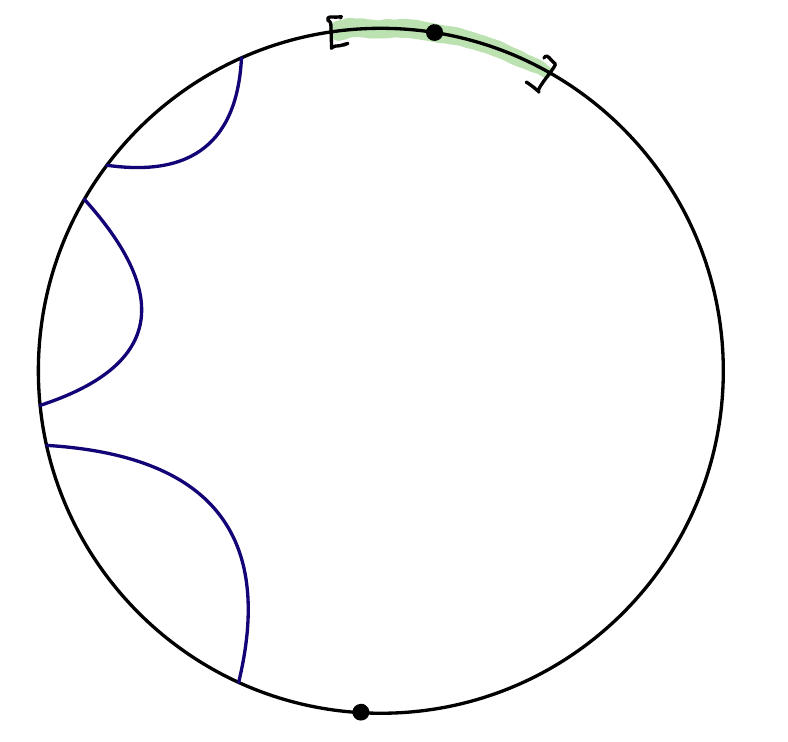}
\begin{picture}(0,0)
\put(-180,103){$\gamma c$}
\put(-159,153){$\gamma^2 c$}
\put(-203,116){$\gamma D^+_c$}
\put(-123,206){$\gamma^+=\alpha(L_\gamma)$}
\put(-150,54){$c$}
\put(-183,52){$D_c^+$}
\put(-125,15){$\gamma^{-}$}
%\put(-196,91){{\small $\cG^s(p)$}}
%\put(-113,225){$A_1^s$}
%\put(-82,220){$A_2^u$}
\end{picture}
\end{center}
\vspace{-0.5cm}
\caption{{\small If no ray lands in $\alpha(L)$ the leaf space cannot be Hausdorff.}}\label{fig-nonhsdff}
\end{figure}
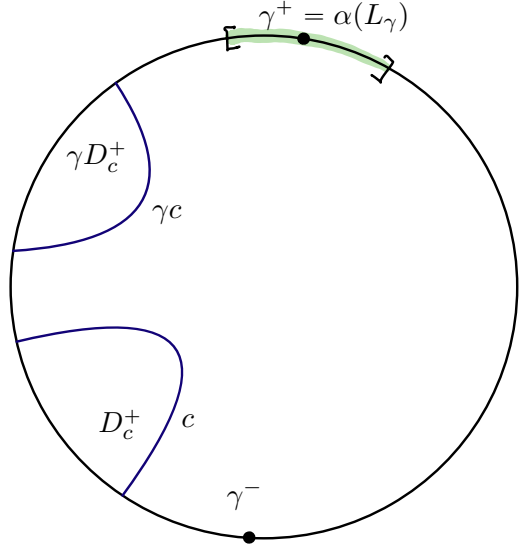

\begin{proof}
Let $L \in \wt{\cF_i}$ and $c \in \cG_L$ and assume that $\alpha(L) \notin \partial^{+} \Phi(c) \cup \partial^- \Phi(c)$. 

Proposition \ref{prop-pushing} allows us to assume that the leaf $L$ is fixed by some $\gamma \in \pi_1(M) \setminus \{\mathrm{id}\}$, and such that both points in $\partial^{+} \Phi(c) \cup \partial^- \Phi(c)$ belong to the same connected component of $\partial \HH^2$ minus the fixed points of $\gamma$. Indeed, it is enough to consider an interval $I$ of the leaf space of $\wt{\cF_i}$ containing $L$ in the interior so that the non-marker points of leaves in $I$ are never in $ \partial^{+} \Phi(c) \cup \partial^- \Phi(c)$:
hence given an even smaller interval we can push the entire leaf $c$ to nearby
leaves of $\wcF_i$ along the leaf of the other foliation with Proposition
\ref{prop-pushing}.
 Then since leaves with non-trivial stabilizer are dense (see \S~\ref{s.minimality}), we can push to one of them in $I$ to get the same property in such a leaf. 
To get that the fixed points of $\gamma$ do not link with 
$ \partial^+\Phi(c) \cup \partial^-\Phi(c) $ we can  use 
for instance that the set of closed geodesics in $S$ is dense in $S$ (cf. \S~\ref{s.minimality}).

We now choose orientations so that $D^+_c$ is the complementary
component of $c$ which does not limit on $\gamma^+, \gamma^-$.
If there is a transversal to $\cG_L$ in $L$ from $c$ to $\gamma c$,
then the image of this under $\gamma$ is a transversal
from $\gamma c$ to $\gamma^2 c$. 
Notice that $D^+_c, \gamma(D^+_c), \gamma^2(D^+_c)$ are
pairwise disjoint.
It now follows that $c, \gamma^2 c$
cannot intersect a common transversal and the leaf space cannot be Hausdorff. (See figure~\ref{fig-nonhsdff}.) 
\end{proof}

\subsection{Construction of the Anosov foliation}

To complete the proof that the foliation $\cG$ is an Anosov foliation we need to show that for every $L \in \wt{\cF_i}$ and $\xi \in \partial \HH^2 \setminus \{ \alpha(L) \}$ there is a leaf $c \in \cG_L$ such that $\partial \Phi(c) = \{\xi, \alpha(L) \}$ and that such a leaf is unique; this will produce an equivariant equivalence with the leaf space of the orbit flow of the geodesic flow in negative curvature \cite{Ghys} which is known to be enough to show that the foliation has the desired properties. 

We divide this into some steps: 

\begin{lema}
Let $L \in \wt{\cF_i}$ and $\xi \in \partial \HH^2$, then, there exists $c \in \cG_L$ with a ray landing in $\xi$ (i.e. such that $\xi \in \partial^{\pm} \Phi(c)$).  
\end{lema}

\begin{proof}
Recall from Lemma \ref{lema.landmarkerthendense} that the set of points in $\partial \HH^2$ for which there is a ray in $L$ which lands in that point is dense. 

Consider a sequence $c_n$ of leaves of $\cG_L$ converging to some leaf $c \in \cG_L$ (which is unique because the leaf space of $\cG_L$ is Hausdorff). 
Let $\xi_n$ be the landing point of $c_n$ different from $\alpha(L)$ and let $\xi$ and $\alpha(L)$ be the landing points of $c$. Assume that (up to taking a subsequence) $\xi_n \to \eta \neq \xi$. One can then consider a ray which lands in a point which belongs to the interval between $\xi$ and $\eta$ not containing $\alpha(L)$. This ray belongs to some leaf $\ell \in \cG_L$. By the
previous lemma, the leaf $\ell$ has one ideal
point in $\alpha(L)$. It follows that $\ell$
must then separate the curves $c_n$ with large $n$ from $c$, contradicting the fact that $c_n \to c$. 

This continuity plus the fact that the landing points are dense implies that every point different from $\alpha(L)$ 
is a landing point as we wanted to show. 
\end{proof}

Now we show uniqueness as a consequence of minimality of the foliations (in particular, Theorem \ref{teo.matsumoto} which implies that the non-marker points of leaves vary monotonically, see \S \ref{ss.nonmarker}). 

\begin{lema}\label{lema-nolandinginterval} 
Let $L \in \wt{\cF_i}$ and $\xi \in \partial \HH^2 \setminus \{\alpha(L)\}$, then, the leaf $c \in \cG_L$ such that $\partial^{\pm}\Phi(c) = \{\xi, \alpha(L)\}$ is unique. 
\end{lema}

\begin{proof}
Assume that in a leaf $L \in \wt{\cF_1}$ (the case where $L \in \wt{\cF_2}$ is symmetric) there are two leaves of $\cG_L$ so that they land in $\xi$ and $\alpha(L)$. Since the leaf space of $\cG_L$ is Hausdorff and one of the landing points of every $c\in \cG_L$ is $\alpha(L)$ we deduce that there is an interval $I$ in the leaf space of $\wt{\cF_2}$ consisting on leaves $E$ such that $E \cap L$ is a leaf $c \in \cG_L$ so that $\partial^{\pm} \Phi_L(c) = \{\xi, \alpha(L)\}$. From the previous lemma applied to $\wt{\cF_2}$ we know that if $E \in I$ and if $c' \in \cG_E$ then $\alpha(E) \in \partial^{\pm} \Phi_E(c')$. Considering $c \in E \cap L$ we deduce that for every $E \in I$ we have that $\alpha(E) \in \{\xi , \alpha(L)\}$. Since the point $\alpha(E)$ varies continuously with $E$ we deduce that there is an interval of $\wt{\cF_2}$ which has the same non-marker point while the non-marker point varies monotonically. 
%We consider $V= \bigcup_{E \in I} E$ which is a closed set in $\wt M$ with non-empty interior. We will show that if $\gamma \in \pi_1(M)$ verifies that $\gamma V \cap V \neq \emptyset$ then $\gamma V = V$. This implies that $V$ projects to a closed $\cF_2$ saturated subset of $M$ which is not all of $M$ and therefore contradicts minimality of $\cF_2$. 
%To see this we will see that for every $L' \in \wt{\cF_2}$ such that $V \cap L' \neq \emptyset$ we have that the intersection is a maximal interval of leaves of $\cG_{L'}$ with common landing points. For this, we first notice that $\alpha(L')$ must be contained in $\partial \HH^2 \setminus \{\xi\}$. This is because for 
\end{proof}

Now we are ready to prove the main result of this section. 

\begin{proof}[Proof of Theorem \ref{teo.hsdffcase}]
%Consider a homeomorphism $h: M \to M$ which maps $\cF_1$ to $\cF_{ws}$ (cf. Theorem \ref{teo.matsumoto}) and consider $\wt{h}$ its lift to the universal cover. Given $L \in \wt{\cF_1}$ we will produce a homeomorphism $v$ from the leaf space of $\cG_L$ and the 
Note that from our previous results we have a bijection between the space of leaves of $\wihat\cG$ (the lift of the foliation $\cG$ to $\wihatM$ the intermediate cover of $M$) and the set of pairs of distinct points of $\partial \HH^2$. This bijection is continuous and equivariant under deck transformations which is enough to show that the foliation is homeomorphic to the orbit foliation of the geodesic flow of a hyperbolic metric on $S$. (See \cite{Ghys} or \cite{BFP} for more details.)  
\end{proof}

%%%%%%%%%%%%%%%%%%%%%

\section{The Matsumoto-Tsuboi example revisited}\label{ss.MatsumotoTsuboi}
Here we revisit the example from \cite{MatsumotoTsuboi} from the point of view of our results. In \cite{MatsumotoTsuboi} an example of a pair of transverse foliations\footnote{In \cite{MatsumotoTsuboi} triples of transverse foliations are considered, but for the interests of our work we will ignore the third foliation.} of $\mathbb{T}^2 \times [-1,1]$ is constructed in such a way 
that the boundaries match with the weak stable and weak unstable foliations seen in the lift to $T^1S$ of a simple closed geodesic. That is, if $\ell$ is a closed geodesic in $S$, it lifts to a torus $T \subset T^1S$ which contains two periodic orbits of the geodesic flow (associated to the orbits associated to $\gamma$ with both orientations). Note that the weak stable and weak unstable foliations are horizontal, so they are transverse to $T$. Figure \ref{fig-transversetori} depicts the intersection of the foliations with the torus in its universal cover (which become tangent at the two periodic orbits which are common to both weak foliations). 

This way one can cut $T^1S$ along such a torus and glue this new foliation of  $\mathbb{T}^2 \times [-1,1]$ to obtain a new pair of transverse foliations of $T^1S$. From the way this pair of foliations is constructed, it is clear that the new foliations, called $\cF_1, \cF_2$, will continue to be minimal and their intersection foliation $\cG$ will verify that for every leaf $L \in \wt{\cF_i}$ the foliation $\cG_L$ will not have Hausdorff leaf space.  In this paper we will not use this construction, but we thought it could be relevant to spend some time explaining the example from our point of view which was relevant for us to formulate the right statements in the next two sections. We plan to study this example as well as other examples of transverse foliations with Reeb surfaces in a future work. 

\subsection{The construction and its possible variants}

Let $S$ be a closed surface with a hyperbolic metric and let $\varphi_t: T^1S \to T^1S$ be the geodesic flow. The metric induces an action of $\Gamma = \pi_1(S)$ on $\mathbb{H}^2$ and on $T^1 \mathbb{H}^2 = \wihatM$ so that the weak stable foliation of $\varphi_t$ is the foliation $\cF_{ws}$ defined in \S \ref{s.minimalfol} by identifying $T^1 \mathbb{H}^2$ with $\HH^2 \times \partial \HH^2$. The weak unstable foliation will be called $\cF_{wu}$ and is obtained in a similar process, only that the identification of $T^1\mathbb{H}^2$ with $\HH^2 \times \partial \HH^2$ is made by identifying the point at infinity as the limit of the geodesic in the backwards direction. Both $\cF_{ws}$ and $\cF_{wu}$ are transverse to the fibers.   

Fix any simple closed geodesic $\ell$ in $M=T^1S$ for the metric $g$ and let $T_\ell$ be the torus obtained by looking at the unit vectors tangent at points of $\ell$. This torus intersects both $\cF_{ws}$ and $\cF_{wu}$ transversally and the flow is transverse to $T_\ell$ except at the two orbits of $\varphi_t$ which are contained in $T_\ell$. See figure~\ref{fig-transversetori}. 

\begin{figure}[ht]
\begin{center}
\includegraphics[scale=0.68]{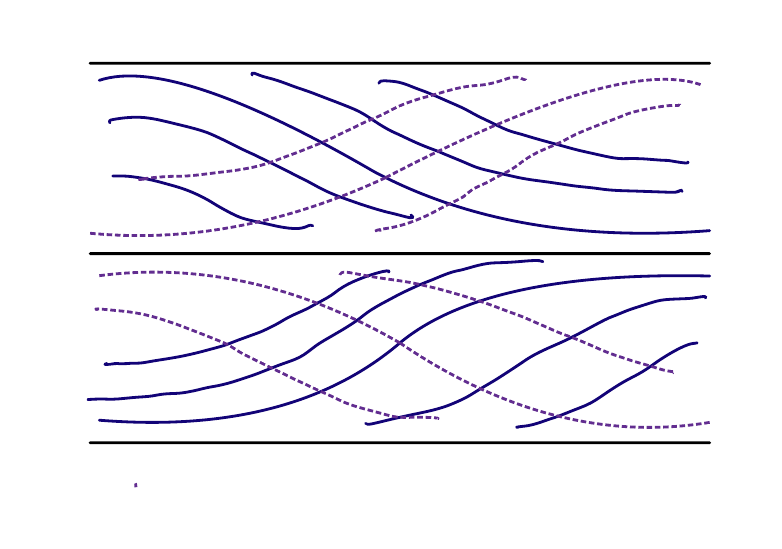}
%\begin{picture}(0,0)
%\put(-203,153){$o_1$}
%\put(-130,153){$p$}
%\put(-53,144){$o_2$}
%\put(-160,52){$A_1^u$}
%\put(-129,52){$A_2^s$}
%\put(-87,77){{\small $\cG^u(p)$}}
%\put(-196,91){{\small $\cG^s(p)$}}
%\put(-113,225){$A_1^s$}
%\put(-82,220){$A_2^u$}
%\end{picture}
\end{center}
\vspace{-0.5cm}
\caption{{\small The intersection of the transverse tori with the weak stable and weak unstable of the geodesic flow.}}\label{fig-transversetori}
\end{figure}

In \cite{MatsumotoTsuboi} a construction of two transverse foliations $\cS, \cU$ of $T_\ell \times [0,1]$ are given intersecting the boundary tori with the following properties: 

\begin{itemize}
\item each leaf $S \in \cS$ is either a cylinder (i.e. homeomorphic to $S^1 \times [0,1]$) or a band (i.e. homeomorphic to $\RR \times [0,1]$), and the intersection of $S$ with each boundary torus $T_\ell \times \{0\}$ and $T_\ell \times \{1\}$ are the same (in the trivialization
$T_\ell \times [0,1]$) and coincides with the intersection of some leaf of $\cF_{ws}$ with $T_\ell$, 
\item each leaf $U \in \cU$ is either a cylinder or a band, and the intersection with each boundary torus $T_\ell \times \{0\}$ and $T_\ell \times \{1\}$ is the same and coincides with the intersection of some leaf of $\cF_{wu}$ with $T_\ell$, 
\item there are exactly two cylinder leaves $S_1, S_2$ of $\cS$ and  $U_1, U_2$ of $\cU$ corresponding to the periodic orbits of the flow in $T_\ell$,
\item each leaf $S \in \cS \setminus \{S_1, S_2\}$ intersects either $U_1$ or $U_2$ in two connected components which are infinite lines and bound a Reeb band.   
\item each leaf $U \in \cS \setminus \{U_1, U_2\}$ intersects either $S_1$ or $S_2$ in two connected components which are infinite lines and bound a Reeb band.   
\item $S_1$ intersects $U_1$ in at least three circles, 
\item $S_2$ intersects $U_2$ in at least three circles and $U_1$ in at least four circles. 
\end{itemize}

We depict one possibility in figures \ref{fig-MT1} and \ref{fig-MT2}. 

\begin{figure}[ht]
\begin{center}
\includegraphics[scale=0.52]{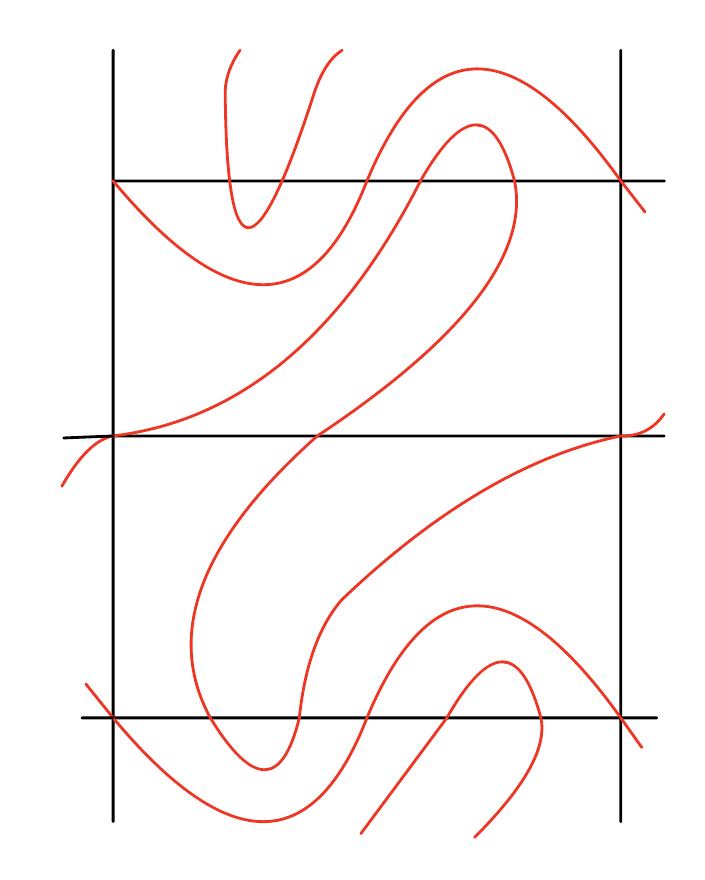}
%\begin{picture}(0,0)
%\put(-203,153){$o_1$}
%\put(-130,153){$p$}
%\put(-53,144){$o_2$}
%\put(-160,52){$A_1^u$}
%\put(-129,52){$A_2^s$}
%\put(-87,77){{\small $\cG^u(p)$}}
%\put(-196,91){{\small $\cG^s(p)$}}
%\put(-113,225){$A_1^s$}
%\put(-82,220){$A_2^u$}
%\end{picture}
\end{center}
\vspace{-0.5cm}
\caption{{\small The horizontal lines represent the cylinders $U_1$ and $U_2$ (and their translations up by deck transformations) and the red curves represent the cylinders $S_1$ and $S_2$ (and their translations up by deck transformations). The figure depicts how the leaves intersect and each intersection corresponds to a circle in the leaf. The rest of the leaves of the $\cU$ foliation are also horizontal leaves, but when going around the holonomy of the
compact leaves, they intersect in a different height (in particular, the corresponding leaf is an infinite band whose intersection with the torus is an infinite familly of horizontal segments which accumulate in $U_1$ and $U_2$ when going forward and backward respectively). The other leaves of the $\cS$ foliation interpolate the traces of $S_1$ and $S_2$ and also change as going around the torus in the flow direction.}}\label{fig-MT1}
\end{figure}

\begin{figure}[ht]
\begin{center}
\includegraphics[scale=0.72]{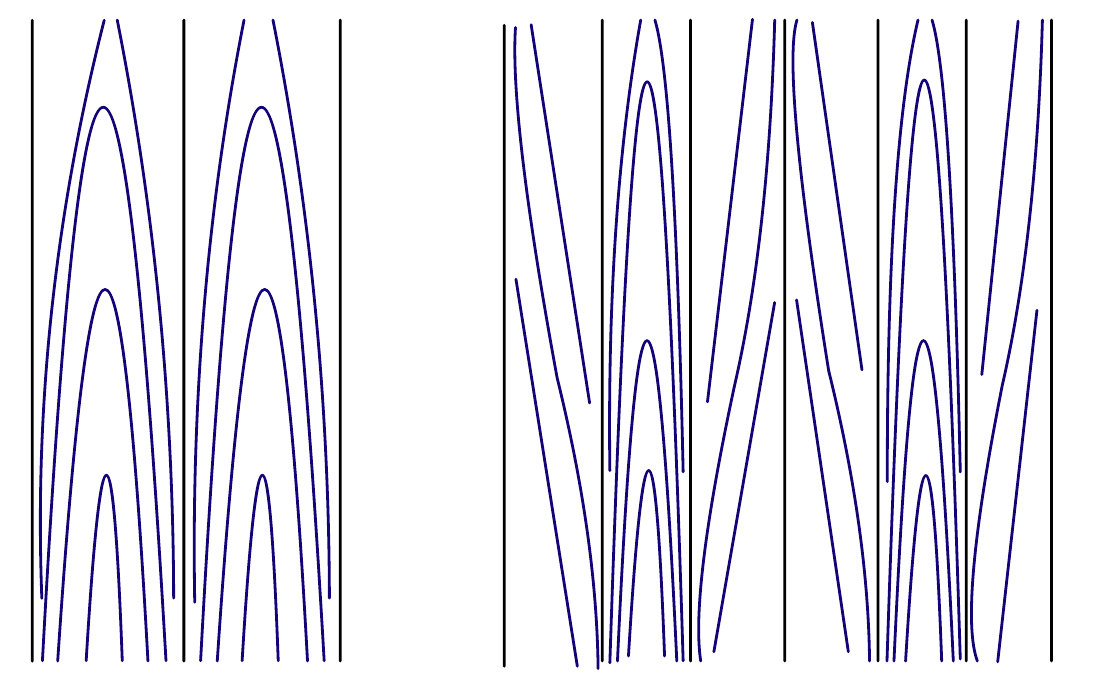}
%\begin{picture}(0,0)
%\put(-203,153){$o_1$}
%\put(-130,153){$p$}
%\put(-53,144){$o_2$}
%\put(-160,52){$A_1^u$}
%\put(-129,52){$A_2^s$}
%\put(-87,77){{\small $\cG^u(p)$}}
%\put(-196,91){{\small $\cG^s(p)$}}
%\put(-113,225){$A_1^s$}
%\put(-82,220){$A_2^u$}
%\end{picture}
\end{center}
\vspace{-0.5cm}
\caption{{\small In the left one can see the lift of one of the cylinders ($U_2$) and the intersected foliation lifted to this cover. In the right one sees the same for the other cylinder ($U_1$). These intersections correspond to intersecting the leaves as one moves in the direction of the flow (which makes leaves of the $\cS$ foliation approach in the forward or backward direction to the leaves $S_1$ and $S_2$). The straight lines corresponds to the intersection with the leaves $S_1$ and $S_2$ which are invariant under moving one step up.}}\label{fig-MT2}
\end{figure}

We now consider the foliations $\cF_1, \cF_2$ in $M_0$ obtained by first cutting $T^1S$ foliated with $\cF_{ws}$ and $\cF_{wu}$ along $T_\gamma$ and gluing a copy of $T_\ell \times [0,1]$ foliated by $\cS$ and $\cU$ and gluing the foliation $\cF_{ws}$ with $\cS$ and $\cF_{wu}$ with $\cU$. The manifold $M_0$ is still diffeomorphic to $T^1S$ (we have changed $T_\gamma$ for $T_\gamma \times [0,1]$ with the trivial identification of $T_\gamma \times \{0\}$ and $T_\gamma \times \{1\}$ with the two copies of $T_\gamma$ one gets after 'cutting' $T^1S$) and the foliations $\cF_{1}$ and $\cF_2$ are everywhere transverse. We note that it is possible to choose the foliations $\cS$ and $\cU$ in order that the resulting foliations $\cF_1$ and $\cF_2$ are \emph{not} uniformly equivalent, that is, the homeomorphisms $h_1$ and $h_2$ given by Theorem \ref{teo.matsumoto} sending $\cF_1$ and $\cF_2$ to $\cF_{ws}$ cannot be homotopic to each other (see figure \ref{fig-mtnonuniform1} and \ref{fig-mtnonuniform2}).The example in \cite{MatsumotoTsuboi} provides examples for which both foliations are uniformly equivalent.

\begin{figure}[ht]
\begin{center}
\includegraphics[scale=0.78]{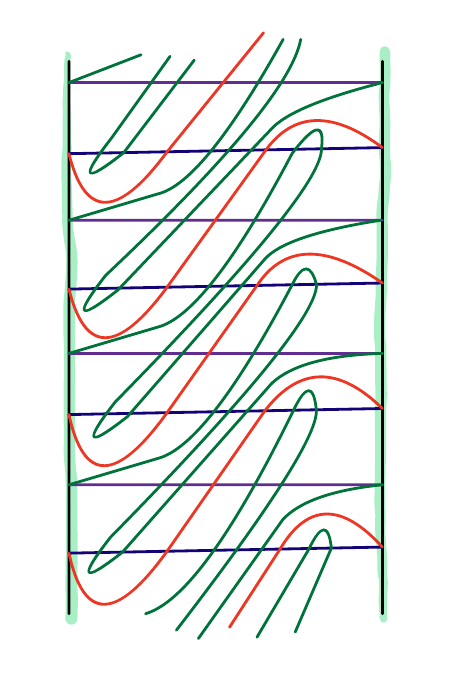}
%\begin{picture}(0,0)
%\put(-203,153){$o_1$}
%\put(-130,153){$p$}
%\put(-53,144){$o_2$}
%\put(-160,52){$A_1^u$}
%\put(-129,52){$A_2^s$}
%\put(-87,77){{\small $\cG^u(p)$}}
%\put(-196,91){{\small $\cG^s(p)$}}
%\put(-113,225){$A_1^s$}
%\put(-82,220){$A_2^u$}
%\end{picture}
\end{center}
\vspace{-0.5cm}
\caption{{\small Diagram on how the cylinders of each foliation can intersect in order to obtain a pair of non-uniformly equivalent foliations intersecting transversally.}}\label{fig-mtnonuniform1}
\end{figure}

\begin{figure}[ht]
\begin{center}
\includegraphics[scale=0.78]{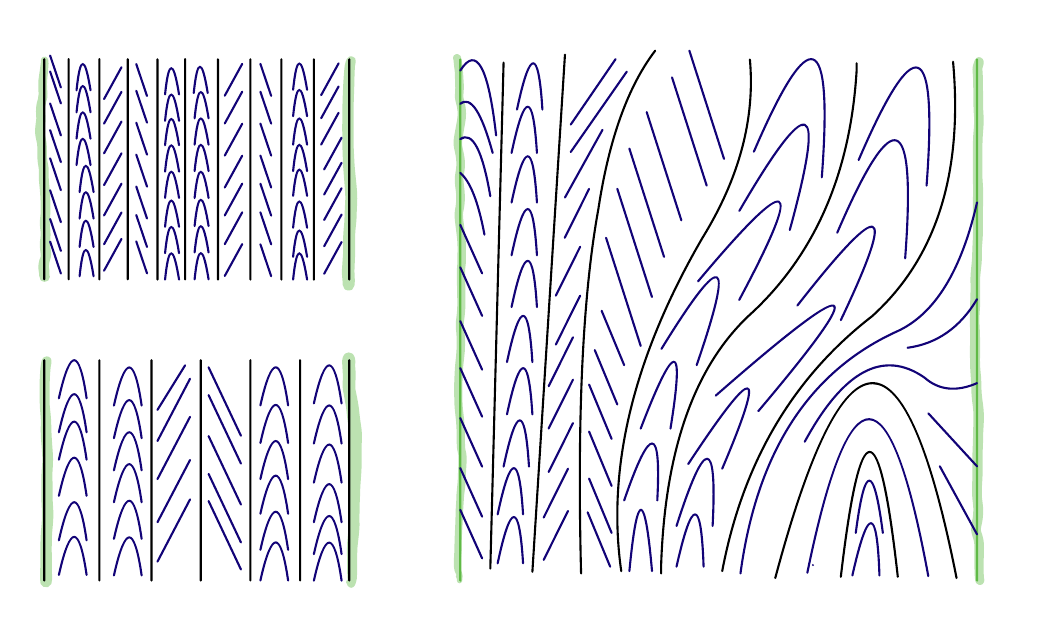}
%\begin{picture}(0,0)
%\put(-203,153){$o_1$}
%\put(-130,153){$p$}
%\put(-53,144){$o_2$}
%\put(-160,52){$A_1^u$}
%\put(-129,52){$A_2^s$}
%\put(-87,77){{\small $\cG^u(p)$}}
%\put(-196,91){{\small $\cG^s(p)$}}
%\put(-113,225){$A_1^s$}
%\put(-82,220){$A_2^u$}
%\end{picture}
\end{center}
\vspace{-0.5cm}
\caption{{\small In the left a diagram of the foliations in the cylinder leaves when crossing the tori. In the right, a depiction of the foliation interpolating both foliations of a leaf that accumulates in both cylinders in the different directions.}}\label{fig-mtnonuniform2}
\end{figure}

\begin{remark}\label{rem-variantsMT} 
Note that one can glue more than one block $T_\ell \times [0,1]$,
concatenetaning consecutive blocks.
It is also possible to do the cut and paste process in several disjoint simple closed geodesic to obtain variants of these examples.  More generally, take any Anosov flow in a closed 3-manifold with an embedded Birkhoff torus with only two Birkhoff annuli (that is, when the intersection of the weak stable and weak unstable foliations of the Anosov flow intersect the tori exactly as in this case). Then, one can glue some of the foliations of $\mathbb{T}^2 \times [0,1]$ and produce similar examples. We leave to a forthcoming work the (similar) analysis of this case.  One can also do this with embedded Birkhoff tori with more Birkhoff annuli,
but then the building blocks have to be accordingly adjusted: one needs
at least the same number of Birkhoff annuli, which can be achieved 
by finite covers of $\mathbb{T}^2 \times [0,1]$. This will also be pursued in the future. 
\end{remark}

\subsection{Properties of the intersected foliation} 

From the point of view of this paper, it is relevant for us to understand the properties of the foliation $\cG$ obtained by intersecting $\cF_1$ and $\cF_2$.  The main observations are the following: 

\begin{itemize}
\item Given $L \in \wt{\cF_1}$ we have that the foliation $\cG_L$ coincides with the foliation by geodesics as long as $L$ does not intersect the lift to the universal cover of $T_\ell \times [0,1]$. This lift has the property that it intersects every leaf in some regions which are bounded by two curves at bounded distance and landing at given points of $\partial \HH^2$ independent of the leaf $L$.
\item Inside each of the regions of intersection of the lift of $T_\ell \times [0,1]$ to $\mt$ the foliation $\cG_L$ has some design that depends on the specific foliation (see for instance figure~\ref{fig-MT2} for a depiction of the leaves in some specific examples). However, it is always the case that in these regions there is a pair of non-separated leaves at bounded distance of each other and landing at the two points at infinity of the region. Every leaf that enters the region must land in one of the endpoints of the region. 
In particular each ray of a leaf that enters this region does
not leave this region.
\item Given two non separated leaves $c_1,c_2 \in \cG_L$ they must accumulate in one of the regions mentioned in the previous point. Note that if the non separated leaves are not contained in the region modified between the two tori, the non-separated leaves may be at unbounded distance inside their leaf. %We point to figure~\ref{fig-MT2} to see how there can be non-separated leaves which are not bounded distance appart. 
\end{itemize}

These properties will become more apparent in \S~\ref{s.consReeb} which actually shows that behavior like this is the only possible.

%%%%%%%%%%%%%%%%%%%%%
\section{A dichotomy Hausdorff vs Reeb surfaces}\label{s.dichotomy} 

Let $\cF_1, \cF_2$ be two transverse minimal foliations on $M=T^1S$ and let $\cG$ be the foliation obtained by intersecting them. 

We define a \emph{Reeb surface} to be a compact surface with boundary $R$ verifying the following properties:

\begin{itemize}
\item $R$ is contained in a leaf $S$ of $\cF_i$ (with $i=1,2$),
\item the boundary components of $R$ are closed leaves of $\cG$, 
\item if $L$ is a lift of $S$ to the universal cover, then, the lift of $R$ to $L$ is a surface $B$ whose boundary consists of exactly two leaves $c_1, c_2$ of $\wt{\cG}$ at bounded distance which are non-separated in $\cG_L$.
\end{itemize}

In particular, it is easy to see that the surface must be either an annulus or a M\"{o}bius band. 

We will define for a leaf $L \in \wt{\cF_i}$ a \emph{non-Hausdorff bigon} to be a pair of non-separated leaves $c_1,c_2 \in \cG_L$ so that $\partial^{+}\Phi(c_1) \neq \partial^-\Phi(c_1)$ and $\partial^{\pm}\Phi(c_1)= \partial^{\pm} \Phi(c_2)$. 
It does not follow that the leaf space of $\cG_L$ inside the bigon
is the reals (see figure \ref{fig-nhb}). For example one can have bubble leaves $c, c'$ in the
bigon so that the disks $D_c, D_{c'}$ are disjoint.
On the other hand it is easy to see that if the boundaries of
the bigon are invariant by a non trivial deck transformation,
then the leaf space of $\cG_L$ inside the bigon is the reals. 
In particular, the existence of a Reeb surface is equivalent to having a non-Hausdorff bigon invariant under some deck transformation $\gamma \neq \mathrm{id}$.

\begin{figure}[ht]
\begin{center}
\includegraphics[scale=0.78]{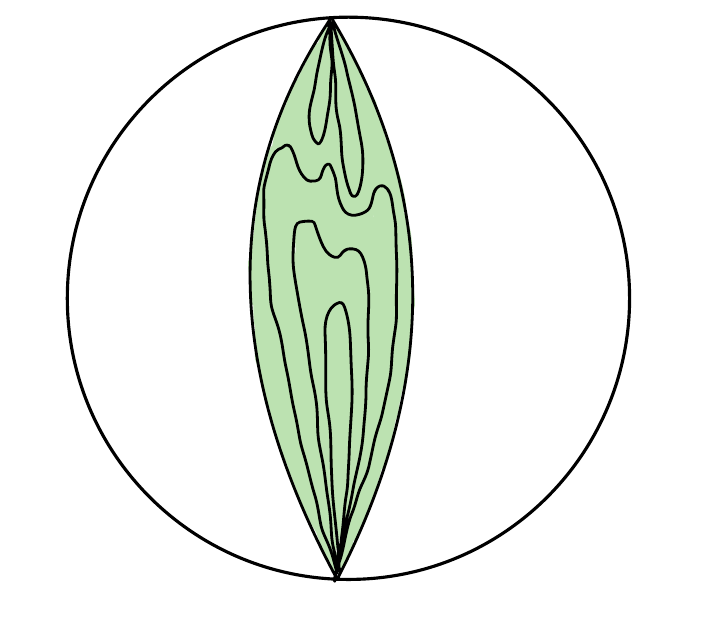}
%\begin{picture}(0,0)
%\put(-203,153){$o_1$}
%\put(-130,153){$p$}
%\put(-53,144){$o_2$}
%\put(-160,52){$A_1^u$}
%\put(-129,52){$A_2^s$}
%\put(-87,77){{\small $\cG^u(p)$}}
%\put(-196,91){{\small $\cG^s(p)$}}
%\put(-113,225){$A_1^s$}
%\put(-82,220){$A_2^u$}
%\end{picture}
\end{center}
\vspace{-0.5cm}
\caption{{\small The interior of a non-Hausdorff bigon may not be trivially foliated if it is not $\gamma$-invariant for some $\gamma \in \pi_1(M)$.}}\label{fig-nhb}
\end{figure}

Here we show: 

\begin{teo}\label{prop.dichotomy} 
Either there is a leaf $L \in \wt{\cF_1}$ so that the leaf space of $\cG_L$ is Hausdorff, or there exists a Reeb surface. 
\end{teo}

Theorem \ref{prop.dichotomy} together with Theorem \ref{teo.hsdffcase} is enough to conclude the proof of Theorem A. The purpose of this section is to prove Theorem \ref{prop.dichotomy}. 

\subsection{Some analysis on non-separated leaves}

We first find a useful  criterion to obtain Reeb surfaces: 

\begin{prop} \label{prop.someReeb}
Let $L \in \wt{\cF_i}$ containing two non-separated leaves $c_1,c_2 \in \cG_L$ such that there is $\gamma \in \pi_1(M) \setminus \{\mathrm{id}\}$ which fixes $L$ (i.e. $\gamma L = L$). Assume that the non-separated rays $r_1, r_2$ of $c_1, c_2$ (cf. \S \ref{ss.nonseparated}) verify that $\partial \Phi(r_1) = \partial \Phi (r_2) = \xi$ with $\gamma \xi = \xi$. Then, $\cF_i$ has a Reeb surface in the projection of $L$ to $M$.  
\end{prop}

Note in particular that if $\xi = \alpha(L)$ then the hypothesis that $\gamma \xi = \xi$ holds automatically. 

\begin{proof}
Let $\ell \subset L$ be a $\gamma$-invariant curve joining $\xi$ with the other fixed point of $\gamma$ in $S^1(L)$, without loss of generality we can assume that $\ell$ intersects $r_1$ and $r_2$.  The small visual measure property (cf. Proposition \ref{prop.visualconsequence}) implies that there is $a_0$ so that every ray $r$ of $\cG_L$ which lands in $\xi$ and intersects $\ell$ verifies that the $a_0$-neighborhood of $r$ contains a ray of $\ell$. This applies to $r_1, r_2$ but also their iterates $\gamma^k r_i$ for $i=1,2$ and $k \in \ZZ$. Since $\gamma^j r_1$ is non-separated with $\gamma^j r_2$ it follows that in a given fundamental domain, one can have at most finitely many such rays. This implies that there is $k \in \ZZ \setminus \{0\}$ so that $\gamma^k r_i \cap r_i \neq \emptyset$ and thus $\gamma^k c_i = c_i$ for $i=1,2$. 

Now, note that this implies that $c_1$ and $c_2$ join $\xi$ with the other fixed point of $\gamma$ and project to a (simple) closed curve in $M$. If $\gamma$ preserves orientation, it follows that both $c_1$ and $c_2$ are $\gamma$-invariant, else, it can be that $\gamma$ permutes them and $\gamma^2$ preserves them. In both cases we deduce that they bound a Reeb surface. 
\end{proof}

%The key point is to show that some power of $\gamma$ preserves both $c_1$ and $c_2$. This will imply that $\gamma^2$ also preserves them and therefore show that the region between $c_1$ and $c_2$ is a non-Hausdorff bigon and thus projects to a Reeb surface. This is because if they are $\gamma^k$ invariant for some $k$ it follows that both $c_1$ and $c_2$ must join the attracting and repelling point of $\gamma$. If $\gamma$ preserves orientation, this forces $c_1, c_2$ to be fixed by $\gamma$, if it reverses orientation, then $\gamma^2$ must fix them. 
%
%To show that $c_1$ and $c_2$ are $\gamma^k$ invariant for some $k>0$, the crucial point is to use the small visual measure property, particularly Proposition \ref{prop.visualconsequence}. In particular, if we consider $\ell$ a curve in $L$ which projects to a closed geodesic we know that $\ell$ has a ray
%which is contained in a uniform neighborhood of $r_1$ and $r_2$ as well as any other ray which lands in $\alpha(L)$. This implies that $\gamma^k r_1$ and $\gamma^k r_2$ also have points boundedly close to $\ell$. So, since $r_1$ and $r_2$ (as well as $\gamma^k r_1, \gamma^kr_2$) are non separated, then, if they are disjoint then they are some distance apart which is uniformly bounded from below. In particular, there cannot be infinitely many such pairs intersecting a a uniform neighborhood of $\ell$ in a given fundamental domain because this projects to a set with finite area. This shows that after some iterate, $\gamma^k r_i \cap r_i \neq \emptyset$ which implies that $\gamma^k c_i = c_i$ and this completes the proof.  

\subsection{Proof of Theorem \ref{prop.dichotomy}}
The following proposition will be proved in the next subsection. 

\begin{prop}\label{p.dichotomytech}
Let $L \in \wt{\cF_i}$ and $c_1,c_2 \in \cG_L$ be two non-separated leaves so that their non-separated rays $r_1,r_2$ verify $\partial \Phi(r_1)=\partial \Phi(r_2) = \xi \neq \alpha(L)$. Then, there is $L' \in \wt{\cF_i}$ and a deck transformation $\gamma \in \pi_1(M) \setminus \{\mathrm{id}\}$ so that: 
\begin{itemize}
\item $\gamma \xi = \xi $ and $\gamma L' = L'$, 
\item  there is a  $\gamma$-invariant non-Hausdorff bigon in $L'$ (which has $\xi$ as one of its endpoints). 
\end{itemize}
\end{prop}

Let us explain how Theorem \ref{prop.dichotomy} (and thus Theorem A) follows from this proposition. Note that as shown in Proposition \ref{p.dichotomyHsdff} if some leaf $L \in \wt{\cF_i}$ verifies that the leaf space of $\cG_L$ is non-Hausdorff then the same holds for every leaf of both foliations. By Proposition \ref{prop.someReeb}, if there is a leaf $L \in \wt{\cF_i}$ with non-trivial stabilizer for which the non-separated rays of a pair of non separated leaves land in $\alpha(L)$ then we can conclude the existence of a Reeb surface. So, we can assume that there is a leaf $L \in \cF_1$ which has non-separated leaves $c_1,c_2$ so that their non-separated rays $r_1,r_2$ verifiy that $\partial \Phi(r_1) = \partial \Phi(r_2) \neq \alpha(L)$. Thus, we can apply Proposition \ref{p.dichotomytech} to deduce the existence of a $\gamma$-invariant non-Hausdorff bigon in some leaf $L' \in \cF_1$ which implies the existence of a Reeb surface by Proposition \ref{prop.someReeb}. This completes the proof of Theorem \ref{prop.dichotomy}. 

\smallskip

For the purposes of the study we will make in \S~\ref{s.consReeb}, it is useful to obtain the following additional information: 

\begin{addendum}\label{add.spiral}
Under the assumptions of Proposition \ref{p.dichotomytech} one can choose $L'$ so that the rays $r_1,r_2$ are asymptotic  to $L'$ in $\mt$. 
\end{addendum}

Here, by `asymptotic' to $L'$ in $M$ we mean that if we parametrize $r_i$ in $M$ by arc length as $r_i(t)$, then, as $t$ goes to infinity, 
$d(r_i(t), L')$ converges to zero.
In other words the projection $\pi(r_i)$ of $r_i$ to $M$ is
asymptotic to a closed curve of $\cG$ in $\pi(L')$.
%all accumulation points of $\hat r_i(t)$ lie in the projection of $L'$. 

\subsection{Proof of Proposition \ref{p.dichotomytech}.}
The proof of this proposition has two main steps: we first analyze some impossible configurations of returns of the curves $r_i$  when projected to $M$ and then we use the small visual measure to determine that the rays $\pi(r_1)$ and $\pi(r_2)$ need to accumulate as desired. 

Let $r_1,r_2$ be non-separated rays of $\cG_L$ converging to $\xi \in S^1(L)$. Let $c_i$ be the leaves of $\cG_L$ containing them. Let $\sigma: [0,1] \to L$ be a properly embedded segment satisfying the following conditions: 
\begin{enumerate}
\item  $\sigma(0) \in r_1$, $\sigma(1) \in r_2$, 
\item $\sigma$ intersects
$c_1 \cup c_2$ only in the endpoints and $\sigma$ transverse to $\cG_L$ except at one interior 
point, 
\item There is a small transverse arc $\beta$ to $\cG_L$ starting
in $\alpha(0)$ and going in the direction of $c_2$ so that if a leaf
of $\cG_L$ intersects the interior of $\alpha$ then it intersects $\beta$.
\end{enumerate}
Let $\cC= r_1 \cup r_2 \cup \sigma([0,1])$ which is a Jordan curve and let $\cB= \cB(r_1,r_2,\sigma)$ be the closure of the connected component of the complement of $\cC$ which accumulates only on $\xi$ in $S^1(L)$. We say that $\cB$ is a \emph{good half band} if $\cB$ is the closure of the union of the arcs of leaves of $\cG_L$ contained in $\cB$ and joining points of $\sigma=\sigma([0,1])$. Equivalently, we say that $\cB$ is a good half band if the boundary contains two non-separated rays and it does not contain any complete leaf in its interior.

\begin{lema}\label{lem-nonseparatedraysGHB}
Let $c'_1,c'_2 \in \cG_L$ be two non-separated leaves in $L \in \wt{\cF_i}$ so that their non-separated rays $r'_1,r'_2$ verify $\partial \Phi(r'_1)=\partial \Phi(r'_2) = \xi$. Then, there exist non separated leaves $c_1,c_2 \in \cG_L$ with non-separated rays $r_1, r_2$ verifying $\partial \Phi(r_1)=\partial \Phi(r_2) = \xi$ and moreover verifies that it bounds a good half band.
\end{lema}

\begin{proof} 
Applying the consequence of the small visual measure property given in Proposition \ref{prop.visualconsequence} we get that for any geodesic ray $\ell_0$ landing in $\partial\Phi(r'_1)=\partial \Phi(r'_2)=\xi$ one has that $\ell_0$ has a subray contained in the $a_0+a_1$-neighborhood of every ray $\beta$ of $\cG_L$ that lands in $\xi$: this is because there is a geodesic ray with 
starting point the starting point of $\beta$, ideal point $\xi$ and
contained in the $a_0$ neighborhood of $\beta$. Then by Gromov hyperbolicity,
there is a global constant $a_1>0$ so that any two geodesic rays
with ideal point $\xi$, have subrays which are $< a_1$ Hausdorff distant
from each other in $L$.
%such a  h(note that since the ray may not start in the starting point of the rays, one may need to modify the value of $a_0$ from the proposition, but since all geodesic rays are asymptotic, the difference is in a compact part that can be absorbed by the constant). 

A priori there could be some leaves of $\cG_L$ non separated from $c'_1$ and $c'_2$ and between $c'_1, c'_2$ (i.e. if we pick a curve $\sigma'$ as above and define $\cB' = \cB(c_1', c_2', \sigma')$, there may be curves non separated from $c_1'$ and $c_2'$ which have both endpoints in $\xi$ and intersect $\cB'$ $-$ and hence contained in $\cB'$). But since all of these keep intersecting the $a_0 + a_1$ neighborhood of $\ell_0$, it follows
that there are only finitely many of them. 
Therefore up to replacing $c'_2$ if necessary by the leaf
non separated from $c'_1, c'_2$ and closest to $c'_1$ we assume there
are no non separated leaves from $c'_1, c'_2$ and between them. This way, up to replacing the curves and choosing the rays conveniently one can produce $\sigma$ so that it produces a good half band as above. 
\end{proof}%

We can then start the analysis with a good half band $\cB_L=\cB(r_1,r_2,\sigma)$.
Denote by  $E=E(r_1,r_2)$ the leaf of $\wt{\cF_2}$ so that $r_1 \cup r_2 \subset L \cap E$.  
Choosing $\alpha$ in a similar way as $\sigma$ but contained in $E$ and so that in the last condition, the curve $\beta$ is small, 
this implies  that every curve of $\cG$ intersecting the interior of the band intersects $\cB_L$ in a compact set.  It also implies that every leaf of $\wt{\cF_2}$ that intersects the interior of $\cB_L$ must intersect a transversal to $E$, in particular, $E$ cannot intersect the interior of $\cB_L$, unless $\gamma E = E$, in which case
we get the conclusion of Proposition \ref{p.dichotomytech}.  

We consider an arc $\tau$ in $E$ joining the endpoints of $\sigma$. Note that there is a uniform $a_1$ so that we can choose $\tau$ to have length less than $a_1$ because leaves of $\wt{\cF_2}$ are uniformly properly embedded\footnote{This follows from Theorem \ref{teo.matsumoto} in our case, but is also a general fact for $\RR$-covered foliations, see e.g. \cite[Lemma 4.48]{Calegari-book}).}. The $r_i$ are contained in leaves $c_i$ of $\wcG$. We also choose $\tau$ to that it does not intersect $c_1 \cup c_2$ in its interior. We denote by $ \cB_E= \wihat\cB(r_1,r_2,\tau)$ the half band in $E$ whose boundary in the compactification of $E$ is $r_1 \cup r_2 \cup \tau \cup \xi$. Note that this may not be a `good half band' as we defined before, because there is no reason for $r_1$ and $r_2$ to be non-separated in $\cG_E$. 

\begin{lema}\label{lem-embeddeddisk}
There is an embedded disk $D$ in $\mt$  whose boundary is $\sigma \cup \tau$ intersecting  $\cB_L \cup \cB_E$ only in the boundary. 
\end{lema}

\begin{proof} The bands $\cB_L$ and $\cB_E$ only intersect at their boundaries $r_1 \cup r_2$, and thus $\cB_L \cup \cB_E$ is a properly (and tamely) embedded copy of $S^1 \times [0,\infty)$ (in fact, it is piecewise $C^1$.  Taking the one point compactification of $\mt$ to $S^3$ we get that $\hat D= \cB_L \cup \cB_E \cup \{\infty\}$ is a tamely embedded disk whose boundary is $\tau \cup \sigma$ and therefore, there is a homeomorphism of $S^3$ sending $\hat D$ to the standard disk in a hemisphere (see \cite[\S 17]{Moise}). Thus one can homotope the disk $\hat D$ rel $\partial \hat D$ away from $\infty$ and so that it does not intersects $\hat D$ in the interior and get an embedded disk $D$ in $\mt$ with the desired properties.   
\end{proof}

We denote by $\cV:= \cV(r_1,r_2,\sigma,\tau)$ the region defined in $\mt$ bounded by $D \cup \cB_L \cup \cB_E$ such that in each leaf inside, it limits only on $\xi$.  This is a topological ball\footnote{Again, this is a consequence of Schoenflies theorem in dimension 3, see \cite[\S 17]{Moise}. Note that the surfaces we are constructing are all tamely embedded by construction (they are piecewise leaves of foliations, and $D$ can be chosen smooth).}.  We call such a region a \emph{good half region}. 

We analyze a good half region $\cV= \cV(r_1,r_2,\sigma,\tau)$ with boundary $D \cup \cB_L \cup \cB_E$ where $\cB_L$ is a good half band in $L \in \wt{\cF_1}$ and $\cB_E$ is a half band in $E$. Let $p_i \in \sigma \cap \tau \cap r_i$ the corner points of $D$ and small foliation boxes $B^i_1$ and $B^i_2$ of $\wt{\cF_1}$ and $\wt{\cF_2}$ respectively around $p_i$ (for $i=1,2$).  We consider $I^i_1, I^i_2$ small non degenerate
closed intervals in the leaf spaces of $\wt{\cF_1}$ and $\wt{\cF_2}$ consisting on the leaves intersecting $B^i_1$ and $B^i_2$ respectively. We can assume that the boxes are small enough so that (for $i=1,2; j=1,2$):

\begin{itemize}
\item every leaf of $\wt{\cF_1}$ through $B^i_1$ intersects $E$,
\item every leaf of $\wt{\cF_2}$ through $B^i_2$ intersects $L$, 
%\item for every $\gamma \in \pi_1(M) \setminus \{\mathrm{id}\}$ if the action of $\gamma$ maps $I^i_j$ into its interior, then $\gamma$ contains a unique fixed point inside $I^i_j$. 
\item the set $\wihatD_\eps (L, I^i_1)$ contains\footnote{This uses the fact that $\xi \neq \alpha(L)$.} $\cB_L$ (cf. Proposition \ref{prop.setDfol}),
\end{itemize}

Note that by slightly changing $B_1^i$ and $B_2^i$ we can and will assume that $I_1^1=I_1^2=I_1$ and $I_2^1 = I_2^2 = I_2$.  These conditions on the sets $B_i^j$ and the intervals $I_1,I_2$ will be assumed in what follows. We can also consider $I_1,I_2$ small enough so that they are disjoint from their images by the deck transformation associated to the center of $\pi_1(M)$ (i.e. they correspond to an interval in the leaf space $\hat \cL$ in $\widehat M$, cf. Corollary \ref{coro.foliations}). Note that this implies that every $\gamma \in \pi_1(M) \setminus \{ \mathrm{id}\}$ can have at most two fixed points in $I_i$ and if it has two, one must be attracting and one repelling.

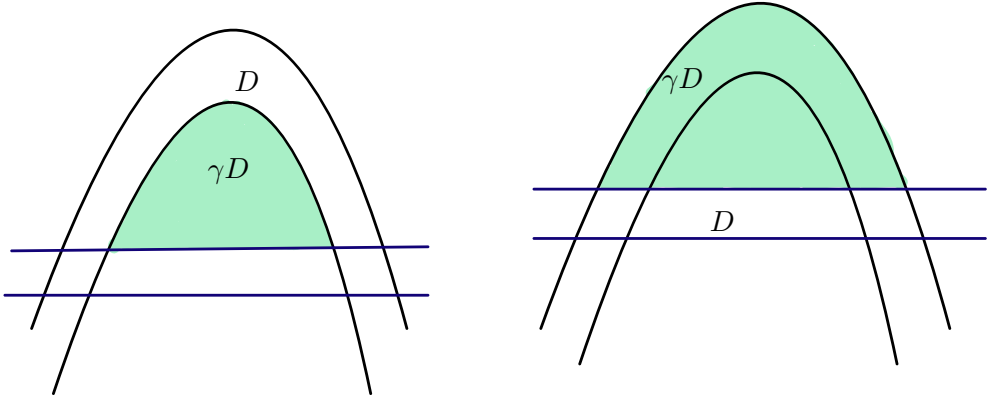
\begin{figure}[ht]
\begin{center}
\includegraphics[scale=0.78]{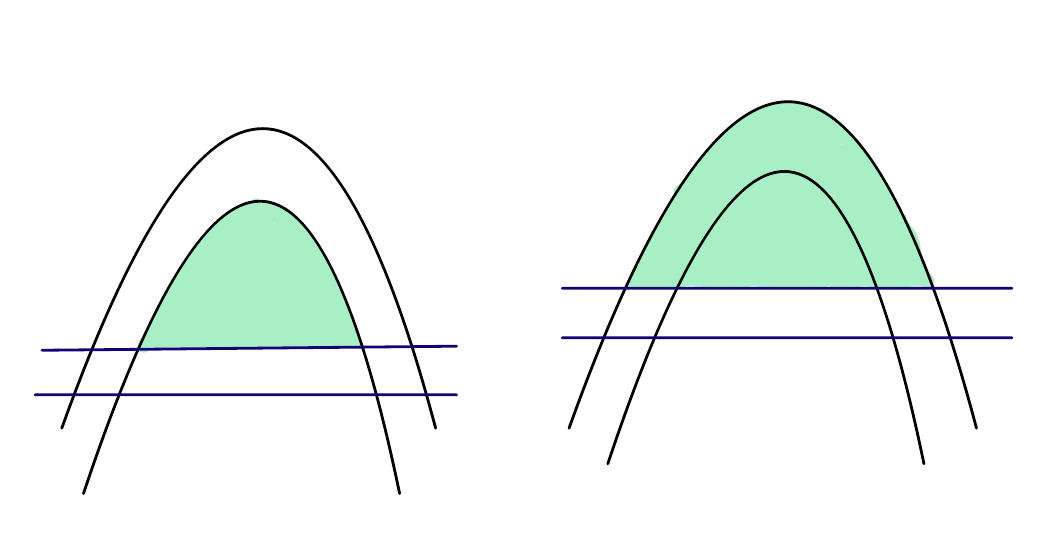}
\begin{picture}(0,0)
\put(78,90){$D$}
\put(-100,143){$D$}
\put(60,144){$\gamma D$}
\put(-110,110){$\gamma D$}
%\put(-129,52){$A_2^s$}
%\put(-87,77){{\small $\cG^u(p)$}}
%\put(-196,91){{\small $\cG^s(p)$}}
%\put(-113,225){$A_1^s$}
%\put(-82,220){$A_2^u$}
\end{picture}
\end{center}
\vspace{-0.5cm}
\caption{{\small Forbiden returns of a good half region according to Lemma \ref{lema.restrictionintersect}.}}\label{fig-noreturns}
\end{figure}

We are going to show that there are some restrictions on how $\cV$ can intersect with its translates under deck transformations. See Figure~\ref{fig-noreturns}. 

\begin{lema}\label{lema.restrictionintersect} 
Assume that $\gamma \in \pi_1(M)$ is such that:
\begin{itemize}
\item there is a point $z_1 \in r_1$ such that $\gamma z_1 \in B^1_1 \cap B^1_2$, 
\item there is a point $z_2 \in r_2$ so that the distance in $L$ from $z_1$ to $z_2$ is less than $2a_0+1$ and such that $\gamma z_2 \in B^2_1 \cap B^2_2$, 
\item $\gamma(I_1) \subset \mathrm{Int}(I_1)$, 
\item $I_2 \subset \mathrm{Int}(\gamma(I_2))$. 
\end{itemize}
Then, $\gamma L$ cannot intersect the interior of $\cV$, 
unless $\gamma E = E$, in which case we achieve
the conclusion of Proposition \ref{p.dichotomytech}. 
\end{lema}

\begin{proof}
We assume by contradiction that $\gamma L$  intersects the interior of $\cV$. From how we chose the intervals it follows that there is a (unique) fixed point of $\gamma$ in $I_1$ (which is attracting) and one in $I_2$ (which is repelling). Let $L'$ be the fixed point by $\gamma$ in $I_1$ and similarly $E'$ the fixed point of $\gamma$ in $I_2$.  Note that since $\gamma$ is atracting on $I_1$ then either $L' = L$ or $L'$ belongs to the same connected component of $I_1 \setminus \{L\}$ as $\gamma L$ and since $\gamma$ is expanding on $I_2$, either $E'=E$ or $E'$ belongs to the connected component of $I_2 \setminus \{E\}$ not containing $\gamma E$ (recall that the intervals where chosen small enough so that the action of $\gamma$ has a unique fixed point in $I_1$ and $I_2$).  In particular, both $E'$ and $L'$ intersect $B_i^j$ for $i=1,2 ; j=1,2$ and using Corollary \ref{coro.foliations} we know that they intersect each such box in a unique plaque. 

Since we have assumed that $\gamma L$ intersects the interior of $\cV$ then
$\gamma L \not = L$, and so by our choice of $I_1$, it follows
that $L'$ also intersects the interior of $\cV$, since it belongs to the connected component of $I_1 \setminus L$ containing $\gamma L$ (the set of leaves of $\wcF_1$ which intersect the interior of $\cV$). Since $L' \in I_1$ and $\cB_L \subset \wihatD_\eps(L,I_1)$ we have that $L$ and $L'$ are asymptotic in $\cB_L$ (i.e. $d(L \cap \cB_L, L') < \eps$,
meaning every point in $\cB_L$ is $< \eps$ from $L'$) and we can apply Proposition \ref{prop-pushing}. Denote by $\cB_{L'}$ the band in $L'$ obtained by pushing from $\cB_L$ to $L'$ in the sense of Proposition \ref{prop-pushing} (i.e. the band $\cB_{L'}$ is the one bounded by the intersections of $E$ with $L'$ which are close to $r_1$ and $r_2$). Note that this implies that $\cB_{L'}$ is a good half band, and its boundary rays that we denote as $r_1', r_2'$ (and are contained in $L' \cap E$) land in $\xi$ and are non-separated.

Let us first assume that $\gamma E$ intersects $\cV$.
In this case, since $\gamma z_1 \in B_1^1 \cap B_2^1$ and $\gamma z_2 \in B_1^2 \cap B_2^2$ it follows that $\gamma \cB_L$ and $\gamma \cB_E$ must  intersect $\cV$.  We obtain that $\cV$ projects to a solid torus in $M_\gamma = \mt/_{<\gamma>}$. One gets that $\gamma \cB_L$, whose boundaries are $\gamma r_1$ and $\gamma r_2$ has a definite width in $\gamma L$, meaning that the distance from one curve to the other is bounded below by some uniform constant due to the fact that they belong to the same leaf of $\wt{\cF_2}$. This implies that the band cannot disappear and is completely contained in $\cV$. It must thus be asymptotic to an  infinite strip strictly inside $\cB_{L'}$ which must be invariant under $\gamma$ and such that the boundaries are curves  which limit on $\xi$ and separate $r_1'$ from $r_2'$ contradicting the fact that $\cB_{L'}$ is a good half band. 

Suppose now that $\gamma E = E$. 
The argument in the previous case still applies so 
the band $\cB_L$ is asymptotic to a band in $L'$ which is invariant 
under $\gamma$.  This produces
a bigon in $L'$ and $\cB_L$ is asymptotic to this 
$\gamma$-invariant bigon. This achieves the conclusion of
Proposition \ref{p.dichotomytech} in this case.

Finally we treat the case that $\gamma E$ does not intersect $\cV$.  This implies that $E'$, the fixed point of $\gamma$ in $I_2$ belongs to the connected component of $I_2 \setminus \{E\}$ not containing $\gamma E$. Thus, $E'$ must intersect $\cV$. Since $\cB_{L'}$ is a good half band, each connected component of $E' \cap \cB_{L'}$ is a piece of a leaf of $\cG_{L'}$ and must be a compact interval by definition of good half band. The same holds for $E' \cap \cB_L$ due to Proposition \ref{prop-pushing}. However, since $\gamma$ fixes both $L'$ and $E'$ and because $E'$ intersects $B_1^2$ in a unique plaque we get that $\gamma$ must fix the connected component of the intersection $L' \cap E'$ intersecting $B_1^2$ and thus projects to a closed curve which lifts to a curve separating $r_1'$ from $r_2'$ we have shown had to be non-separated. This is a contradiction and completes the proof. 
\end{proof}

Now, we will show that it is possible to find deck transformations with the properties required by Lemma \ref{lema.restrictionintersect}. 

Consider a geodesic ray $\ell_0$ in $L$ so that it limits in $\xi$. The small visual measure property implies that there is some $a_0>0$ so that  for each $x$ in $\ell_0$ there is $z$ in $r_i$ with $d_L(x,z) < a_0$ (see Proposition \ref{prop.visualconsequence}).

%Now chooose $x_n$ in $\ell_0$ so that $x_n \to \xi$.

\begin{lema}\label{lem.constructcontract} 
Let $\cV= \cV(r_1,r_2,\sigma, \tau)$ be a good half region and $p_1 = \sigma \cap r_1 \cap \tau$. Consider $B^1_1, B^1_2$ and $I_1,I_2$ as above and assume that there exists a sequence of points $z_n$ in $r_1$, with $d_L(z_n, x_n) < a_0$ where  the $x_n$ are in a geodesic ray 
$\ell_0$ as above and $x_n \to \xi$. Assume moreover that there are deck transformations $\gamma_n$ so that $\gamma_n z_n \in B^1_1 \cap B^1_2$. Then, for large enough $n$ we have that $\gamma_n I_1 \subset \mathrm{Int}( I_1)$ and $\gamma_n^{-1} I_2 \subset \mathrm{Int}(I_2)$. 
\end{lema}

\begin{proof} 
This will use that $\xi \neq \alpha(L)$,  but $\xi = \alpha(E)$ by Proposition \ref{prop-nonseplands}.  
Up to subsequence assume that 
$\gamma_n z_n$ and $\gamma_n x_n$ converge. 

Fix a transversal $\zeta$ to $\cF_1$ through $p_1$ intersecting
exactly the leaves of $\wcF_1$ in $I_1$.
Let $L_1, L_2$ be the leaves of $\wcF_1$ through the endpoints
of $\zeta$. Recall that the ideal point of $r_1, r_2$ and hence
that of $\ell_0$ is $\xi \neq \alpha(L)$ so it is a contracting direction
for $\wcF_1$. So for big enough $n$ there is a transversal
to $\wcF_1$ of arbitrarily short length, through $x_n$ and connecting
the leaves $L_1, L_2$. 
By assumption $d_L(x_n,z_n)$ is bounded
by $a_0$ independently of $n$, so using the local product structure of the foliation
$\cF_1$ we obtain that there is also a transversal $\beta_n$ 
to $\wcF_1$ through $z_n$
 of very small length and connecting $L_1$ to $L_2$. 
Then $\gamma_n \beta_n$ is a transversal
of very small length passing through $\gamma_n z_n$ which 
is very close to $p$. In particular for 
$n$ sufficiently $\gamma_n \beta_n$ intersects a set of leaves
of $\wcF_1$ which is strictly contained in $I_1$.
This implies that $\gamma_n I_1 \subset \mathrm{Int}(I_1)$ for $n$ 
big enough.

Next consider $I_2$ and $\wcF_2$. 
We first have to verify that in $E$ the points $z_n$ are a
bounded distance from a geodesic ray in $E$.
First recall that Proposition \ref{prop.QI} proves that
there is a constant $Q_0 > 0$ so that for any $F$ a leaf
of $\wcF_i$ then $\Phi^i_F: F \to \HH^2$ is a
$Q_0$ quasi-isometry. %The metric in $F$ is the path metric.

The $z_n$ are a bounded distance from the geodesic ray
$\ell_0$ in $L$. 
By Proposition \ref{prop.QI} the image $\Phi^1_L(\ell_0)$ 
is a quasigeodesic and hence the points $\Phi^1_L(z_n)$ are a bounded
distance from a geodesic ray in $\HH^2$ which we denote by
$\ell_1$. Lemma \ref{lema.proyection} implies that there is $a_1 > 0$ such
that for any point $x$ in $\mt$, if $x \in L \in \wcF_1$
and $x \in E \in \wcF_2$, then
$d_{\HH^2}(\Phi^1_L(x), \Phi^2_E(x)) < a_1$.
Hence $\Phi^2_E(z_n)$ are a bounded distance from $\ell_1$.
Applying $(\Phi^2_E)^{-1}$ shows that $z_n$ are a bounded
distance from $(\Phi^2_E)^{-1}(\ell_1)$. The last curve
is a quasigeodesic in $E$ (again Proposition \ref{prop.QI}) and 
we conclude that $z_n$ are a bounded distance from a geodesic
ray in $E$ as claimed.

We know that $\ell_1$ converges to
$\alpha(E)$, so holonomy of $\wcF_2$
along $\ell_1$ is expanding. We have chosen $B^1_2$ sufficiently small, so that the transversal through
$z_n$ intersecting the leaves corresponding to $B^1_2$ will have length much bigger than length of a transversal through $p_1$. Then one gets that $\mathrm{Int}(\gamma_n I_2) \supset I_2$, or equivalently $\gamma^{-1}_n I_2 \subset \mathrm{Int}( I_2)$ as claimed.
\end{proof}

Before we show that an accumulation point must be fixed by some deck transformation we will show some general property about geodesic rays in leaves of
$\wcF_i$.

\begin{lema}\label{lema-returnsbothsides}
Let $\ell$ be a geodesic ray in some leaf $L \in \wt{\cF_i}$ with $\partial \Phi(\ell)= \xi \in \partial \HH^2$ then one of the following holds: \begin{enumerate}
\item there is $\gamma \in \pi_1(M) \setminus \{\mathrm{id}\}$ and a leaf $L' \in \wt{\cF_i}$ such that $\gamma L' = L'$ and $\gamma \xi = \xi$ and the projection of $\ell$ spirals towards the projection of $L'$, or, 
\item for every $\eps_1>0$ %and $m \in \ZZ_{>0}$,
there is a foliated box $U$ in $\mt$ of diameter less than $\eps_1$, a sequence of points $y_n \in \ell$ going to infinity and deck transformations $\eta_n \in \pi_1(M)$  for which $\eta_n y_n \in U$ are in different leaves of $\wt{\cF_i}$ and such that $\eta_n L$ accumulate in infinitely many distinct leaves of $\wcF_i$.  
\end{enumerate}
\end{lema}

\begin{proof}
Up to changing $\ell$ by a uniformly bounded amount which does not affect the result, one can assume that $\cF_i$ is the weak stable foliation of the geodesic flow on $M=T^1S$ for a hyperbolic metric on $S$ (cf. Theorem \ref{teo.matsumoto}). Note that since this is the case, the accumulation set of $\ell$ when projected to $M$ is the same as the accumulation of the orbit of the geodesic flow from $\alpha(L)$ to $\xi$ inside $L$ traversed in 
the direction opposite to the flow. In particular 
it consists of a compact connected set saturated by orbits that we shall denote by $\Lambda$. 

First, note that unless $L = \gamma L$ and $\gamma \xi = \xi$ for some $\gamma \in \pi_1(M) \setminus \{\mathrm{id}\}$ (in which case $\Lambda$ is the unique closed geodesic in $L$ and we are in the first option) we have that all returns of the projection of $\ell$ to $M$ to a foliation box must happen in distinct plaques of $\cF_1$. This is the same for any backward orbit of the geodesic flow, if it limits in a point $x$ and it is not periodic, then it must intersect a transversal to $x$ in infinitely many distinct plaques of $\cF_1$ (because $\cF_1$ is the weak stable foliation of the geodesic flow by construction).  

Next, we assume that there is some $\eps_1>0$ so that for every $x\in \Lambda$ the two-dimensional transversal of size $\eps_1$ to the geodesic flow through the point $x$ does not intersect $\Lambda$ except perhaps at $x$ itself. In this case, we get that $\Lambda$ is a closed geodesic which belongs to some leaf $\cF_1$ and considering some lift $L'$ of this leaf to $\mt$ gives us the first option of the lemma. 

Finally, assume that $\Lambda$ is not a periodic orbit. Then, there is $x \in \Lambda$ such that every transversal to the flow through $x$ intersects $\Lambda$ in infinitely many distinct leaves. For every foliated neighborhood $U$ of $x$ we get that the returns need to have accumulation points every point of $\Lambda \cap U$ which has infinitely many different leaves obtaining the second option of the lemma.  
\end{proof}

Using the small visual measure property, one can extend this to rays of $\cG_L$ for every $L$. 

\begin{coro}\label{coro-accum}
Let $r$ be a ray in some leaf of $\cG_L$ for $L \in \wt{\cF_i}$ with
ideal point $\xi$ and let $\ell_0$ be a geodesic ray
in $L$ with ideal point $\xi$. Suppose that
$\ell_0$ does not spiral towards a leaf $L' \in \wt{\cF_i}$ invariant under some deck tansformation. Then, for every $\eps_2>0$ there is a sequence of points $x_n \in r$ going to infinity and deck transformations $\eta_n \in \pi_1(M)$ such that for every $n_0>0$ there are $n_1< n_2 < n_3$ so that $\eta_{n_1}x_{n_1}, \eta_{n_2}x_{n_2}, \eta_{n_3}x_{n_3}$ are $\eps_2$ distance apart and $\eta_{n_2}L$ is between $\eta_{n_2}L$ and $\eta_{n_3}L$.  
\end{coro}

\begin{proof} 
%Let $\xi = \partial \Phi(r) \in \partial \HH^2$ be the landing point of the ray $r$.  Let $\ell_0$ be a geodesic ray in $L$ whose limit is $\xi$. 
Fix $a_0$ given by the small visual measure so that for every $y \in \ell_0$ there is a point $x \in r$ so that $d(x,y)<a_0$. For $y \in \mt$ denote by $D_y$ to the disk of $\wcF_i$ of radius $a_0$ centred in $y$.

We fix $\eps_2$ and choose $\eps_1$ sufficiently small so that if two points $z,w$ are $\eps_1$-close, then, the disks $D_z$ and $D_w$ are at Hausdorff distance less than $\eps_2$. Now, choose $k$ sufficiently large so that if one chooses $k$-points $z_1, \ldots, z_k$ in a transversal to $\wcF_i$ of length $\eps_1$ and picks one point $w_i$ in each $D_{z_i}$ then it holds that at least 3 of them are $\eps_2$ apart. 

Note that the assumption on $r$ implies that $\ell_0$ verifies the second option of the previous lemma, implying that for every $\eps_1>0$ there is  a sequence of points $y_n \in \ell_0$ going to infinity in $\ell_0$ and deck transformations $\eta_n$ so that all points $\eta_n y_n$ belong to an open set $U$ of diameter less than $\eps_1$ and so that the leaves $\eta_n L$ have infinitely many accumulation points. In particular, we can assume that $\eta_n y_n$ accumulate in $k$-points $z_1, \ldots, z_k \in U$ which belong to different leaves of $\wcF_i$. 

Now, given $n_0$, we can choose $n_1< n_2< n_3$ and points $y_{n_1},y_{n_2}$ and $y_{n_3}$ so that $\eta_{n_i}y_{n_i}$ are in $U$ and the leaf $\eta_{n_2}L$ is between the leaves $\eta_{n_1}L$ and $\eta_{n_3}L$ and so that the corresponding points $x_{n_i}$ in $r$ are $\eps_2$ close. This completes the proof. 
\end{proof}

The next lemma then completes the proof of  Proposition \ref{p.dichotomytech} and Addendum \ref{add.spiral}. 

\begin{lema}\label{lem.fixedpointboundary}
Let $c_1,c_2 \in \cG_L$ be two non-separated leaves in $L \in \wt{\cF_i}$ so that their non-separated rays $r_1,r_2$ verify $\partial \Phi(r_1)=\partial \Phi(r_2) = \xi \neq \alpha(L)$.
Then there exists some $\gamma \in \pi_1(M) \setminus \{\mathrm{id}\}$ and $L' \in \wt{\cF_i}$ which verify that $\gamma L' = L'$ and $\gamma \xi =\xi$. Moreover, the rays $r_i$ when projected to $M$ spiral towards the projection of $L'$ and there exists a $\gamma$-invariant bigon in $L'$.
\end{lema}

\begin{proof} 
%We assume as above that $L \in \wt{\cF_1}$ and consider a geodesic ray $\ell_0$ in $L$ landing in $\xi$ that we assume is contained in the $a_0+1$ neighborhood of $r_1$ and $r_2$. As before, we can a

Assume without loss of generality (cf. Lemma \ref{lem-nonseparatedraysGHB}) that $r_1$ and $r_2$ are boundaries of a good half band $\cB_L$ in $L$. 

Let $\cV:= \cV(r_1,r_2, \sigma, \tau)$ a good half region with boundary $D \cup \cB_L \cup \cB_E$ constructed as in Lemma \ref{lem-embeddeddisk}. We can choose an ordering of the leaf space of $\wt{\cF_1}$ so that leaves that intersect $\cV$ are above $L$. 

By Proposition \ref{prop.someReeb} we can assume that $r_1$ verifies the assumptions of Corollary \ref{coro-accum}. Orient $r_1$ and given $x \in r_1$ consider $r_1(x)$ to be the ray starting at $x$. We can also define $\cV_x :=\cV(r_1(x), r_2, \sigma_x, \tau_x)$ by changing the arcs to intersect in $x$. Now, using Corollary \ref{coro-accum} given $\eps_2$ we can find points $x, y$ in $r_1$ so that $y$ is in $r_1(x)$ and a deck transformation $\eta$ so that $x$ and $\eta y$ and are very close and such that $\eta L$ is above $L$. This contradicts  Lemma \ref{lema.restrictionintersect}.

 It follows that the point $\xi$ is invariant under some $\gamma$ and that the ray $r_1$ converges to a closed a leaf $L'$ of $\wt{\cF_1}$ invariant under $\gamma$. 

In addition, the region between $r_1$ and $r_2$ in $L$
is asymptotic to a region in $L'$. This region in $L'$ 
has ideal point $\xi \not = \alpha(L')$. 
 The rays $r_1$ and $r_2$ are thus converging to some limit rays $e_1$ and $e_2$ in $L'$.
These rays are non separated in $L'$ and they have ideal
point $\xi$ such that $\gamma \xi = \xi$ and $\gamma L' = L'$.
Proposition \ref{prop.someReeb} implies that $e_1, e_2$
are the boundaries of a $\gamma$-invariant bigon.
This completes the proof of Lemma \ref{lem.fixedpointboundary}.
\end{proof}

\begin{remark}\label{rem-nhb}
Note that in the proof we also obtain that the endpoint $\xi$ is not equal to $\alpha(L')$ for the leaf $L'$ we have found. 
\end{remark}

%%%%%%%%%%%%%%%%%%%%%%%%%%%%

\section{Consequences of the existence of Reeb surfaces}\label{s.consReeb}

The purpose of this section is to show the following proposition that gives more structure and improves Theorem A. In a certain sense, it says that if the intersected foliation $\cG$ is not an Anosov foliation, then the foliations should look very much like the ones studied in \S~\ref{ss.MatsumotoTsuboi}. 

\begin{theorem}\label{prop.consequenceReebSurface}
Assume that $\cF_1$ and $\cF_2$ have a Reeb surface and let $B$ be its lift to $L \in \wt{\cF_1}$ (a non Hausdorff bigon). Denote by $\xi, \eta$ its limit points in $\partial \HH^2$. Then, for every $L' \in \wt{\cF_i}$ there is a pair of non-separated leaves $c_1,c_2 \in \cG_{L'}$ such that $\partial^{\pm} \Phi(c_i) = \{\xi, \eta\}$ for $i=1,2$.
\end{theorem}

To prove this, we first provide some stability properties of non-Hausdorff bigons in nearby leaves. Let $B$ be a non-Hausdorff bigon in  a leaf $L \in \wt{\cF_i}$ with boundaries $c_1,c_2 \in \cG_{L}$, we denote by $\xi_B^+$ and $\xi_B^{-}$ the points in $\partial \HH^2$ given by $\partial^{\pm} \Phi(c_i) = \{\xi_B^+, \xi_B^-\}$ for $i=1,2$ and such that $\{\xi_B^+\}$
is the ideal point of the rays of $c_1, c_2$ which
are non separated from each other. See Figure~\ref{fig-idealpoints}.

\begin{figure}[ht]
\begin{center}
\includegraphics[scale=0.78]{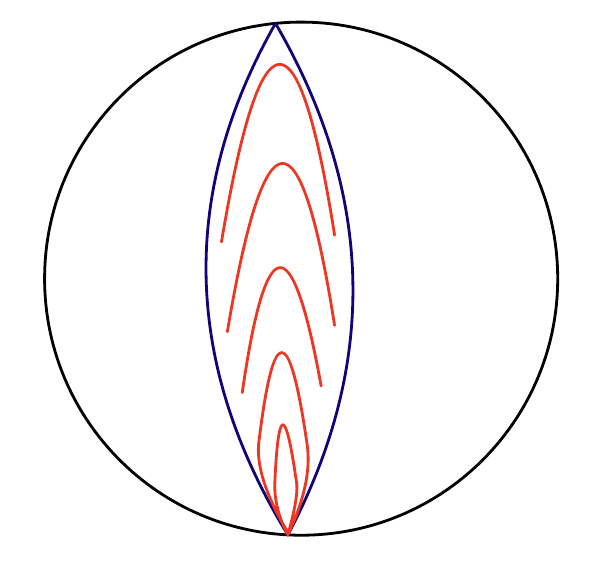}
\begin{picture}(0,0)
\put(-163,143){$c_1$}
\put(-135,212){$\xi_B^+$}
\put(-103,144){$c_2$}
\put(-130,0){$\xi_B^-$}
%\put(-129,52){$A_2^s$}
%\put(-87,77){{\small $\cG^u(p)$}}
%\put(-196,91){{\small $\cG^s(p)$}}
%\put(-113,225){$A_1^s$}
%\put(-82,220){$A_2^u$}
\end{picture}
\end{center}
\vspace{-0.5cm}
\caption{{\small The ideal points of a non-Hausdorff bigon $B$.}}\label{fig-idealpoints}
\end{figure}
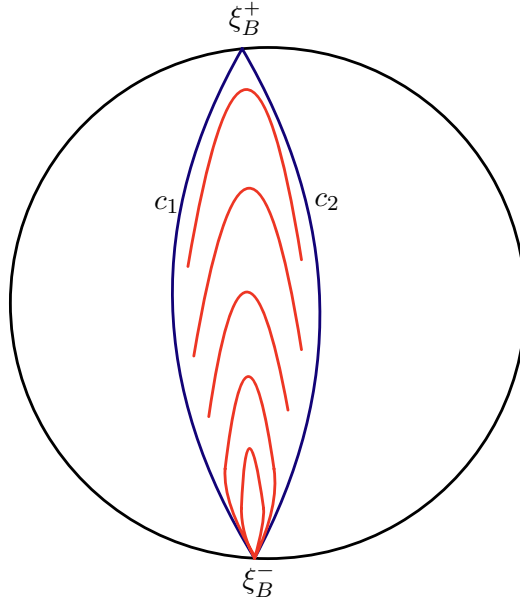

%
%In this section we will make the following abuses of notation for simplicity: 
%\begin{itemize}
%\item If $L \in \wt{\cF_i}$ and $\ell \in \cG_L$ when we say that points or sets are separated by $\ell$ in $L$ (or in different connected components of $L \setminus \ell$), we may include points that are in $L \cup S^1(L)$ where $S^1(L)$ is the circle compactification that identifies via $\Phi$ with $\partial \HH^2$. So, we will for instance say that a point $\xi \in \partial \HH^2$ is separated in $L$ from a point $ p \in L$ by a curve $\ell \in \cG_L$ if the point in $S^1(L)$ associated to $\xi$ is in a different connected component of $L \cup S^1(L) \setminus (\ell \cup \{\partial^{\pm} \Phi(\ell))$ from $p$.  
%\item When a statement requires determining leaves of both $\wt{\cF_1}$ and $\wt{\cF_2}$ we will state the result making one choice, while it is clear that by switching the roles of the foliations, the statement will also hold. 
%\item When we say an object is $\gamma$ invariant, we will many times 
%assume that $\gamma \in \pi_1(M)$ is a non trivial deck transformation
%without mentioning.
%\end{itemize}
%
%Also, some arguments that have been done in detail in previous sections (in particular use of Proposition \ref{prop-pushing}, Proposition \ref{prop-nonseplands} or the arguments involving small visual measure) will be covered more quickly, typically making reference to the corresponding place in the paper where the argument is done with more detail. 

\subsection{Persistence and extension to the closure} 
In this subsection we consider a Reeb surface lifting to a non-Hausdorff bigon $B$ in some leaf $L \in \wt{\cF_i}$ which is $\gamma$-invariant, and let $I_\gamma^{\pm}$ be the connected components of $\partial \HH^2 \setminus \{\gamma^{\pm}\}$ where $\gamma^\pm$ are the fixed points of $\gamma$. We will show that for one of the intervals $I_{\gamma}^{\pm}$, say $I_\gamma^+$ it holds that for every $L' \in \wt{\cF_i}$ so that $\alpha(L') \in I_\gamma^+$ the leaf $L'$ has a non-Hausdorff bigon $B'$ sharing the same endpoints as $B$. To do this, we will need to prove several stability properties, and along the way we will obtain some useful information that will help us to prove Theorem \ref{prop.consequenceReebSurface}. 

We first show the easy consequence of Proposition \ref{prop-pushing} that bigons persist in nearby leaves when both endpoints are marker points. 

\begin{lema}\label{lema-bigonst1}
Let $B$ be a non-Hausdorff bigon in a leaf $L \in \wt{\cF_i}$ with boundaries $c_1,c_2 \in \cG_{L}$ such that $\alpha(L) \notin \{\xi_B^+, \xi_B^-\}$. Then, there is a neighborhood $I$ of $L$ in the leaf space of $\wt{\cF_i}$ such that for every $L' \in I$ there is a non-Hausdorff bigon $B'$ in $L'$ which shares both endpoints with $B$.  
\end{lema}

\begin{proof}
This is just an application of Proposition \ref{prop-pushing} since the non-Hausdorff bigon is contained in $\wihatD_\eps(L,I)$ (see Proposition \ref{prop.setDfol}) for some $\eps$ for which Proposition \ref{prop-pushing} applies and some small neighborhood $I$ of $L$ in the leaf space of $\wt{\cF_i}$.  
\end{proof}

Now we show that such bigons extend to the boundary $\gamma$-invariant
leaves.

\begin{lema}\label{lem.ext}
Let $B$ be a non-Hausdorff bigon in a leaf $L \in \wt{\cF_i}$ with boundaries $c_1,c_2 \in \cG_{L}$ such that $\alpha(L) \notin \{\xi_B^+, \xi_B^-\}$. Assume that there is $\gamma \in \pi_1(M)$ whose fixed point at $\partial \HH^2$ are exactly $\xi_B^+$ and $\xi_B^-$. Then, for every $L' \in \wt{\cF_i}$ such that $\alpha(L')$ is in the same connected component as $\alpha(L)$ in $\partial \HH^2 \setminus \{\xi_B^+, \xi_B^-\}$ the leaf $L'$ has a non-Hausdorff bigon with the same endpoints as $B$. Moreover, this property passes to the closure of that set of leaves and thus extends to the boundary leaves $L_0, L_1$ such that $\alpha(L_0)= \xi_B^+$ and $\alpha(L_1) = \xi_B^-$. 
\end{lema}

We note that up to composing $\gamma$ with a power of the deck transformation associated to fibers we can assume that the leaves $L_0$ and $L_1$ in the conclusion are $\gamma$-invariant. 

\begin{proof}
We assume that $L \in \wt{\cF_1}$. The set of leaves for which $\alpha(L')$ belongs to the same connected component of $\alpha(L)$ in $\partial \HH^2 \setminus \{\xi_B^+, \xi_B^-\}$ is a countable union of open intervals. For any pair of such intervals, there is power of the deck transformation associated to the fiber (which induces the identity in $\partial \HH^2$) which maps one interval to the other, so it is enough to show that one such open interval $I_0$ (the one containing $L$) verifies the desired property. We can also assume that $\gamma$ fixes $I_0$. 

Let $E \in \wt{\cF_2}$ be the leaf so that $E\cap L$ contains the boundaries of the non-Hausdorff bigon and consider the set of leaves $I_1 \subset I_0$ consisting of leaves $L'$ so that $L' \cap E$ contains curves that bound a non-Hausdorff bigon joining $\{\xi_B^+,\xi_B^-\}$. Note that since $\xi_B^+$ is the point where the non-separation happens, it follows from Proposition \ref{prop-nonseplands} that $\alpha(E)= \xi_B^+$ and is thus  $\gamma$-invariant. We claim that this implies that $\gamma E = E$: note that since $\gamma \alpha(E)= \alpha(E)$ then either $E$ is fixed, or it is mapped to a leaf $E'$ obtained by acting on $E$ with a power of the deck transformation associated to the fiber and therefore $\gamma$ acts freely on the leaf space of $\wt{\cF_2}$  (recall that we are assuming that $\gamma$ is the element of $\pi_1(M) \setminus \{\mathrm{id}\}$ which also fixes $I_0$, which is an interval in the leaf space of $\wt{\cF_1}$).  To show that $\gamma E=E$ we therefore note that $\gamma^n(c_1), \gamma^n(c_2)$ are boundaries in bigons in $\gamma^n(L)$ and since $\gamma$ fixes $I_0$ which is a bounded interval in the leaf space of $\wt{\cF_1}$ the iterates $\gamma^n(L)$ converges to some leaf $L' \in \wt{\cF_1}$. By the small visual measure
property it follows that $\gamma^n(c_1)$ must remain intersecting a compact set in $\mt$ and hence hence $\gamma^n(E)$ which contains $\gamma^n(c_1)$ must also intersect a compact set in $\mt$ for all $n$. This shows that $\gamma$ cannot act freely on the leaf space of $\wt{\cF_2}$ and since $\alpha(E)$ is $\gamma$-invariant we deduce $\gamma E=E$.

Applying the argument in the previous lemma and using the fact that as long as $\alpha(L') \notin  \{\xi_B^+, \xi_B^-\}$ we can push the non Hausdorff bigon to nearby leaves we deduce that $I_1$ is open in $I_0$. 

To complete the proof we will show that $I_1$ is also closed in $I_0$. 

Consider $L_n \to L' \in I_0$ with $L_n$ in an open interval $I$ contained in $I_1$.  The leaf $L_n$ contains a non-Hausdorff bigon $B_n$ bounded by leaves $r_1^n, r_2^n$ of $\cG_{L_n}$ contained in $L_n$ but also in $E$ $-$ because pushing
preserves the $\wcF_2$ leaves they are i $-$ because pushing
preserves the $\wcF_2$ leaves they are in. Denote by $D_n$ the region in $E$ bounded by $r_1^n$ and $r_2^n$. We can assume that the sequence $D_n$ is monotonic. Note first that if the region $D_n$ decreases with $n$ (i.e. $D_{n+1} \subset D_n$) then since $D_n$ must contain a non-Hausdorff bigon of $E$ we get that in the limit the curves $r_1^n$ and $r_2^n$ converge to curves joining $\xi_B^+$ and $\xi_B^-$ as we want to show. 

We assume then that the sets $D_n$ increase with $n$ and consider the set $D_\infty= \overline{\bigcup_n D_n}$ whose closure in $E \cup S^1(L)$ is compact, connected, and cannot be the whole $E \cup S^1(E)$ (because of the small visual measure property).  Let $R_1' = \{\ell_1, \ldots, \ell_k, \ldots\}$ be the (possibly finite, but countable) collection of all limits of the curves $r_1^n$ in $\cG_{E}$. Similarly, denote by $R_2' = \{\ell_1^2, \ldots, \ell_m^2, \ldots \}$ the limits of $r_2^n$.  Since the
$r^n_1, r^n_2$ are in $L_n \cap E$, the limits are in
$E \cap L'$. Hence they belong to $\cG_E$ and $\cG_{L'}$. By Proposition \ref{prop-nonseplands} the limit points of the leaves $\ell_i^j$ can be either the ones of $r_j^n$ (that is $\xi_B^+= \alpha(E)$ or $\xi_B^-$), or $\alpha(L')$ if there is more than one limit curve.  

The set $R_1' \cup R_2'$ in $E \cup S^1(E)$ bounds the set $D_\infty$ which is $\cG_E$-saturated and has non empty interior since the region between $r_1^n$ and $r_2^n$ must contain a non-Hausdorff bigon which belongs to all the $D_n$. Since
the possible limit points of any leaf in the boundary
are only $\xi^+_B, \xi^-_B, \alpha(L')$, It follows that $R_1'$ or $R_2'$ must contain at least one leaf with an ideal point 
in $\xi_B^+$ and one with an ideal point in $\xi_B^-$ (they could be
the same leaf).

Therefore, if $L'$ does not contain a non-Hausdorff bigon joining $\xi_B^+$ and $\xi_B^-$ bounded by leaves in $R_1' \cup R_2' \subset L' \cap E$, there must be a region whose limit in $\partial \HH^2$ 
is either exactly $\{\xi_B^+, \xi_B^-, \alpha(L')\}$; 
or two regions, one with ideal points 
$\{\xi_B^+,  \alpha(L')\}$, and one with ideal points 
$\{\xi_B^-,  \alpha(L')\}$. 
We analyze the first possibility, the second one is similar.
Since $E$ is $\gamma$-invariant and since $\gamma$ does not fix $\alpha(L')$, then applying $\gamma$ or $\gamma^{-1}$ to
a curve with ideal points $\xi^-_B, \alpha(L')$ we obtain
a curve with one ideal point $\xi^-_B$ and another in the open
interval $(\xi^+_B, \alpha(L'))$ which does not contain $\xi^-_B$.
This gives a contradiction since the curves in $\cG_E$ cannot cross. 

Note that the boundary leaves bound a non Hausdorff bigon
in $L'$. Indeed these  boundary leaves have ideal points $\xi^-_B, \xi^+_B$,
and Proposition \ref{prop-pushing}  implies that the same local picture has to be seen in $L'$ as for $L_n$ for large $n$.  In other words the bigons in $L_n$ push through to $L'$.

Note that in the case of $L_n \to L_i$ with $i = 0,1$ we get the same conclusion (even simpler, since we do not have the possibility to have three limit points), only that when $\alpha(L_i)=\alpha(E)$ we cannot ensure that the region bounded by the leaves is a non-Hausdorff bigon. However, since there is no transversal intersecting both boundary leaves, it is easy to see that there must be one non-Hausdorff bigon in between. 
\end{proof}

When the endpoints of the bigon contain the non-marker point of $L$, stability is harder to establish. 
There are two cases depending on whether $\xi^+_B = \alpha(L)$ or $\xi^-_B = \alpha(L)$.

\subsection{Half interval stability: the case where $\alpha(L) = \xi_B^+$} 

The goal of this subsection is to give a proof of the following

\begin{prop}\label{lema-bigonst2}
Let $B$ be a non-Hausdorff bigon in a leaf $L \in \wt{\cF_i}$ with boundaries $c_1,c_2 \in \cG_{L}$ such that $\alpha(L) =\xi_B^+$.  Assume moreover that for some $\gamma \in \pi_1(M) \setminus \{\mathrm{id}\}$ we have that $\gamma B = B$. Then, there is a half neighborhood $I$ of $L$ in the leaf space of $\wt{\cF_i}$ (i.e. $I$ is a connected component of $J \setminus \{L\}$ where $J$ is a neighborhood of $L$ in the leaf space)   such that for every $L' \in I$ there is a non-Hausdorff bigon $B'$ in $L'$ which shares both endpoints with $B$.  
\end{prop}

Let us assume that $L \in \wt{\cF_1}$. Since $c_1$ and $c_2$ are non-separated in $\cG_L$, we know that there is $E \in \wt{\cF_2}$ such that $c_1 \cup c_2 \subset E \cap L$. 

Let $B_E$ denote the region in $E$ bounded by $c_1 \cup c_2$ which is an infinite band with bounded width, limiting on $\xi_B^\pm$ (cf. Lemma \ref{lema.proyection}). 
This set $B_E$ is not necessarily a non-Hausdorff bigon since $c_1$ and $c_2$ could be separated in $\cG_E$, but there must exist some non-Hausdorff bigon contained in $B_E$ and limiting in the same points.
By assumption $\alpha(L)=\xi_B^+$ and using Proposition \ref{prop-nonseplands} we have that $\alpha(E)=\xi_B^+ = \alpha(L)$. Up to considering the inverse, we can assume that $\gamma$ acts as an expansion on $\xi_B^+$. (Note also that since $\alpha(E)=\alpha(L)$ we have that $E = \gamma E$.) 

Fix small transversals $\tau_1$ and $\tau_2$ to $\cG_E$ in $E$ parametrized in such a way that the leaf $L_t \in \wt{\cF_1}$ through $\tau_1(t)$ also passes through $\tau_2(t)$ and that $\tau_i(0)$ belongs to $c_i$. We also assume these are parametrized so that for $t>0$ we have that $\tau_1(t)$ belongs to $B_E$.

For $t<0$ denote by $c_1^t$ and $c_2^t$ the leaves of $\cG_E$ through $\tau_1(t)$ and $\tau_2(t)$ respectively, which belong to $L_t \cap E$. Note that these cannot coincide because they are separated by $c_1$ (and $c_2$). Note also that the rays of $c_1^t$ and $c_2^t$ in the direction of $\xi_B^-$ must approximate $c_1$ and $c_2$ for $t$ small because a neighborhood of $\xi_B^-$ in $L$ is contained in the set where  Proposition \ref{prop-pushing} applies.

The leaves $c_1^t$ and $c_2^t$ cannot be connected by a transversal in $L_t$ 
(since they are both in $E$). Between them (maybe coinciding with one of them), there is a pair of leaves $e^t_1, e^t_2$ which is non-separated and such that $e_1^t, e_2^t$ both separate $c_1^t$ from $c_2^t$ (unless they coincide with them). It follows that both have $\xi_B^-$ as an ideal point. 

We call $E_t \in \wt{\cF_2}$ the leaf so that $e_1^t \cup e_2^t \subset L_t \cap E_t$. We will show that the other landing point of $e^t_1, e^t_2$
is in $\xi_B^+$ and thus there is a non-Hausdorff bigon in $L_t$ 
with ideal points $\xi^-_B, \xi^+_B$ as desired.  
Notice that a priori $E_t$ does not vary continuously with $t$,
even though we will show in the next lemma that $E_t$ is continuous
with $t$ when $t = 0$.

The other rays of $e_1^t$ and $e_2^t$ must land in $\alpha(E_t)$ (because the rays we showed limit in $\xi_B^-$ are separated by Proposition \ref{prop-pushing} and thus the other rays are non separated therefore  Proposition \ref{prop-nonseplands} applies).  In particular, the curves $e_1^t$ and $e_2^t$ bound a non-Hausdorff bigon $B_t$ in $L_t$ which limits in $\xi_t = \alpha(E_t)$ and in $\xi_B^-$. Let $L_0 = L$ and $\xi_0 = \alpha(L_0) = \alpha(E)$. 

\begin{lema}\label{lema-claim1} The function $\xi_t$ of $t$  is continuous at zero. More precisely, for every $J$ neighborhood of $\xi_B^+$ there is $\delta>0$ so that for every $t \in (\delta, 0]$ we have that $\xi_t \in J$. %Moreover, the point $\xi_t$ always belongs to the same 
%
%More precisely, there is a half neighborhood $J$ of $\xi_B^+$ in $\partial \HH^2$ such that for some small $\delta>0$ we have that if $t \in (\delta, 0]$ then $\xi_t \in J$. Moreover, for every $\hat J \subset J$ half neighborhood of $\xi_B^+$ there is some $\hat \delta < \delta$ such that if $t \in (\hat \delta, 0]$ we have that $\xi_t \in \hat J$. 
\end{lema}

\begin{proof}
To see this fix some small neighborhood $U$ of $\xi_0$ in $\partial \HH^2$.
Now choose a transversal $\eta$ to $\cG_L$ starting at a point
in $c_1$ and entering $B$. 

We choose $\eta$  small enough such that
if a leaf $\ell$ of $\cG_{L_0}$ intersects  $\eta$,
then the leaf $\ell$ is contained in a leaf $E' \in \wt{\cF_2}$ so that $\alpha(E') \in U$.  
Denote by $\ell_0$ the leaf of $\cG_L$ through the endpoint
of $\eta$ which is not in $c_1$. In particular $\ell_0$ is contained
in the interior of $B$.
Now choose $I$ a neighborhood of $\alpha(L)$ so that 
$\ell_0 \cup \eta$ is contained in 
$\wihatD_\eps(L, I)$ 
(this set is defined as in equation \eqref{eq:setDfol} where $\eps$ is chosen so that Proposition \ref{prop-pushing} holds). 
Recall that $L = L_0$. 
Let $\ell_t$ be the push through to $L_t$ of the leaf $\ell_0 \subset L$.
In other words $\ell_t$ is the component
of $\wcF_2(\ell_0) \cap L_t$ which is $\eps$ close to $\ell_0$.
In the same way $\eta$ can be pushed to a transversal $\eta_t$
to $\cG_{L_t}$ 
in $L_t$ starting in $c^t_1$ and ending in $\ell_t$.
For each $t$, $\ell_t$ is contained in between $c^t_1, c^t_2$ 
in $L_t$. In addition there is a transversal $\eta_t$
to $\cG_{L_t}$ from $c^t_1$ to $\ell_t$.
Therefore $\ell_t$ is
in between $e^t_1, e^t_2$ in $L_t$.
In particular this implies that $e^t_1$ is in between
$c^t_1$ and $\ell_t$ in $L_t$.
We have that $c^t_1 \subset E$, $e^t_1 \subset E_t$ and $\ell_t \subset 
\wcF_2(\ell_0)$, with $\alpha(\wcF_2(\ell_0))$ in $U$.
 Since $\alpha(E)$ is in $U$, it now follows that $\alpha(E_t)$
is in $U$ all well, as long as $\alpha(L_t)$ is in $I$.
Since $\xi_t = \alpha(E_t)$ this completes the proof.
\end{proof}

Note that we have also shown that there is a transversal 
to $\cG_{L_t}$ from $e^t_1$ to $c^t_1$ which $\eps$ pushes
to $L = L_0$ and similarly there is a transversal from $e^t_2$ to
$c^t_2$ which also pushes to $L$. It now makes sense to talk about monotonicity of $\xi_t$, we can indeed show: 

\begin{lema}\label{lema-claim2} The point $\xi_t$ varies in a weakly monotonic way, that is, for small $t, t' \in (-\delta,0]$ we have that if $t' < t$ then $\xi_{t'} \leq \xi_{t}$ for the orientation of $J$ making $\xi_B^+$ the maximal point. 
\end{lema}
\begin{proof}
For small $\delta'>0$, consider $\eta: (-\delta', 0] \to L$ such that $\eta$ is transverse to $\cG_L$ and such that $\eta(0) \in c_1$ and $\eta(s) \in B$ for all $s \in (-\delta',0)$. Note that as we have shown in Proposition
\ref{prop-nonseplands}, that if $\delta$ is small enough, we know that for every $t \in (-\delta,0)$ we have that $e_1^t$ belongs to the same leaf of $\wt{\cF_2}$ as $\eta(s)$ for some $s \in (-\delta',0]$. This identification will be recorded by a function $\rho: (-\delta, 0] \to (-\delta',0]$ so that $\rho(t)=s$. 

If $\delta$ is small enough, then the image of $\eta$ is contained in $\wihatD_\eps(L,I_\delta)$ where $I_\delta$ is the interval of $\partial \HH^2$ made by $\alpha(L_t)$ with $t \in (-\delta,0]$, so, applying Proposition \ref{prop-pushing} we find transversals $\eta^t: (-\delta',0] \to L_t$ intersecting the same $\wt{\cF_2}$ leaves. Denote by $\ell_s \in \cG_L$ the leaf of $\cG_L$ through the point $\eta(s)$ and let $E'_s \in \wt{\cF_2}$ the leaf that contains $\ell_s$. We note that all $\ell_s$ with $s \in (-\delta',0)$ are bubble leaves with endpoint in $\xi_B^-$. 

Note that for the function $\rho$ defined above, we have $E_t = E'_{\rho(t)}$. Consider $e^t_{t'}$ to be the leaf of $\cG_{L_t}$ containing $\eta^t(\rho(t'))$. We get that $e^t_t=e^t_1$ by definition. 

We can thus restate the claim stating that whenever $t' < t < 0$ one has that $\rho(t') \leq \rho(t)$. We assume by contradiction that this does not hold, 
that is, for a pair $t' < t < 0$ we have
$\rho(t') > \rho(t)$.
Then we get that the leaf $e^{t'}_t$ is a bubble leaf with endpoint $\xi_B^-$. On the other hand  $e^t_t = e^t_1$ is a leaf in $L_t$ with one ideal
point  $\xi_B^-$.
%We showed before Lemma \ref{lema-claim1} that 
%the other ideal point of $e^t_1$ is $\alpha(E_t)$, and Lemma \ref{lema-claim1} shows
%that this is continuous with $t$ when $t = 0$. In particular this shows
%that for small $t$, then $\alpha(E_t) \not = \xi_B^-$, so the
%ideal points of $e^t_1$ are different from each other.
%This is impossible because of the following: $e^{t'}_t is a bubble leaf
%in $E_t$ which
It follows that the leaf $E_t= E'_{\rho(t)}$ must intersect $L_{t'}$ and $L_0$ in bubble leaves while it intersects $L_t$ in at least $e_1^t \cup e_2^t$. We will show that this is impossible: 
Denote by $c$ the corresponding bubble intersection of $E_t$ with $L$.
By the remark after the proof of Lemma \ref{lema-claim1}, we know that there is 
a small transversal $\beta$ to $\cG_{L_{t'}}$ in $L_{t'}$ 
from $e^{t'}_t$ to $c^{t'}_2$ and this pushes to $L_0$, through
$\wcF_2(e^{t'}_t) = E_t$. Then $E_t$  intersects $L_t$ near
the push through of $\beta$. The same happens for
$e^t_2$. Hence the local leaf of $E_t$
passes through $e^t_2$ and also $c$.
In other words there is a small transversal $\nu_2$ to $\cG_{E_t}$
in $E_t$ intersecting $e^{t'}_t$, $e^t_2$ and $c$ in
turn. In the same way there is a small transversal 
$\nu_1$ to $\cG_{E_t}$
in $E_t$ from $e^{t'}_t$ to $e^t_1$
to $c$. Consider the closed curve in $E_t$  which is
the concatenation of $\nu_1$, a segment in $c$, $\nu_2$ and
a segment in $e^{t'}_t$.  This closed curve does not bound
a disk in $E_t$ because $e^t_1, e^t_2$ intersect this curve
in a single point. 
This would show that $E_t$ is not a plane.
This contradiction completes the proof. 
\end{proof} 

 Note that Lemma \ref{lema-bigonst1} applies directly as soon as $\xi_t \neq \alpha(L_t)$, so, we get: 
 
\begin{lema}\label{lema-claim3} For every $t \in (-\delta,0]$ either $\xi_s$ is locally constant near $t$ or $\xi_t= \alpha(L_t)$. 
\end{lema}

Finally, we show: 
\begin{lema}\label{lema-claim4}
If $\xi_t= \alpha(L_t)$ for $t \not = 0$ then, one cannot have that $L_t$ is invariant under some $\gamma_t \in \pi_1(M) \setminus \{\mathrm{id}\}$. 
\end{lema}
\begin{proof}
First notice that if this is the case then also one has that $\gamma_t \xi_t = \xi_t$ since $\gamma_t L_t = L_t$ implies that $\alpha(L_t)$ is $\gamma_t$-invariant. On the other hand, by definition, $L_t$ contains a non-Hausdorff bigon $B_t$ whose endpoints are $\xi_t$ and $\xi_B^-$. Since $t \not = 0$
and $\xi_t = \alpha(L_t)$ then $\xi_t \not = \alpha(L_0)$. But
as $\xi_B^+ = \alpha(L_0)$, then $\xi_t \neq \xi_B^+$. Also we know that $\xi_B^+$ and $\xi_B^-$ are the fixed points of $\gamma=\gamma_0$ (the deck transformation leaving $L=L_0$ invariant) then $\gamma_t \xi_B^-\neq \xi_B^-$. 
Let $B_t$ be a bigon in $L_t$ with ideal points $\xi_t$ and
$\xi^-_B$.

Now we argue as in Proposition \ref{prop.someReeb}: Due to the small visual measure property there is a uniform bound on the number of distinct non-Hausdorff bigons that share an endpoint in a given leaf of $\wt{\cF_i}$. On the other hand, applying $\gamma_t^n$ to $B_t$ we obtain infinitely many disjoint non-Hausdorff bigons in $L_t$ sharing one of the endpoints, namely $\xi_t$. This gives a contradiction and proves the lemma.  
\end{proof}

We can now give a proof of Proposition \ref{lema-bigonst2}

\begin{proof}[Proof of Proposition \ref{lema-bigonst2}] 
 We will now show that $\xi_t$ must be constant and equal to $\xi_B^+$ as desired. 
 Assume first that we have that $\xi_t = \alpha(L_t)$ in an open interval $I \subset (-\delta,0)$. Since leaves with non-trivial stabilizer are dense, it follows that for some $t$ close to $0$ we have that $L_t$ is $\gamma_t$ invariant for some $\gamma_t \in \pi_1(M) \setminus \{\mathrm{id}\}$ contradicting Lemma \ref{lema-claim4}.

Now, assume that for some $t \in (-\delta,0)$ we have that $\xi_t \neq \alpha(L_t)$. Consider $A= \{ s \in (-\delta,0) \ : \ \xi_s = \xi_t \}$. Note that Lemma \ref{lema-claim1} implies $\xi_t$ is continuous at $t=0$  we know that $A$ avoids a neighborhood of $0$.  Therefore, if $s_0 <0$ is the supremum of $A$, it follows from Lemma \ref{lema-claim3} that $s_0$ verifies that $\xi_{s_0} = \alpha(L_{s_0})$. 
%Now, since $\xi_{s}$ cannot coincide with $\alpha(L_s)$ in an interval of $s$, it follows that for $0> s' > s_0$ close to $s_0$ we have that $\xi_{s'}=\xi_{s_0}$. 
In $L_{s_0}$ we have a bigon $B_{s_0}$ with $\xi^+_{B_{s_0}} = \xi_{s_0}$.
Applying Proposition \ref{p.dichotomytech} to $L_s$ with $s$ in the
interior of the interval (so $s \not = s_0$)  we deduce that $\xi_{s_0}$ is fixed by some deck transformation $\hat\gamma_s \in \pi_1(M)$ and not acting trivially on $\partial \HH^2$. Since $\alpha(L_{s_0}) = \xi_{s_0}$ then up to changing $\hat \gamma_s \in \pi_1(M)$ by some power of the deck transformation generated by the fiber we get some $\gamma_s \in \pi_1(M) \setminus \{\mathrm{id}\}$ which fixes $L_{s_0}$ again contradicting Lemma \ref{lema-claim4}, unless $s_0 = 0$. In other words the 
interval where $\xi_s$ is constant has an endpoint in $0$. 

This completes the proof that $\xi_t$ must be constant equal to $\xi_B^+$ for $t \in (-\delta,0)$ and thus completes the proof of Proposition \ref{lema-bigonst2}. 
\end{proof}

\begin{remark}\label{rem-bigonst2} 
We have in fact showed the following: If $B$ is a $\gamma$-invariant bigon in a leaf $L \in \wt{\cF_1}$, and $E \in \wt{\cF_2}$ is the leaf which intersects $B$ in the boundary $c_1, c_2 \in \cG_L$, is such that the non-separated rays of $c_1,c_2$ in $L$ land \emph{in} $\alpha(L)$, then for every leaf $L' \in \wt{\cF_1}$ close to $L$ intersecting $E$ \emph{outside} the region $B_E$ bounded by $c_1 \cup c_2$,
the leaf $L'$  contains a bigon $B'$ whose endpoints coincide with those of $B$ and its boundaries correspond to the intersection with $E$.  
\end{remark}

\subsection{Half interval stability: the non-separated side is a marker point.}
This means that $\xi_B^+ \not = \alpha(L)$, but $\xi_B^- = \alpha(L)$. 
When the endpoint of the bigon is a marker point for $L$, we can also push to one side, but the argument is different: 

\begin{prop}\label{lema-bigonst3}
Let $B$ be a non-Hausdorff bigon in a leaf $L \in \wt{\cF_i}$ with boundaries $c_1,c_2 \in \cG_{L}$ such that $\alpha(L) =\xi_B^-$. Assume moreover that for some $\gamma \in \pi_1(M) \setminus \{\mathrm{id}\}$ we have that $\gamma B = B$.  Then, there is a half neighborhood $I$ of $L$ in the leaf space of $\wt{\cF_i}$ such that for every $L' \in I$ there is a non-Hausdorff bigon $B'$ in $L'$ which shares both endpoints with $B$.  
\end{prop}

\begin{proof}
Again we assume for concreteness that $L \in \wt{\cF_1}$. Let $E$ be the leaf in $\wt{\cF_2}$ so that $c_1 \cup c_2 \subset L \cap E$. We have that $\alpha(L) = \xi_B^-$ by assumption and that $\alpha(E)= \xi_B^+$ by Proposition \ref{prop-nonseplands}. 

Now, denote by $B_E$ the region in $E$ bounded between $c_1$ and $c_2$. Suppose that $B'$ is a non-Hausdorff bigon contained in $B_E$. We first  claim that
$B'$ must have its non-separated rays land in $\xi_B^+$ (and in particular, 
this shows that $c_1$ and $c_2$ cannot be non-separated in $E$ because they belong to $L$ and $\alpha(L) = \xi_B^-$ and Proposition \ref{prop-nonseplands} would imply that the non separated rays land there). To prove the claim, consider first some $E'$ very close to $E$ in such a way that it intersects $B$ in a curve which is very close to $c_1 \cup c_2$. If $C$ is a non-Hausdorff bigon in $B_E$ limiting in $\xi_B^-$, first let 
$L''$ be the $\wcF_1$ leaf containing $\partial C$,  and then can take a leaf $L'$ of $\wcF_1$  very close to $L''$ and in such a way that the intersection of $L'$ with $E$ is a bubble leaf with points very close to $E'$. It follows that the intersection of $L'$ and $E'$ in the region between $B$ and $B_E$ must be compact, which is a contradiction because it would give a circle leaf in $\cG_{L'}$ (and $\cG_{E'}$). 
This proves the claim.

In fact the same argument shows that any bigon in $B_E$ cannot have
a boundary leaf in $c_1$ or $c_2$. This implies that inside 
$B_E$ both $c_1, c_2$ have a neighborhood which does not intersect
a bigon. Starting from $c_1$ consider leaves $e$ of $\cG_E$ inside
$B_E$ and near $c_1$. There is a small interval of leaves
where $\gamma$ acts as a contraction or expansion in this interval. Similarly
for $c_2$. Hence there is a maximal interval $[c_1,e_1]$ where
$\gamma$ fixes the endpoints and no other leaf of $\cG_E$
in between. Similarly there is a maximal interval $[c_2,e_2]$ with
the same properties. The leaves $e_1, e_2$ are in the same leaf $L'$ of
$\wcF_1$ and it is invariant under $\gamma$. By construction
$\alpha(L') = \xi^+_B$.
Finally any leaf $L''$ of $\wcF_1$ between $L$ and $L'$ has
a bigon $B_{L''}$ which is asymptotic to $B$ in one direction
and to a bigon $B_{L'}$ in the other direction.

This provides the interval required by the proposition.
\end{proof} 

\begin{remark}\label{rem-orientbigons} 
Note that as a consequence of the proof we have the following property: Let $B$ a $\gamma$-invariant non-Hausdorff bigon in a leaf $L \in \wt{\cF_1}$ bounded by curves $c_1, c_2$ so that $c_1 \cup c_2 \subset L \cap E$ for some $E \in \wt{\cF_2}$. Let $\{\xi_B^+, \xi_B^-\}$ the endpoints of the bigon, where as above, $\xi_B^+$ denotes the non-separated point of the curves $c_1,c_2$.  Then, we have that all non-Hausdorff bigons in $E$ (note that there is at least one since there is no transversal in $E$ from $c_1$ to $c_2$) that are contained between $c_1$ and $c_2$ have its non separated point in $\xi_B^+$. 
\end{remark}

In contrast with Remark \ref{rem-bigonst2} we notice the following: 

\begin{remark}\label{rem-orientbigons2}
If $B$ is a $\gamma$-invariant bigon in a leaf $L \in \wt{\cF_1}$ and $E \in \wt{\cF_2}$ is the leaf which intersects $B$ in the boundary $c_1, c_2 \in \cG_L$ is so that the non-separated rays of $c_1,c_2$ in $L$ land in a point \emph{different from} $\alpha(L)$, then, for every leaf $L' \in \wt{\cF_1}$ close to $L$ intersecting $E$ \emph{inside} the region $B_E$ bounded by $c_1 \cup c_2$ contains a bigon $B'$ whose endpoints coincide with those of $B$ and its boundaries correspond to the intersection with $E$.  
\end{remark}
%The same holds by switching $\wt{\cF_1}$ and $\wt{\cF_2}$.  To state this precisely, we need to introduce some notation. For concreteness we will consider leaves in one foliation, but all statements hold by switching the foliations. 

\subsection{Putting all stability together} 

Let $B$ be a non-Hausdorff bigon in $L \in \wt{\cF_1}$ with boundaries $c_1,c_2 \in \cG_L$ and endpoints $\xi_B^+$ and $\xi_B^-$ (recall that $\xi_B^+$ denotes the endpoint which is the limit of the non-separated rays of $c_1$ and $c_2$). Consider a transversal $\tau: [0,\delta) \to L$ to $\cG_L$ so that $\tau(0) \in c_1$,  so that
$\tau((0,\delta))$ is contained in $B$.
Let then $E_t \in \wt{\cF_2}$ be the leaf through the point $\tau(t)$. We will denote by $\cI_B$ the closed interval in $\partial \HH^2$ which is the closure of the connected component of $\partial \HH^2 \setminus \{\xi_B^+, \xi_B^-\}$ containing $\alpha(E_t), t > 0$, that is: 

\begin{equation}\label{eq:IB}
\cI_B = \overline{ \mathrm{cc}_{\alpha(E_t)} (\partial \HH^2 \setminus \{\xi_B^+, \xi_B^-\}) }.
\end{equation} 

We note that the definition of $\cI_B$ is independent on the choice of $\tau$ and $t > 0$ (and also works if $\tau(0) \in c_2$ instead of $c_1$). 

\begin{lema}\label{lem-stableorient} 
Let $B$ be a $\gamma$-invariant non-Hausdorff bigon in a leaf $L \in \wt{\cF_1}$ with boundaries $c_1,c_2 \in \cG_{L}$ such that $\alpha(L) =\xi_B^-$ and $\gamma \in \pi_1(M) \setminus \{\mathrm{id}\}$. Then, for every $L' \in \wt{\cF_1}$ such that $\alpha(L') \in \cI_B$ (cf. equation \eqref{eq:IB}) we have that $L'$ contains a non-Hausdorff bigon $B'$ with endpoints $\xi_B^-, \xi_B^+$. In the same way if $\alpha(L) =\xi_B^+$ we also get that for every leaf $L' \in \wt{\cF_1}$ such that $\alpha(L') \in \cI_B$ we have that $L'$ contains a non-Hausdorff bigon $B'$ with endpoints $\xi_B^-, \xi_B^+$. 
\end{lema}

\begin{proof}
Note that both $\cF_1$ and $\cF_2$ are transversally orientable. We choose an orientation so that if $\mu:(-\eps,\eps) \to \mt$ is a positively oriented path transverse to {\emph{both}} 
foliations $\wcF_1, \wcF_2$, then if $L_t \in \wcF_1$ and $E_t \in \wcF_2$ are the leaves through $\mu(t)$ then $\alpha(L_t)$ and $\alpha(E_t)$ move both clockwise. To see that this is possible, recall that when lifted to $\mt$ it follows that the transverse orientation induce a direction on which the point $\alpha(L_t)$ varies as $L_t$ varies in the leaf space of $\wt{\cF_i}$. Note that since both foliations inherit the orientation of the base (because they are horizontal) the transverse orientations match (recall that from Theorem \ref{teo.matsumoto} we know that there are homeomorphisms inducing the identity in the base which maps the foliations $\cF_i$ to $\cF_{ws}$), and so we can speak of moving in the clockwise or counterclockwise direction in $\partial \mathbb{H}^2$.

Let $E$ be the $\wcF_2$ leaf containing $c_1 \cup c_2$, and 
let $B_E$ be the region in $E$ bounded by $c_1 \cup c_2$.
Let $\tau$ be a transversal to $\cG_L$ starting
at $c_1$ and entering $B$.

The following happens: moving $E_t$ (the leaf through $\tau(t)$, so $E_0 = E$) moves $\alpha(E_t)$ in one direction of $\partial \HH^2$, starting at $\xi_B^+$, by definition $\alpha(E_t)$ moves into
$\cI_B$. Without loss of generality, we can assume that $\alpha(E_t)$ moves clockwise, equivalently the leaves of $\wt{\cF_2}$ intersecting $B$ are 
\emph{above} $E_0$.
In other words $B$ is above $E_0$.
Next consider a curve $\eta:(-\delta,\delta) \to E_0$ transverse to $\cG_{E_0}$ so that $\eta(0) \in c_1$ such that the parametrization is clockwise, 
i.e. if $L_s$ is the leaf of $\wt{\cF_1}$ through the point $\eta(s)$ we have that $\alpha(L_s)$ moves clockwise as $s$ increases.

%In particular $\eta(t)$ is outside the region $B_E$ if $t > 0$
%and inside $B_E$ if $t < 0$.

Now if $\alpha(L_0) = \xi_B^-$, the previous lemma tells us that the leaves 
$L_s$ which have bigons bounded by the curves in $E \cap L_s$
are the leaves intersecting the region $B_E$ (see Remark \ref{rem-orientbigons2}). 
We saw above that the region $B$ is above $E_0$ and therefore
$B_E$ is \emph{below} $L_0$.

Therefore the leaves $L_t$ which have these bigons, satisfy that
the point $\alpha(L_t)$ is locally counterclockwise to $\xi_B^-$ and thus in the same connected component as $\alpha(E_s)$ (for small $s$) in $\partial \HH^2 \setminus \{\xi_B^+,\xi_B^-\}$.

Applying Lemma \ref{lem.ext} we get that the full closed interval between $\xi_B^+$ and $\xi_B^-$, has a bigon with ideal points $\xi^-_B, \xi^+_B$. 
If on the other hand we assume that $\alpha(E_t)$ moves counterclockwise
when $t$ increases, we get that $B$ is \emph{below} $E_0$,
and as above it will follows that $B_E$ is above $L_0$. 
One obtains the same result as above.

Finally we consider the case that $\alpha(L) = \xi^+_B$.
In this case the important fact to 
note is that in the proof of Proposition \ref{lema-bigonst2} we obtain the half neighborhood by moving in the opposite direction (see  Remark \ref{rem-bigonst2}), where we explain that the 
$L_s$ near $L_0$ which intersect $E_0$ forming a bigon in $L_s$
intersect $E_0$ {\emph{outside}} $B_E$ (as opposed to inside
$B_E$ in the previous case). 
Therefore with the conventions as in the previous case 
we have the following:
if  $\alpha(E_t) > \alpha(E_0)$ (for $t > 0$), 
we produce bigons in leaves $L_s$ above $L_0$,
so $\alpha(L_s) > \alpha(L_0)$. Therefore $\alpha(E_t)$ moves
clockwise and $\alpha(L_t)$ moves clockwise.
But $\alpha(E_t)$ moving clockwise when $t$ increases
means that $\alpha(E_t)$ moves inside $\cI_B$ for $t > 0$,
and henceforth $\alpha(L_s)$ moves inside $\cI_B$ for $t > 0$.
Thus we obtain the second statement of the  lemma.
\end{proof}

Putting together what we have shown, we can deduce: 

\begin{prop}\label{prop-halfintervals}
Let $\cF_1, \cF_2$ be two transverse minimal foliations of $M = T^1S$ and let $\cG$ be their intersection. Then if $\cG$ is not homeomorphic to the foliation given by the geodesic flow of a hyperbolic metric, it follows that
there are finitely many disjoint simple closed curves $s_1, \ldots, s_k$ 
in $S$ such that for every periodic non-Hausdorff bigon $B$ in a leaf $L \in \wt{\cF_i}$ we have that $\partial \cI_B$ (cf. equation \eqref{eq:IB}) corresponds to the endpoints of a lift of one of the curves $s_j$ to $\HH^2$. In particular: 
\begin{itemize}
\item If two periodic non-Hausdorff bigons share one endpoint, then they must share both endpoints. 
\item The endpoints of two different periodic non-Hausdorff bigons cannot be linked.
\item Up to deck transformations, there are finitely many periodic non-Hausdorff bigons. Equivalently there are finitely many Reeb surfaces of $\cG$ in $\cF_i$.
% (i.e. if $B,B'$ are periodic non-Hausdorff bigons then, either $\partial \cI_B = \partial \cI_{B'}$ or $\cI_B \cap \cI_{B'} = \emptyset$).  
\end{itemize}
\end{prop}

\begin{proof}
Assume that a leaf $L \in \wt{\cF_1}$ has a $\gamma$-periodic non-Hausdorff bigon $B$. Due to Propositions \ref{lema-bigonst2} or \ref{lema-bigonst3} we know that we can push such a bigon to nearby leaves. Then, thanks to Lemma \ref{lem.ext}  we know that for every leaf $L' \in \wt{\cF_1}$ with $\alpha(L') \in \cI_B$ 
verifies that has a bigon joining the endpoints $\xi_B^+$ and $\xi_B^-$ of $B$. 

First, assume that $\cI_B$ and $\cI_{B'}$ share an endpoint, then, since $\gamma$ and $\gamma'$ must fix those endpoints, we deduce that $\gamma'$ and $\gamma$ belong to the same cyclic group of $\pi_1(M)$. Thus, we deduce that both endpoints must coincide. 

Assume now that there are two distinct non-Hausdorff bigons $B,B'$ so that $B$ is $\gamma$-periodic and $B'$ is $\gamma'$-periodic and so that 
the ideal points of $B$, $B'$ are linked.

Again using Propositions \ref{lema-bigonst2} and \ref{lema-bigonst3},
we can find bigons with same ideal points as $B$
in all leaves $L''$ in $\wcF_1$ in the interval 
defined by $\alpha(L'')$ in the closure
of one complementary component of $\xi^+_B, \xi^-_B$.
The set of such $\alpha(L'')$ produces an interval $I_B$ of $\partial \HH$.
%(the particular complementary component depends on whether $\xi^+_B = \alpha(L)$ or not).
The same holds for $B'$, with corresponding interval $I_{B'}$. So if the ideal points link, it follows that the interiors of $I_B, I_{B'}$ intersect, and
we can find a leaf $L'$ which has both a bigon with same endpoints as $B$ and one which has both endpoints as $B'$. This is a contradiction since bigons cannot cross (since they are bounded by leaves of $\cG_{L'}$ which is a foliation). 

Finally, note that since a $\gamma$-invariant  non-Hausdorff bigon cannot cross with its translates by other deck transformations, each such non Hausdorff bigon corresponds to a simple closed curve in $S$. Similarly, distinct periodic non-Hausdorff bigons correspond to disjoint curves, and at most finitely many such disjoint curves can exist in $S$. 

To prove the final property: we may assume periodic bigons $B_1, B_2$ are 
associated with same simple closed curve $s_j$ of $S$, and are both bigons 
in say $\wcF_1$. There is a 
unique element $\gamma$ of $\pi_1(M)$ which projects to $s_j$ 
in $S$ and acts with fixed points in $\wcF_i$. This $\gamma$
has a discrete set of fixed points in the leaf space of $\wcF_1$. 
So we may assume up to fiber translates, that $B_1, B_2$ are in
the same leaf of $\wcF_1$. But then there are finitely many
such.
This proves finiteness of Reeb surfaces in $M$. 
\end{proof}

\begin{figure}[ht]
\begin{center}
\includegraphics[scale=0.98]{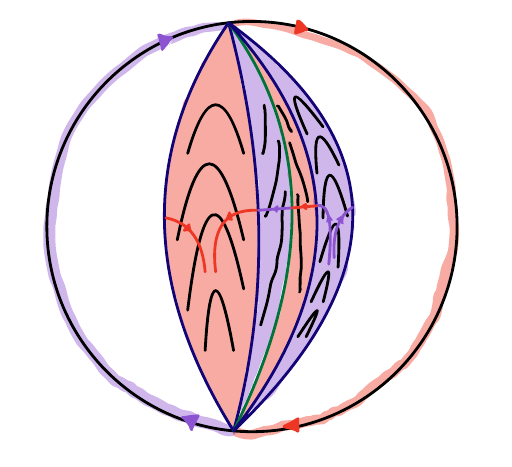}
%\begin{picture}(0,0)
%\put(-203,153){$o_1$}
%\put(-130,153){$p$}
%\put(-53,144){$o_2$}
%\put(-160,52){$A_1^u$}
%\put(-129,52){$A_2^s$}
%\put(-87,77){{\small $\cG^u(p)$}}
%\put(-196,91){{\small $\cG^s(p)$}}
%\put(-113,225){$A_1^s$}
%\put(-82,220){$A_2^u$}
%\end{picture}
\end{center}
\vspace{-0.5cm}
\caption{{\small How do the non-marker points move on the other foliation as one considers the leaf associated to a transversal of $\cG_L$.}}\label{fig-orientation}
\end{figure}

\subsection{Creating new bigons}

As one can see in the examples of \cite{MatsumotoTsuboi} it is possible that some bigons do not have continuations beyond one half interval of the leaf space. Thus, to show that there are bigons in every leaf, we need to construct new bigons (i.e. which do not come from varying continuously from the original one) in order to produce bigons in the other half interval. This requires a careful analysis of the orientation of the foliations. 

We first show that the existence of some periodic bigon forces landing points of rays to be rather restricted. We state this in somewhat more generality. 

\begin{lema} \label{lema-crossing}
Assume that there is a leaf $L \in \wt{\cF_i}$ such that it contains a leaf $e \in \cG_L$ so that $\alpha(L) \notin \{\partial^+ \Phi(e) \cup \partial^-\Phi(e)\}$.  In addition assume that $\partial^+ \Phi(e) \not = \partial^- \Phi(e)$. 
Let $J$ be a non trivial interval in the leaf space of $\wcF_i$.
Then there is a leaf $L'$ in $J$ fixed by a non trivial deck
transformation $\beta$ so that no leaf of $\cG_{L'}$ has one
ideal point fixed by $\beta$.
\end{lema}

\begin{proof}
Decrease $J$ if necessary so that $\alpha(L')$ never attains either of
$\partial^+ \Phi(e)$ or $\partial^-\Phi(e)$ for $L' \in J$. In particular $\alpha$ is injective in $J$.
Let $a_1 = \partial^+ \Phi(e), \ a_2 = \partial^- \Phi(e)$.

Let $I$ be an open interval in $\partial \mathbb{H}^2$ contained in the connected component of $\partial \mathbb{H}^2 \setminus \{a_1,a_2\}$ which does not contain $\alpha(L)$. %In addition chooose $I$ so that $\{ a_1 \cup a_2 \}$ link with$I, \alpha(J)$.

Now choose a non trivial deck transformation $\beta$ with one fixed 
point in $I$ and one in $\alpha(J)$.
Let $L'$ be a leaf of $\wcF_i$ fixed by $\beta$, and we can
assume that $L'$ is in $J$. 
This leaf satisfies the conclusion. Suppose that $\cG_{L'}$ has 
a leaf $c$ which is not a bubble leaf and has one
ideal point fixed by $\beta$.
Without loss of generality assume that $c$ has an ideal point
in $I$ which is fixed by $\beta$. Then iterating by $\beta$ or
$\beta^{-1}$ we eventually obtain a leaf  $c'$ of $\cG_{L'}$ with
one ideal point in $I$ and one ideal point in $\alpha(J)$. But $a_1, a_2$
link with $I, \alpha(J)$. This contradicts that $c'$ is a leaf of
$\cG_{L'}$  with ideal points $a_1, a_2$, which would cause crossing of
different leaves of $\cG_{L'}$.

If $c$ is a bubble leaf landing in a point fixed by $\beta$ one can iterate and produce a leaf which is not a bubble leaf and has the same characteristics. 
\end{proof}

\begin{remark} \label{rem-hiplema}
Note that if the leaf space of $\wt{\cG}$ is not Hausdorff, it follows from Theorem \ref{prop.dichotomy} that there are some non-Hausdorff bigons in some leaf, and by Proposition \ref{prop-halfintervals} that an open set of leaves contains non-Hausdorff bigons. Using minimality, we deduce that every leaf $L \in \wt{\cF_i}$ has infinitely many non-Hausdorff bigons with distinct endpoints in $\partial \mathbb{H}^2$ and therefore  we are in the hypothesis of the previous lemma for every $L \in \wt{\cF_i}$. 
\end{remark}

To produce new non-Hausdorff bigons we need to push to the other side of the bigons which is more delicate. We first show the following: 

\begin{lema}\label{l.pushchangealpha} 
Let $L \in \wt{\cF_1}$ and $E \in \wt{\cF_2}$ with $\alpha(L) \neq \alpha(E)$ so that there is $\gamma \in \pi_1(M) \setminus \{\mathrm{id} \}$ fixing $L$ and $E$. 
Let $\ell_0 \in E \cap L$ be a connected component of the intersection so that $\{\partial^+ \Phi(\ell_0), \partial^- \Phi(\ell_0) \} = \{\alpha(L), \xi_0\}$ for some $\xi_0 \neq \alpha(L)$. Let $\tau: [0,\delta) \to L$ be a transversal to $\cG_L$, with $\tau(0) \in \ell_0$, and such that
if $E_t$ denotes the leaf through $\tau(t)$ for $t>0$ we have that $\alpha(E_t)$ and $\tau(t)$ are in different connected components of $L \setminus \ell_0$. Then, for small $t>0$, if $\ell_t$ denotes the curve of $\cG_L$ through $\tau(t)$ it follows that $\{\partial^+ \Phi(\ell_t), \partial^- \Phi(\ell_t) \} \subset \{\alpha(L), \xi_0\}$ and at least one of the points is $\alpha(L)$. 
\end{lema}

\begin{proof}
Consider $I_t$ the interval of the leaf space of $\wt{\cF_2}$ made of the leaves $E_s$ with $s \in (0,t)$ and fix a small $\eps$ so that Proposition \ref{prop-pushing} applies for $\wt{\cF_2}$. 
For $t$ small we identify $I_t$ with $\bigcup \{ \alpha(E_s), s \in (0,t) \}$.
Consider then the set $\wihat D_\eps(E,I_t)$ for this foliation. It follows (see Proposition \ref{prop.setDfol}) that the set $\wihat D_t$ which is the projection of $\wihatD_\eps(E,I_t)$ to $L$ contains a complementary component 
of a neighborhood of uniform size around the geodesic joining the points $\alpha(E)=\alpha(E_0)$ and $\alpha(E_t)$.
This complementary component has $\alpha(L)$ in its closure.

In particular, one can choose $t$ small so that $\wihatD_t$ is disjoint from the image of $\tau$. In particular, one gets that one of the rays of $\ell_t$ must converge to $\alpha(L)$. We now want to understand the other ray. If $\alpha(E) \neq \xi_0$ then, the same argument shows that the other ray of $\ell_t$ converges to $\xi_0$, so we will assume in what follows that $\alpha(E)= \xi_0$. 

We assume by contradiction that the other ray of $\ell_t$ starting at $\tau(t)$,  that we call $r_t$, lands in some point $\xi \notin \{\alpha(L),\alpha(E)\}$. It cannot land in any point of the interior of $I_t$ since it would need to intersect $\ell_0$. 

We can thus apply Proposition \ref{prop-pushing} to the leaves $E_s$ for $s \in (t,0)$ to obtain a familly of rays $r_s$ in $L$ that start at $\tau(s)$ and always land in $\xi$. To see this, note that while $r_s$ intersects $\wihat D_t$ then the ray varies continuously with $s$, but if stops intersecting $\wihatD_t$ it could in principle split into more than one ray. However, due to Proposition \ref{prop-nonseplands} the landing should occur in $\alpha(E_s)$ which is impossible since it would force the ray to intersect $\ell_0$. Finally, note that when $s\to 0$ there could be splitting, but in this case we obtain that there is a curve in $E_0 \cap L$ (note that $E_0=E$) which goes from $\alpha(E)$ to $\xi$. 

In conclusion, we have shown that if the result  does not hold, then $E \cap L$ must have a curve $\ell'$ from one of the fixed points of $\gamma$ to $\xi$. Since $E$ and $L$ are $\gamma$-invariant we can iterate this intersected curve $\ell'$ to obtain a sequence of distinct intersections of $E\cap L$ landing in the same point (and since they do not admit a common transversal, they must have some distance in between).  This contradicts the small visual measure property and completes the proof of the lemma.  
\end{proof} 

There is a similar phenomenon when the splitting goes in the opposite direction: 

\begin{lema}\label{l.pushsamealpha}
Assume that $\wt{\cG}$ is not Hausdorff\footnote{Note that this is a standing hypothesis of this section, but we emphasize it here because it is crucial in the proof to be able to use Lemma \ref{lema-crossing}.}. Let $L \in \wt{\cF_1}$ and $E \in \wt{\cF_2}$ so that $\alpha(E)=\alpha(L)$  so that there is $\gamma \in \pi_1(M) \setminus \{\mathrm{id} \}$ fixing both $L$ and $E$. Let $\ell_0 \in E \cap L$ be a connected component of the intersection so that $\{\partial^+ \Phi(\ell_0), \partial^- \Phi(\ell_0) \} = \{\alpha(L), \xi_0\}$ for some $\xi_0 \neq \alpha(E)=\alpha(L)$. Let $\tau: [0,\delta) \to L$ be a transversal to $\cG_L$, with $\tau(0) \in \ell_0$. Then, for small $t>0$, if $\ell_t$ denotes the leaf of $\cG_L$ through $\tau(t)$ it follows that $\{\partial^+ \Phi(\ell_0), \partial^- \Phi(\ell_0) \} \subset \{\alpha(L), \xi_0\}$ and at least one of the points is $\xi_0$.
\end{lema}

\begin{proof}
Note first that since $E$ and $L$ are $\gamma$-invariant then, one must have that $\ell_0$ also is. Otherwise $\ell_0$ is
not $\gamma$ periodic, so we would get infinitely many distinct connected components  of $E\cap L$, all of them sharing the endpoint $\alpha(L)=\alpha(E)$. 
But this contradicts small visual measure since no two of such components can intersect a common transversal. 
It follows that $\xi_0, \alpha(L)$ are the two fixed points of $\gamma$.

Consider $I_{\delta}$ the interval of the leaf space of $\wt{\cF_2}$ which 
consists of the leaves $E_t$ through 
$\tau(t)$ with $t \in (0,\delta)$. The curve $\ell_t \subset E_t \cap L$ is the curve of $\wt{\cG}$ passing through $\tau(t)$. Since $\xi_0 \neq \alpha(E)$ we can apply Proposition \ref{prop-pushing} to deduce that one of the rays of $\ell_t$ converges to $\xi_0$. We now want to understand the other ray.

Consider for $t\in[0, \delta)$ the other ray of $\ell_t$ starting at $\tau(t)$, that we call $r_t$, and let 
$\xi_t \in \partial \HH^2$ be the landing point point of $\xi_t$. 
As in the proof of Lemma \ref{lema-bigonst2} we claim that the point $\xi_t$ is weakly monotonic as we move $t$, that it is continuous at $t=0$ and that either it is $\xi_s=\alpha(E_t)$ or $\xi_t$ is locally constant. Weak monotonicity is simpler in this case, since if $t'>t$ we have that $\ell_{t}$ separates $\ell_{t'}$ from $\ell_0$. 
Using the pushing argument of Proposition \ref{prop-pushing} we get that if $\xi_t$ is not $\alpha(E_t)$ then we must have that $\xi_t$ must be locally constant.  
Finally for continuity at $0$:  we assume that $\xi_t \neq \xi_0$ for some small $t$ (else we get the result if $\xi_t = \xi_0$ for all $t$ sufficiently small).
Assume that $\alpha(L)$ is the attracting fixed point of $\gamma$.
Then $\gamma^n(\ell_t)$ converges to $\alpha(L)$ when $n \to \infty$.
Since $\gamma^n(\ell_t) = \ell_{t_n}$ with $t_n \to 0$, 
continuity of $\xi_t$ at $t = 0$ follows.

Suppose now that $\xi_t = \alpha(E_t)$ for a non trivial
interval $J$ in $(0,t)$. 
We identify $J$ with an interval in the leaf space of $\wcF_2$,
which we can identify with an interval in $\partial \mathbb{H}^2$
as well.
As noted in Remark \ref{rem-hiplema} 
we satisfy the hypothesis of Lemma \ref{lema-crossing}. Hence there is a leaf $E_t$ in $J$ such that
$E_t$ is fixed by a non trivial deck transformation $\beta$
and no leaf of $\cG_{E_t}$ has one ideal point fixed by $\beta$.
However we know that $\xi_t = \alpha(E_t)$ and $\alpha(E_t)$
is fixed by $\beta$. This is a contradiction since $\ell_t$
has one ideal point $\xi_t$.

Exactly as in Proposition \ref{lema-bigonst2} we deduce that if $\xi_s$ is not constant and equal to $\alpha(E)=\alpha(L)$ then either it is constant and equal to $\xi_0$ or the function $\xi_s$ is locally constant with jumps in a discrete set of $(0,\delta)$ in points $a_n \to 0$ so that for each $n$ we have $\ell_{a_n}$ lands in $\xi_{n} := \alpha(E_{a_n})$. This implies that the intersection of $E_{a_n} \cap L$ contains, besides $\ell_{a_n}$ a curve joining $\xi_{n-1}$ and $\xi_n$. These curves are accumulated by the curves $\ell_t$ with $t \in (a_{n-1},a_n)$ and therefore, due to Proposition \ref{p.dichotomytech} we know that there is some $\gamma_n \in \pi_1(M) \setminus \{\mathrm{id}\}$ fixing $E_{a_n}$. This implies that $E_{a_n}$ has a non-Hausdorff bigon joining $\xi_n$ with some point in the interval from $\xi_0$ to $\xi_{n-1}$ not containing $\xi_n$. This contradicts Proposition \ref{prop-halfintervals} because it produces bigons whose endpoints are linked or have a unique common endpoint.   This completes the proof. 
\end{proof}

We can now prove the main result of this section: 

\begin{prop}\label{prop-bigonst3}
 Let $B$ be a $\gamma$-periodic non-Hausdorff bigon in a leaf $L \in \wt{\cF_i}$ for some $\gamma \in \pi_1(M) \setminus \{\mathrm{id}\}$.  
 Then, one of the following options hold: 
 \begin{enumerate}
 \item There are $\gamma$-invariant bigons $B_1$ and $B_2$ in $L$ so that $\xi_{B_1}^+ = \xi_{B_2}^+$ and such that $\cI_{B_1} \cup \cI_{B_2} = \partial \HH^2$, or, 
\item There is an even number of ordered\footnote{This means that in $L$ the bigon $B_j$ separates $B_{j-1}$ from $B_{j+1}$ for all $j =2 \ldots k-1$.} $\gamma$-invariant bigons $B_1, B_2, \ldots B_{2k}$ in $L$ so that consecutive ones have different non separated points. Moreover, if the order is chosen so that $B_1$ is the bigon such that $\xi_{B_1}^- = \alpha(L)$ then one has that $\cI_{B_1}$ is the interval which is separated from $B_1$ by $B_2$. 
 \end{enumerate}
 \end{prop}

 \begin{proof}
%Since $\wcG$ has non Hausdorff leaf space, then by Remark
%\ref{rem-hiplema} we can use Lemma \ref{lema-crossing}.
%\marg{S: What does this mean?}

To start the proof, note 
first that due to small visual measure, 
there are finitely many $\gamma$-invariant non-Hausdorff bigons in $L$ that we denote as $B_1, \ldots, B_k$ in order (i.e. $B_j$ separates $B_{j-1}$ from $B_{j+1}$ in $L$). Moreover, each such $\gamma$-periodic non-Hausdorff bigon $B_j \in L$ corresponds to the intersection of $L$ with some $E_j \in \wt{\cF_2}$ such that $\alpha(E_j) \in \{\xi,\eta\}$ which are the fixed points of $\gamma$ (and also the separated and non-separated points of $B$). Note that since $L$ is $\gamma$-invariant we know that $\alpha(L) \in \{\xi,\eta\}$. Moreover, if $\alpha(E_j) = \alpha(L)$ then it follows that $\xi_{B_j}^+ =  \alpha(L)$ while if $\alpha(E_j) \neq \alpha(L)$ then one has that $\xi_{B_j}^+ \neq \alpha(L)$.

There are also a finite number $\ell_0, \ldots, \ell_m$ of $\gamma$-invariant leaves of $\cG_L$ which correspond to intersections with leaves $F_i \in \wt{\cF_2}$ such that $\alpha(F_i) \in \{\xi,\eta\}$ (indeed, these leaves are at some minimal distance away from each other and each one must be uniformly close to the geodesic joining $\xi$ and $\eta$). We also consider these leaves ordered in the same direction. Some of them are contained in
the boundaries of the bigons $B_j$. 
We know that $m\geq 1$ since there is at least one $\gamma$-periodic non-Hausdorff bigon.

Let $D$ be the closure of the region in $L$ bounded by
$\ell_0, \ell_m$.
Since $m \geq 1$ then $D$ is non empty, with non empty interior.
Fix a transversal orientation to $\cF_2$.

Consider, for every $i \in \{0, \ldots, m-1\}$ 
a transversal $\tau_i: [0,\delta) \to L$ so that $\tau_i(0) \in \ell_i$ and such that $\tau_i(t)$ is between $\ell_i$ and $\ell_{i+1}$. Denote by $\cC_i$ the (closed) interval from $\xi$ to $\eta$ and containing $\alpha(H_t)$ where 
$H_t \in \wt{\cF_2}$ is the leaf so that $\tau_i(t) \in H_t$. Note that this 
interval is well defined and is independent  of the choice of $\tau_i$ and $t$. 

We can show:

\vskip .08in
\noindent
{\bf Claim 1:} $\cC_i \cup \cC_{i+1} = \partial \HH^2$. Moreover, the region between $\ell_i$ and $\ell_{i+1}$ is a bigon if and only if $\alpha(F_i) = \alpha(F_{i+1})$. 
\vskip .05in

Note that the interior of $\cC_i$ is exactly the set of points $\nu$ in $\partial \HH^2$ such that there is a leaf $E \in \wt{\cF_2}$ which intersects the region between $\ell_i$ and $\ell_{i+1}$ in $L$ and $\alpha(E) = \nu$. 
This is because both sets are $\gamma$ invariant and contain
small neighborhoods of the endpoints.
Hence if one considers a transversal to $\cG_L$ intersecting in
the interior the leaf $\ell_{i+1}$ (which is
a connected component of $L \cap F_{i+1}$),
one sees that $\cC_i$ and $\cC_{i+1}$ correspond to distinct intervals and thus $\cC_i \cup \cC_{i+1} = \partial \HH^2$ as desired. 

To see the last property, notice that if the region is a bigon
then $F_i = F_{i+1}$ so the $\alpha$'s coincide. Otherwise
$F_i, F_{i+1}$ are distinct, fixed by $\gamma$, and there is no
other $\gamma$ invariant leaf between $F_i$ and $F_{i+1}$.
Hence when acting on the leaf space of $\wcF_2$, it follows
that up to inverse, $\gamma$ is attracting in $F_i$ and
repelling in $F_{i+1}$. This implies that $\alpha(F_i)
\not = \alpha(F_{i+1})$.

\vskip .08in
\noindent
{\bf Claim 2:} Consider $B_j, B_{j+1}$ consecutive non-Hausdorff bigons in $L$. Then if $\xi_{B_j}^+ = \xi_{B_{j+1}}^+$ we have that $\cI_{B_{j}} \cup \cI_{B_{j+1}} = \partial \HH^2$, and if $\xi_{B_j}^+ = \xi_{B_{j+1}}^-$ then $\cI_{B_j}= \cI_{B_{j+1}}$. 
\vskip .05in

 To see this notice that in the first case one has between the boundaries of $B_j,B_{j+1}$ an even (possibly zero, if $B_j \cap B_{j+1}$ intersect in some leaf $\ell_j$) number of regions between consecutive curves $\ell_i$. Since each intersection makes a half turn in $\partial \HH^2$ without changing orientation (because this only happens when one crosses a non-Hausdorff bigon) we obtain that the leaves of $\wt{\cF_2}$ intersecting the interior of $B_j$ and $B_{j+1}$ correspond to different intervals in $\partial \HH^2$ whose boundaries are $\{\xi, \eta\}$ thus $\cI_{B_{j}} \cup \cI_{B_{j+1}} = \partial \HH^2$. The other case is similar, since one has an odd number of such regions in between.

\vskip .1in
\noindent
 {\bf Conclusion: } Claim 2 implies that the only way that all non-Hausdorff bigons with the same non-separated point verify that their associated interval is the same is that adjacent non Hausdorff bigons must have distinct non-separated points. Note that we still could have that there is a unique such non-Hausdorff bigon.  
\vskip .05in
 
 To obtain that there are an even number of such bigons, we will use Lemmas \ref{l.pushchangealpha} and \ref{l.pushsamealpha}.
Using Lemma \ref{l.pushsamealpha} we
deduce that $\ell_0$ and $\ell_m$ verify that $\alpha(F_0) = \alpha(F_m)  \neq \alpha(L)$. 
The reason is as follows: suppose first on the contrary that $\alpha(L) = \alpha(F_0)$. In this case we can apply Lemma 
\ref{l.pushsamealpha} to the curve $\ell_0$ contained in $L \cap F_0$.
It has ideal points $\xi, \eta$, which are distinct (one of them
is $\alpha(L) = \alpha(F_0)$, the other we call $\xi_0$).
Now consider a transversal $\tau$ starting in $\ell_0$ and exiting $D$
and  $\ell'$  a curve of $\cG_L$ intersecting $\tau$ outside $D$
and near $\tau(0)$.
Apply Lemma \ref{l.pushsamealpha} to $\ell_0$: it implies that 
the ideal points of $\ell'$ 
are contained in $\alpha(L), \xi_0$. But that is impossible
given the choice of $\ell_0$: no curve outside $D$ has 
ideal points contained in $\{ \xi, \eta \}$. 
This shows that $\alpha(F_0) \not = \alpha(L)$.
This also applies to $\ell_m$, showing that $\alpha(L) \neq
\alpha(F_m)$. Hence $\alpha(F_m) = \alpha(F_0) \neq \alpha(L)$.

Moreover Lemma \ref{l.pushchangealpha} implies that $\cC_0$ is the interval of $\partial \HH^2 \setminus \{\xi, \eta\}$ which is in the opposite 
connected component of $\partial \mathbb{H}^2 \setminus \{ \eta, \xi \}$ of $\ell_m$ in $L \setminus \ell_0$. Similarly, $\cC_{m-1}$ is the oposite interval, that is, the interval of $\partial \HH^2 \setminus \{\xi, \eta\}$ which is in the opposite connected component to $\ell_0$ in $L \setminus \ell_m$. In particular, the two intervals $\cC_0, \cC_{m-1}$ are different.   
 
We will now check the progress along leaves $\wcF_2$ when we cross
the region $D$ starting from $\ell_0$ all the way through $\ell_m$.
%Assume that crossing $\ell_0$ towards $\ell_m$ is in the
%positive transversal orientation to $\wcF_2$.
 Notice that crossing each periodic non-Hausdorff bigon changes the orientation in which one makes progress in the leaves of $\wt{\cF_2}$.
Lemma \ref{l.pushchangealpha} 
%and \ref{l.pushsamealpha} 
implies that in both $\ell_0$ and $\ell_m$ when one crosses
into $D$, then the 
$\alpha$'s of the $\wcF_2$ leaves move in the opposite direction
(e.g. when crossing $\ell_0$ it is in the complementary
component of $\ell_0$ which does not contain $D$).
In particular crossing throuh $\ell_0$ into $D$ and
through $\ell_m$ outside $D$, then one needs to 
cross $\wcF_2$ leaves 
with the same orientation. This implies that  there must be an even number of non-Hausdorff bigons in between. 

Finally if we choose the order so that $B_1$ is the one such that $\xi_{B_1}^-= \alpha(L)$ we have that one needs to intersect an even number of curves $\ell_i$ after $\ell_0$ to 
get to $B_1$  (possibly $\ell_0$ is the boundary of $B_1$, in which case we do not intersect any). We get that the boundary of $B_1$ is 
of the form $\{ \ell_{2i} \cup \ell_{2i+1} \}$ for some $i \geq 0$. This
implies that  $\cC_{2_i}= \cC_0$. This is because if $G$ is the leaf of $\wcF_2$ containing $\partial B_1$, then $\alpha(G) = \xi_{B_1}^+$. This completes the proof. 
\end{proof}

\subsection{Proof of Theorem \ref{prop.consequenceReebSurface} and applications}

\begin{proof}[Proof of Theorem  \ref{prop.consequenceReebSurface} ]
By assumption we know that there is at least one leaf $L$ with a non-Hausdorff bigon $B$ and some $\gamma \in \pi_1(M) \setminus \{\mathrm{id}\}$ for which we have  $\gamma B = B$ and if we denote the endpoints as $\{\xi,\eta\} \subset \partial \HH^2$ these are also $\gamma$-invariant. 
Without loss of generality we assume that $L \in \wcF_1$.
In particular, $\alpha(L) \in \{\xi,\eta\} =\{\xi_B^+, \xi_B^-\}$. Let $E
\in \wcF_2$ be such that the boundary of $B$ is contained in $E \cap L$. 

%Let us first assume that the non-Hausdorff bigon in the leaf $L$ verifies that $\alpha(L)$ is not the non-separated point of the bigon, that is, $\alpha(L) = \xi_B^-$. 

Denote by $\hat L$ a leaf in $\wt{\cF_1}$ so that $\alpha(\hat L) \in \{\xi,\eta\}$ and so that $\alpha(\hat L) \neq \alpha(L)$. The leaf $\hat L$ must also be $\gamma$-invariant. It follows from applying either Proposition \ref{lema-bigonst2} or Proposition \ref{lema-bigonst3} that $\hat L$ contains at least one non-Hausdorff bigon which up to changing $\hat L$ by a deck transformation associated to the fiber we can assume is bounded by leaves of $E \cap \hat L$. 

We now apply Proposition \ref{prop-bigonst3}. If the first condition of the proposition happens for either $L$ or $\hat L$, we get that there are $\gamma$-periodic non-Hausdorff bigons $B_i, B_j$ in $L$ (or $\hat L$) with the same non-separated point and so that $\cI_{B_i} \cup \cI_{B_j}
= \partial \mathbb{H}^2$. Applying Lemma \ref{lem-stableorient}  we deduce that for every $L' \in \wt{\cF_1}$ there must be a non-Hausdorff bigon joining $\xi$ and $\eta$ in $L'$. This completes this case.

 Thus, we can assume that both for $L$ and $\hat L$ we have that the second option of Proposition \ref{prop-bigonst3} holds. We need to show:

\vskip .08in
\noindent
{\bf Claim 1:} If we consider $B_1, \ldots, B_{2k}$ the $\gamma$-invariant non-Hausdorff bigons in $L$ with a chosen order and use the same order in $\hat L$ and get $\gamma$-invariant non-Hausdorff  bigons $\hat B_1, \ldots, \hat B_{2m}$ we have that $k=m$ and that the non-separated point of $B_1$ and $\hat B_1$ coincide. 
\vskip .05in

Using Lemma \ref{lem-stableorient} we know that there is a closed interval $\cI \subset \partial \HH^2$ (with boundary $= \{ \xi, \eta \}$) so that each bigon $B_i$ in $L$ pushes to a bigon in every $L'$ with $\alpha(L') \in \cI$. Moreover, if $E_i$ is the leaf of $\wt{\cF_2}$ so that $E_i$ intersects $L$ in the boundary of $B_i$, then we have that every leaf $L'$ as above has the corresponding bigon contained in between the intersection of $L'$ and $E_i$. Note that Remark \ref{rem-orientbigons} implies that the non-separated point of the corresponding bigons are in the same direction.

Since the intersection of $E_i$ with the leaves $L'$ from $L$ to $\hat L$ separates the region between $L$ and $\hat L$ we get that the order of the bigons cannot be reversed. Thus, we deduce that in $\hat L$ we have bigons in the same order and the same directions. It remains to show that there cannot be bigons in $\hat L$ that do not come from pushing those in $L$, but this just follows by a symmetric argument. This completes the proof of the claim.

\vskip .1in
\noindent
{\bf Conclusion: } Assume that $B_1$ verifies 
that $\xi_{B_1}^- = \alpha(L)$, 
the other case is symmetric.
In this case we deduce from Claim 1 that 
$\xi_{\hat B_{1}}^+ = \alpha(\hat L)$, and hence
$\xi_{\hat B_{2k}}^- = \alpha(\hat L)$.
By Proposition \ref{prop-bigonst3} it follows that
that $\cI_{B_1} \cup \cI_{\hat B_{2k}} = \partial \HH^2$:
first apply it to $L$ so $\cI_{B_1}$ is the interval
opposite to $B_i, i > 1$. When seen in $\hat L \cup S^1(\hat L)$
we have that $\cI_{B_1}$ is the interval in the same
complementary component of $\hat L \setminus  {\hat B_{2k}}$
that contains $\hat B_1$. Now apply it to $\hat L$:
$\cI_{\hat B_{2k}}$ is the component opposite from $\hat B_1$
when seen from $\hat B_{2k}$.
%that contains $\hat B_1$. Now apply it to $\hat L$:
%$\cI_{B_{2k}}$ is the component opposite from $\hat B_1$
%when seen from $\hat B_{2k}$.
Hence $\cI_{B_1} \cup \cI_{\hat B_{2k}} = \partial \mathbb{H}^2$.
\vskip .05in

 Finally applying Lemma \ref{lem-stableorient}  we deduce that for every $L' \in \wt{\cF_1}$ there must be a non-Hausdorff bigon joining $\xi$ and $\eta$ in $L'$. 
\end{proof}

%\begin{remark}
%Note that as part of the proof we show that the second option in Proposition \ref{prop-halfintervals} cannot happen. We also remark that if one glues appropriately along tori as in the example from \cite{MatsumotoTsuboi} one can get some alternating bigons in some leaves. 
%\end{remark}

\begin{coro}\label{coro.consReeb}
Assume that $\cF_1$ and $\cF_2$ are two transverse minimal foliations of $M= T^1S$ so that the foliation $\cG$ obtained as their intersection is not homeomorphic to the orbit foliation of the geodesic flow for a hyperbolic metric on $S$. Then, there are finitely many disjoint simple closed curves $s_1, \ldots, s_k$ in $S$ so that for every lift $\tilde s_j$ of some of these curves, if $\xi, \eta$ denote the endpoints of $\tilde s_j$ in $\partial \HH^2$,
they satisfy the following: for every leaf $L \in \wt{\cF_i}$ we have that $L$ has a non-Hausdorff bigon joining $\xi$ and $\eta$. 
\end{coro}

Note that even if we fix the homotopy classes of the curves $s_i$ which contain Reeb surfaces, it could be that the intersected foliations are not equivalent since there are possible variants to the Matsumoto-Tsuboi construction (see Remark \ref{rem-variantsMT}). In particular we can add several copies
of $\mathbb{T}^2 \times I$ associated with the same $s_i$ and they 
will not produce foliations which are equivalent. This statement gives a combinatorial way to describe the possible intersected foliations.

\begin{proof}
It follows from  Theorem \ref{prop.consequenceReebSurface} that once a leaf $L$ contains a non-Hausdorff bigon invariant by some $\gamma \in \pi_1(M) \setminus \{\mathrm{id}\}$ then every leaf contains a non-Hausdorff bigon joining the same endpoints. Note that since everything is equivariant under the action of $\pi_1(M)$ we get that every leaf must contain a non-Hausdorff bigon joining the fixed points of all elements in $\pi_1(M)$ conjugated to $\gamma$. Since non-Hausdorff bigons cannot intersect, these form a lamination in $\partial \HH^2$ (by this we mean a collection of pairs of distinct points in $\partial \HH^2$ which are pairwise not linked) and thus corresponds to a simple closed curve in $S$ (see also Proposition \ref{prop-halfintervals}). Moreover, if there is another $\gamma'$-invariant non-Hausdorff bigon, it must be also disjoint from the first one, so we get that it is associated to a disjoint simple closed curve in $S$. Since there are only finitely many 
homotopy classes of disjoint simple closed curves in a surface $S$ of genus $g \geq 2$ we obtain the result. 
\end{proof}

This section motivates the following: 

\begin{question}\label{q.MT}
Let $\cF_1, \cF_2$ be two transverse minimal foliations in a closed 3-manifold $M$. Assume that the intersected foliation $\cG = \cF_1 \cap \cF_2$ verifies that its lift $\tilde \cG$ to the universal cover does not have Hausdorff leaf space. Is it true that $M$ contains a $\pi_1$-injective torus $T$ so that every leaf $L \in \wt{\cF_i}$ contains non-Hausdorff bigons joining the endpoints of the intersections of lifts of $T$ with $L$? In particular, is it true that if $M$ is atoroidal then two transverse foliations intersect with Hausdorff leaf space in the universal cover?
\end{question}

%%%%%%%%%%%%%%%%%%%
\section{Application to partially hyperbolic diffeomorphisms}\label{s.partiallyhyperbolic}

A diffeomorphism $f: M \to M$ will be said \emph{partially hyperbolic} if it admits a $Df$-invariant splitting $TM = E^s \oplus E^c \oplus E^u$ verifying that there is some $n>0$ so that for $x \in M$ and unit vectors $v^s \in E^s(x)$, $v^c \in E^c(x)$ and $v^u \in E^u(x)$ we have that: 
\begin{equation}\label{eq:PH} 
\| Df^n v^s \| < \frac{1}{2} \min \{ 1, \|Df^n v^c \| \}   \ \ , \ \  \|Df^n v^u\| > 2 \max \{ 1, \| Df^n v^c\| \} .  
\end{equation}

The problem of classification of partially hyperbolic diffeomorphisms in 3-manifolds introduced in \cite{BonattiWilkinson, BI} has seen a lot of activity in the last few years. In particular, we point out to \cite{BFFP,BFFP2,FP2} for the analysis when the $3$-manifold is hyperbolic, and more generally for
diffeomorphisms homotopic to the identity in general $3$-manifolds. In \cite{BFP} we have introduced a class of systems called \emph{collapsed Anosov flows} which would provide a natural and useful notion of classification of such systems. In \cite{FP2} we have shown that this class contains all partially hyperbolic diffeomorphisms in hyperbolic 3-manifolds. The proof involves a careful study of pairwise transverse foliations, but also dynamics is introduced at several points in a crucial way. As mentioned in the introduction, the goal of this paper is to see to which extent we can extract dynamical information by using only the geometric properties of transverse foliations. 

We note that besides the classical examples of time one maps of geodesic flows (or more generally, \emph{discretized Anosov flows}, see \cite{BFFP, Martinchich}), unit tangent bundles admit many other classes of examples (see \cite{BGP,BGHP}) some of which have been studied in \cite{BFFP3, FP2}. 

The goal of this section is to prove Corollary B which states that every conservative partially hyperbolic diffeomorphism in $M = T^1 S$  is a collapsed Anosov flow up to finite cover and it is thus accessible (cf. \cite{FP3}).  This section will assume some familiarity with some standard results and notions of partial hyperbolicity, all of them can be found in \cite{BFP}. 

\subsection{Preliminary results and precise statement} 

Recall from \cite{BI} that when a partially hyperbolic diffeomorphism has orientable bundles whose orientation is preserved by $f$, then it preserves transverse \emph{branching foliations} $\cW^{cs}$ and $\cW^{cu}$ tangent respectively to $E^s \oplus E^c$ and $E^c \oplus E^u$. In \cite[Theorem 3.1]{HPS} some conditions are obtained which imply that these foliations do not have vertical leaves (for instance, being volume preserving is one of such conditions). 
Vertical means the leaf can be homotoped to be a union of Seifert
fibers in $M = T^1 S$.
If no vertical leaves exist, then, arguments like in \cite{BFFP, BFFP2} (see in particular \cite[Proposition 8.3]{FP2}) will allow us to prove the
following: 

\begin{teo}\label{teo.coroB}
Let $f: M \to M$ be a partially hyperbolic diffeomorphism on $M=T^1S$ where $S$ is a surface of genus $g \geq 2$. Assume moreover that $f$ preserves branching foliations which do not have vertical leaves. Then, $f$ is a collapsed Anosov flow.  
\end{teo} 

Here, being a collapsed Anosov flow means that there is a semiconjugacy between $f$ and a self-orbit equivalence of an Anosov flow of $M$ (which by a classical result of Ghys must be orbit equivalent to a geodesic flow, see \cite{Ghys}). The semiconjugacy is required to have some technical properties relating the flow and the center direction. We obtain the strongest such condition, called \emph{strong collapsed Anosov flow} in \cite{BFP}. Since we will deduce a property of the leaves that implies such condition by \cite[Theorem D]{BFP} we will refer the reader to \cite{BFP} for the actual definition of collapsed Anosov flow. 

These are the two results we shall use to be able to apply Theorem A to partially hyperbolic diffeomorphisms:

\begin{teo}[Burago-Ivanov]\label{teo.BI}
Let $f: M \to M$ be a partially hyperbolic diffeomorphism preserving an orientation of its invariant bundles. Then, $f$ admits branching foliations $\cW^{cs}$ and $\cW^{cu}$ which are approximated by foliations $\cF^{cs}$ and $\cF^{cu}$. Moreover, if the branching foliations $\cW^{cs}$ and $\cW^{cu}$ are minimal, then one can choose $\cF^{cs}$ and $\cF^{cu}$ to also be minimal. 
\end{teo}

We refer the reader to \cite[\S 3]{BFP} for a discussion on this result, in particular the final property. 
The approximation is such that there is a {\emph{collapsing}} map $h: M \to M$
which sends leaves of $\cF^{cs}$ to leaves of $\cW^{cs}$ (same for cu),
is homotopic to the identify, is $C^1$ along leaves, is $\eps$ close
to the identity, and has derivatives along the leaves which
is $C^1$ close to the identity along the leaves,  and has further 
properties. Many properties transfer between the two foliations,
in particular the topological types of leaves of both $\cW^{cs}$ and
$\cF^{cs}$ is the same.

To obtain Theorem \ref{teo.coroB} we will then apply the following criterion given by \cite[Theorem D]{BFP}: 

\begin{teo}\label{teo.enoughQG}
Let $f: M \to M$ be a partially hyperbolic diffeomorphism preserving
an orientation of its invariant bundles. If leaves of the foliation obtained by intersecting $\cF^{cs}$ and $\cF^{cu}$ are quasigeodesics in the universal cover of their respective leaves, then $f$ is a collapsed Anosov flow. 
\end{teo}

Accessibility and ergodicity then follow from \cite[Theorem A]{FP3}. We note that Corollary B follows since volume preserving partially hyperbolic diffeomorphisms verify that their branching foliations are what we call $f$-minimal, so they cannot have vertical leaves, and therefore satisfy the assumptions of Theorem \ref{teo.coroB}. We note in fact that $f$-minimality is studied in \cite{HPS} in many situations, including the case where $f$ is \emph{chain-recurrent} (something weaker than volume preserving or transitive) or belongs to certain isotopy classes. It is also shown in \cite[Proposition 4.8]{BFP} that being $f$-minimal is an open and closed condition on partially hyperbolic diffeomorphisms, so that if some $f$ is known to be isotopic to a chain recurrent one along partially hyperbolic diffeomorphisms, then it will be in the hypothesis of Theorem \ref{teo.coroB}.

\subsection{Partially hyperbolic foliations do not admit Reeb surfaces}

In this subsection we will explain why existence of a Reeb surface in the approximating foliations is not possible if the branching foliations come from a partially hyperbolic diffeomorphism. This will reduce the proof of Theorem \ref{teo.coroB} to showing that under its assumptions the branching foliations exist and are minimal. 

\begin{prop}
Let $f: M \to M$ be a partially hyperbolic diffeomorphism preserving branching foliations $\cW^{cs}$ and $\cW^{cu}$ which are approximated in the sense of Theorem \ref{teo.BI} by true transverse foliations $\cF^{cs}$ and $\cF^{cu}$. Then, the approximating foliations do not have Reeb surfaces. 
\end{prop}

\begin{proof}
The Reeb surfaces of $\cF^{cs}$ are finitely covered by 
an annulus so that the boundary components are leaves
of $\cF^{cs} \cap \cF^{cu}$ and the interior
are infinite curves spiralling to the boundary.
When collapsing to $\cW^{cs}$ and $\cW^{cu}$ these annuli 
collapse to ``branched" annuli
%Up to finite cover we get that the existence of a Reeb surface for the foliations $\cF^{cs}$ and $\cF^{cu}$ implies that there is a surface 
$S$ contained in some leaf of $\cW^{cs}$ or $\cW^{cu}$ (without loss of generality we assume $\cW^{cs}$) which is an annulus, the boundaries are leaves of the one dimensional branching foliation induced by the intersection and no transversal from one side intersects the other. Using Poincare-Bendixon's theorem we deduce that every flow transverse to the boundaries of $S$ must have a periodic orbit. Applying this to the flow generated by a unit vector field tangent to $E^s$ we deduce the existence of a closed curve tangent to $E^s$ which is a contradiction with partial hyperbolicity (because $f$ would contract the curve until you find a circle tangent to $E^s$ in an arbitrarily small ball). 
\end{proof}

\subsection{Proof of Theorem \ref{teo.coroB}.} 

Using Theorem \ref{prop.dichotomy}, Theorem \ref{teo.hsdffcase},  and Theorem \ref{teo.enoughQG} it is enough to show that the approximating foliations $\cF^{cs}$ and $\cF^{cu}$ given by Theorem \ref{teo.BI} are minimal. For this, it is enough to show that this is true for the branching foliations under the assumption that there are no vertical leaves for such foliations.

% To prove minimality, we proceed in a similar way as in \cite[\S 6.2]{BFFP2} (see also \cite[Proposition 3.15]{BFFP}) and we will assume some familiarity with those arguments. Assume there is a proper minimal $\cW^{cs}$ saturated set $\Lambda \subset M$ and consider $U$ a connected component of its complement. Note first that every such connected $U$ must contain a full unstable leaf inside. This is clear if $U$ is periodic because it is open and thus contains an unstable arc far from the boundary (iterating forward this arc accumulates in a full unstable far from the boundary).  If $U$ were not periodic, iterating forward we get infinitely many disjoint subsets of $M$ with volume uniformly bounded below, a contradiction. 
%

Theorem \ref{teo.BI} shows that the collapsing map can be chosen
to be a bijection between the sets of leaves of $\cF^{cs}$ and
$\cF_{ws}$ (say), and the collapsing map preserves their
homotopic properties. Therefore $\cF^{cs}, \cF^{cu}$ do not
have vertical leaves by assumption.

Since $\cF^{cs}$ is horizontal, the leaf space of $\wcF^{cs}$ is 
homeomorphic to the reals $\mathbb{R}$, and $\cF^{cs}$ blows
down to a minimal foliation (see \cite{Fenley2}). 
By Theorem \ref{teo.matsumoto} this foliation has only planes and annuli
leaves. Hence $\cW^{cs}$ has only planes and annuli leaves (here, the topology of the leaf is, by definition, the topology of the quotient of the leaf in the universal cover by the deck transformations that fix the given leaf). 
%\marg{R: Added this parenthesis.} 

Since every leaf of $\cW^{cs}$ is a cylinder or a plane and the foliation is $\RR$-covered we can argue exactly as in \cite[Proposition 8.3]{FP2} to get a contradiction with partial hyperbolicity (the quasigeodesic property is used in the proof of \cite[Proposition 8.3]{FP2} only to show that leaves are cylinders or planes). This completes the proof of minimality. % Since there are no vertical leaves, this torus bounds a solid torus which is periodic and this is incompatible with partial hyperbolicity as shown in \cite[Proposition 5.12]{HPS} (see \cite[Proposition 8.3]{FP2} for the same argument with more details). This completes the proof of minimality. 

\subsection{Proof of Corollary B}

To get Corollary B we first take a regular finite cover $M_1$ 
and iterate in order to have orientability of the bundles as well as their preservation. As explained in \cite[\S 7]{HPS} once we take a finite cover $M_1$, since the foliations will not have vertical leaves they need to be horizontal.
To show that there are no vertical leaves and that the branching foliations
are minimal in the cover $M_1$,  we use the volume preservation assumption as in \cite{HPS}. 
In \cite[Lemma 7.1]{HPS} and \cite[Subsection 6.4]{HPS} it is proved
that deck translations associated with the cover $M_1 \to M$ preserve the orientations of all the bundles. 
This implies that the
original bundles were horizontal and the orientability 
conditions were satisfied in $M$.

Hence the partially hyperbolic diffeomorphism $f$ 
is a collapsed Anosov flow. It follows that $f$ is accessible, and if $f$ is $C^2$ (and volume
preserving) then $f$ is ergodic (see \cite{FP3}). This proves Corollary B.

\end{document}